\documentclass{article}
\usepackage[T1]{fontenc}
\usepackage[english]{babel}
\usepackage{graphicx}
\usepackage{stmaryrd}
\usepackage{amsmath,amsfonts,amssymb}

\usepackage{ntheorem}
\usepackage{dsfont}
\usepackage{slashed}
\usepackage{hyperref}
\usepackage{tikz}
\usepackage{pgfplots}
\usetikzlibrary{patterns}
\usetikzlibrary{positioning,arrows,arrows.meta}
\usepackage{geometry}
\usepackage{enumitem}
\theoremseparator{.} 
\newtheorem{Th}{Theorem}[section]
\newtheorem{Def}[Th]{Definition}
\newtheorem{Rq}[Th]{Remark}
\newtheorem{Pro}[Th]{Proposition}
\newtheorem{Cor}[Th]{Corollary}
\newtheorem{Lem}[Th]{Lemma}
\newcommand{\R}{\mathbb{R}}
\newcommand{\E}{\mathbb{E}}
\newcommand{\Po}{\mathbb{P}}

\newcommand{\Or}{\mathbb{O}}
\newcommand{\Sp}{\mathbb{S}}
\newcommand{\Vv}{\mathbb{V}}
\newcommand{\V}{\mathbf{k}_0}
\newcommand{\K}{\widehat{\mathbb{P}}_0}

\newenvironment{proof}{\noindent\textit{Proof.~}}{\hfill$\blacksquare$\bigbreak} 
\newlength{\plarg}
\title{Sharp asymptotics for the solutions of the three-dimensional massless Vlasov-Maxwell system with small data}

\author{L\'eo Bigorgne\footnote{Laboratoire de Math\'ematiques, Univ. Paris-Sud, CNRS, Universit\'e Paris-Saclay, 91405 Orsay. E-mail adress : leo.bigorgne@u-psud.fr.}}

\begin{document}

\maketitle
    
\begin{abstract}

This paper is concerned with the asymptotic properties of the small data solutions to the massless Vlasov-Maxwell system in $3d$. We use vector field methods to derive almost optimal decay estimates in null directions for the electromagnetic field, the particle density and their derivatives. No compact support assumption in $x$ or $v$ is required on the initial data and the decay in $v$ is in particular initially optimal. Consistently with Proposition $8.1$ of \cite{dim4}, the Vlasov field is supposed to vanish initially for small velocities. In order to deal with the slow decay rate of the solutions near the light cone and to prove that the velocity support of the particle density remains bounded away from $0$, we make crucial use of the null properties of the system.
\end{abstract}

    \tableofcontents
\section{Introduction}

This article is part of a series of works concerning the asymptotic behavior of small data solutions to the Vlasov-Maxwell equations. The system is a classical model for collisionless plasma and is given, for $K$ species of particles, by\footnote{We will, throughout this article, use the Einstein summation convention so that $v^{i} \partial_{i} f = \sum_{i=1}^3 v^i \partial_{i} f$. A sum on latin letters starts from $1$ whereas a sum on greek letters starts from $0$.}
\begin{eqnarray}
\nonumber \sqrt{m_k^2+|v|^2}\partial_t f_k+v^i \partial_i f_k +e_k\left( \sqrt{m_k^2+|v|^2} {F_{0}}^j+ v^{q} {F_{q}}^j \right) \partial_{v^j} f_k & = & 0, \\ \nonumber
\nabla^{\mu} F_{\mu \nu} & = & \sum_{k=1}^K e_k \int_{v \in \R^3} \frac{v_{\nu}}{\sqrt{m_k^2+|v|^2}} f_k dv , \\ \nonumber
\nabla^{\mu} {}^* \!  F_{\mu \nu} & = & 0,
\end{eqnarray}
where
\begin{itemize}
\item $m_k \geq 0$ is the mass of the particles of the species $k$ and $e_k \neq 0$ is their charge.
\item The function $f_k(t,x,v)$ is the particle density of the species $k$, where $(t,x,v) \in \R_+ \times \R^3 \times \left( \R^3 \setminus \{0 \} \right)$ if $m_k=0$ and $(t,x,v) \in \R_+ \times \R^3 \times  \R^3 $ otherwise.
\item The $2$-form $F(t,x)$, with $(t,x) \in \R_+ \times \R^3$, is the electromagnetic field and ${}^* \!  F(t,x)$ is its Hodge dual.
\end{itemize}
In \cite{dim4}, we studied the massless Vlasov-Maxwell system in high dimensions ($n \geq 4$) and we proved that if the particle densities initially vanish for small velocities and if certain weighted $L^1$ and $L^2$ norms of the initial data are small enough, then the unique classical solution to the system exists globally in time. Moreover, as the smallness assumption only concerns $L^1$ and $L^2$ norms, no compact support assumption in $x$ or $v$ was required. We also obtained optimal pointwise decay estimates on the velocity averages of $f_k$ and their derivatives as well as improved decay estimates on the null components of the electromagnetic field and its derivatives. In the same article, we also proved that there exists smooth initial data such that the particle densities do not vanish for small velocities and for which \eqref{VM1}-\eqref{VM3} does not admit a local classical solution\footnote{Note that this result holds for dimensions $n \geq 2$.}.

Similar results for the massive Vlasov-Maxwell system in high dimensions are also obtained in \cite{dim4}. A main difference however is that $f_k$ does not have to be supported away from $v=0$. The $3d$ massive case requires a refinement of our method and will be treated in \cite{dim3}. We will also study the solutions of \eqref{VM1}-\eqref{VM3} in the exterior of a light cone. The strong decay satisfied by $f_k$ in such a region will allow us to lower the initial decay hypothesis on the electromagnetic field and to obtain asymptotics on the solutions in a simpler way than for the whole spacetime. This will be done in \cite{ext}.

In this paper, we study the asymptotic properties of the small data solutions to the three-dimensional massless Vlasov-Maxwell, so that $m_k=0$. We start with optimal decay in $v$ on the particle densities in the sense that we merely suppose $f_k(0,x,.)$ to be integrable in $v$, which is a necessary condition for the source term of the Maxwell equations to be well defined. In massive Vlasov systems, powers of $|v|$ are often lost in order to gain time decay or to exploit null properties\footnote{See for instance \cite{dim4} for the Vlasov-Maxwell system, where we used the inequality $1 \leq 4v^0 v^{\underline{L}}$ in order to take advantage of decay in $t-r$.}. Our assumptions will force us to better understand the null structure of the equations. In fact, one of the goals of this article is to describe in full details the null structure of the system, which appears to be fundamental for proving integrability and controling the velocity support of the particle density.

In view of their physical meaning, the functions $f_k$ are usually supposed to be non negative. However, as their signs play no role in this paper and since we will consider neutral plasmas, we suppose for simplicity that $K=1$ and we do not restrict the values of $f_1$ to $\R_+$. We also normalize the charge $e_1$ to $1$ and we denote $f_1$ by $f$. The system can then be rewritten as
\begin{eqnarray}\label{VM1}
|v|\partial_t f+v^i \partial_i f +\left(|v|{F_{0}}^j+ v^{q} {F_{q}}^j \right) \partial_{v^j} f & = & 0, \\ \label{VM2}
\nabla^{\mu} F_{\mu \nu} & = & J(f)_{\nu} \hspace{2mm} := \hspace{2mm} \int_{v \in \R^3} \frac{v_{\nu}}{|v|} f dv , \\ \label{VM3}
\nabla^{\mu} {}^* \!  F_{\mu \nu} & = & 0.
\end{eqnarray}
Note that we can recover the more common form of the Vlasov-Maxwell system using the relations
$$E^i=F_{0i} \hspace{8mm} \text{and} \hspace{8mm} B^i=-{}^* \!  F_{0i},$$
so that the equations \eqref{VM1}-\eqref{VM3} can be rewritten as 
\begin{flalign*}
 & \hspace{3cm} |v| \partial_t f+v^i \partial_i f + (|v|E+v \times B) \cdot \nabla_v f = 0, & \\
& \hspace{3cm} \nabla \cdot E =  \int_{v \in \R^3}fdv, \hspace{1.9cm} \partial_t E^j = (\nabla \times B)^j -\int_{v \in \R^3} \frac{v^j}{|v|}fdv, & \\
& \hspace{3cm} \nabla \cdot B = 0, \hspace{3,2cm} \partial_t B = - \nabla \times E. &
\end{flalign*} 
We choose to work with a neutral plasma to simplify the proof but the case of a non zero total charge will be covered in \cite{dim3} and \cite{ext}.

\subsection{Previous results on small data solutions for the massive Vlasov-Maxwell system}

Global existence for small data in dimension $3$ was first established by Glassey-Strauss in \cite{GSt} under a compact support assumption (in space and in velocity). In \cite{GSc}, a similar result is obtained for the nearly neutral case. The compact support assumption in $v$ is removed in \cite{Sc} but the data still have to be compactly supported in space. Note that none of these results contain estimates on $\partial_{\mu_1}...\partial_{\mu_k} \int_v f dv$ and the optimal decay rate on $\int_v f dv$ is not obtained by the method of \cite{Sc}. They all proved decay estimates on the electromagnetic field up to first order derivatives. 

In \cite{dim4}, we used vector field methods, developed in \cite{CK} for the electromagnetic field and \cite{FJS} for the Vlasov field, in order to remove all compact support assumptions for the dimensions $n \geq 4$. We then derived (almost) optimal decay on the solutions of the system and their derivatives and we described precisely the behavior of the null components of $F$.

Recently, Wang proved in \cite{Wang} a similar result for the $3d$ case. Using both vector field method and Fourier analysis, he replaced the compact support assumption by strong polynomial decay hypotheses in $(x,v)$ on $f$ and obtained optimal pointwise decay estimates on $\int_v f dv$ and its derivatives. 

\subsection{Previous works on Vlasov systems using vector field methods}

Properties of small data solutions of other Vlasov systems were obtained recently using vector field methods. First on the Vlasov-Nordstr\"om system, in \cite{FJS} and \cite{FJS3}, and the Vlasov-Poisson system (see \cite{Poisson}). Vector field methods led to a proof of the stability of the Minkowski spacetime for the Einstein-Vlasov system, obtained independently by \cite{FJS2} and \cite{Lindblad}.

Note that vector field methods can also be used to derive integrated decay for solutions to the the massless Vlasov equation on curved background such as slowly rotating Kerr spacetime (see \cite{ABJ}).
\subsection{Statement of the main result}
The following theorem is the main result of this paper. For the notations not yet defined, see Section \ref{sec2}.
\begin{Th}\label{theorem}
Let $N \geq 10$, $\epsilon >0$ and $(f^0,F^0)$ an initial data set for the Vlasov-Maxwell equations \eqref{VM1}-\eqref{VM3} satisfying the smallness assumption\footnote{We could avoid any hypotheses on the derivatives of order $N+1$ and $N+2$ of $F^0$ (see Remark \ref{rqFhighderiv} for more details).}
$$  \sum_{ |\beta|+|\kappa| \leq N+3} \int_{x \in \R^3} \int_{v \in \R^3} (1+|x|)^{|\beta|+2}(1+|v|)^{|\kappa|} \left| \partial_x^{\beta} \partial_v^{\kappa} f^0 \right| dv dx +\sum_{  |\gamma| \leq N+2} \int_{x \in \R^3} (1+|x|)^{2|\gamma|+2} \left| \nabla_{\partial^{\gamma}_x} F^0 \right|^2 dx \leq \epsilon,$$
the neutral hypothesis
\begin{equation}\label{eq:neutral}
\int_{x \in \R^3} \int_{v \in \R^3} f^0 dv dx =0
\end{equation}
and the support assumption
$$\forall \hspace{1mm} 0 < |v| \leq 3, \hspace{12mm} f^0(.,v)=0.$$
There exists $C>0$ and $\epsilon_0>0$ such that if $0 \leq \epsilon \leq \epsilon_0$, then the unique classical solution $(f,F)$ of the system which satisfies $f(t=0)=f^0$ and $F(t=0)=F^0$ is a global solution and verifies the following estimates.
\begin{itemize}
\item Energy bound for the electromagnetic field $F$: $\forall$ $t \in \R_+$,
$$\sum_{\begin{subarray}{}   Z^{\gamma} \in \mathbb{K}^{|\gamma|} \\ \hspace{1mm} |\gamma| \leq N \end{subarray}}  \int_{ \R^3}  \tau_+^2 \left(  \left| \alpha \left( \mathcal{L}_{ Z^{\gamma}}(F) \right) \right|^2  +  \left| \rho \left( \mathcal{L}_{ Z^{\gamma}}(F) \right) \right|^2  +  \left|\sigma \left( \mathcal{L}_{ Z^{\gamma}}(F) \right) \right|^2 \right)+\tau_-^2 \left|\underline{\alpha} \left( \mathcal{L}_{ Z^{\gamma}}(F) \right) \right|^2  dx \leq C\epsilon \log^4(3+t).$$
\item Sharp pointwise decay estimates for the null components of $\mathcal{L}_{Z^{\gamma}}(F)$: $\forall$ $|\gamma| \leq N-2$, $(t,x) \in \R_+ \times \R^3$,
\begin{eqnarray}
\nonumber \left| \rho \left( \mathcal{L}_{Z^{\gamma}}(F) \right) \right|(t,x) & \lesssim & \sqrt{\epsilon} \frac{\log^2 (3+t)}{\tau_+^{2} \tau_-^{\frac{1}{2}}}, \hspace{20mm} \left| \alpha \left( \mathcal{L}_{Z^{\gamma}}(F) \right) \right|(t,x) \hspace{2mm} \lesssim \hspace{2mm} \sqrt{\epsilon} \frac{\log^2 (3+t)}{\tau_+^{\frac{5}{2}} }, \\ \nonumber
\left| \sigma \left( \mathcal{L}_{Z^{\gamma}}(F) \right) \right|(t,x) & \lesssim & \sqrt{\epsilon} \frac{\log^2 (3+t)}{\tau_+^{2} \tau_-^{\frac{1}{2}}} , \hspace{20mm}  \left| \underline{\alpha} \left( \mathcal{L}_{Z^{\gamma}}(F) \right) \right|(t,x) \hspace{2mm} \lesssim \hspace{2mm} \sqrt{\epsilon} \frac{\log^{\frac{5}{2}} (1+\tau_-)}{\tau_+ \tau_-^{\frac{3}{2}}} . 
\end{eqnarray}
\item Energy bound for the particle density: $\forall$ $t \in \R_+$,
$$ \sum_{\begin{subarray}{}  \hspace{0.5mm} \widehat{Z}^{\beta} \in \K^{|\beta|} \\ \hspace{1mm} |\beta| \leq N \end{subarray}} \sum_{z \in \V} \int_{x \in \R^3} \int_{v \in \R^3} \left|z \widehat{Z}^{\beta} f \right|(t,x,v) dv dx \leq C\epsilon \log (3+t).$$
\item Vanishing property for small velocities:
$$ \hspace{-5mm} \forall \hspace{0.5mm} (t,x,v) \in \R_+ \times \R^3 \times \left( \R^3 \setminus \{ 0 \} \right), \hspace{1.5cm} |v| \leq 1 \hspace{1mm} \Rightarrow \hspace{1mm} f(t,x,v) =0.$$
\item Sharp pointwise decay estimates for the velocity averages of $\widehat{Z}^{\beta} f$: $\forall$ $|\beta| \leq N-5$, $z \in \V$,
$$\forall \hspace{0.5mm} (t,x) \in \R_+ \times \R^3, \hspace{1.5cm} \int_{ v \in \R^3} \left| z^2 \widehat{Z}^{\beta} f \right| dv \hspace{2mm} \lesssim \hspace{2mm} \frac{\epsilon}{\tau_+^2 \tau_-} .$$
\end{itemize}
\end{Th}
\begin{Rq}
One can prove a similar result if $f^0$ vanishes for the velocties $v$ such that $|v| \leq R$, with $R >0$ ($\epsilon_0$ would then also depend on $R$).
\end{Rq}
\begin{Rq}
We say that $(f^0,F^0)$ is an initial data set for the Vlasov-Maxwell system if the function $f^0 : \R^3_x \times \left( \R^3_v \setminus \{ 0 \} \right) \rightarrow \R$ and $F^0$ are both sufficiently regular and satisfy the constraint equations
$$\nabla^i \left( F^0 \right)_{i0} =-  \int_{v \in \R^3} f^0 dv \hspace{10mm} \text{and} \hspace{10mm} \nabla^i \left( {}^* \! F^0 \right)_{i0} =0.$$ 
\end{Rq}
\begin{Rq}
The neutral hypothesis \eqref{eq:neutral} is a necessary condition for $\int_{\R^3} (1+r)^2|F|^2 dx$ to be finite. This means that, for a sufficiently regular solution to the Vlasov-Maxwell system $(f,F)$, the total electromagnetic charge 
$$ Q(t) := \lim_{r \rightarrow + \infty} \int_{\mathbb{S}_{t,r}} \frac{x^i}{r} F_{0i} d \mathbb{S}_{t,r} =  \int_{x \in \R^3} \int_{v \in \R^3} f dv dx ,$$
which is a conserved quantity in $t$, vanishes. More precisely, if $Q(0) \neq 0$, then
$$\int_{\R^3} r \left| \rho \left( F^0 \right) \right|^2 dx = +\infty, \hspace{8mm} \text{where} \hspace{4mm} \rho \left( F^0 \right) := \frac{x^i}{r} \left( F^0 \right)_{i0}.$$
We prove in Appendix \ref{AppendixC} that the derivatives of $F$ are automatically chargeless, whether or not $Q$ vanishes.
\end{Rq}
\subsection{Strategy of the proof and main difficulties}

The proof of Theorem \ref{theorem} is based on energy and vector field methods and essentially relies on bounding sufficiently well the spacetime integrals of the commuted equations. The solutions of the massless Vlasov equation enjoy improved decay estimates in the null directions. More precisely, one can already see that with the following estimate (see Lemma \ref{weights1} and Proposition \ref{KS1}), for $g$ a solution to the free transport equation $|v| \partial_t g+v^i \partial_{i} g=0 $,
\begin{equation}\label{extradecayintro}
\hspace{-1mm} \forall \hspace{0.5mm} (t,x) \in \R_+ \times \R^3, \hspace{9mm} \int_{v \in \R^3} \left| \frac{v^L}{|v|} \right|^p \left| \frac{v^A}{|v|} \right|^k \left| \frac{v^{\underline{L}}}{|v|} \right|^q |g|(t,x,v) dv \lesssim \sum_{|\beta| \leq 3} \frac{\| (1+r)^{|\beta|+p+k+q} \partial^{\beta}_x g \|_{L^1_{x,v}}(t=0)}{(1+t+r)^{2+k+q}(1+|t-r|)^{1+p}}.
\end{equation}
This strong decay is a key element of our proof. Without it, we would have to consider modifications of the commutation vector fields of the free transport operator as in \cite{FJS3}, \cite{Poisson}, \cite{FJS2} and \cite{Lindblad} for, respectively, the Vlasov-Nordstr\"om, the Vlasov-Poisson and the Einstein-Vlasov systems. As the particles are massless, the characteristics of the transport equation and those of the Maxwell equations have the same velocity\footnote{Note that this is not the case for particles of mass $m >0$ since the free transport operator is then $\sqrt{m^2+|v|^2} \partial_t +v^i \partial_i$.}. The consequence is that, in a product such as $\mathcal{L}_{Z^{\gamma}}(F).\widehat{Z}^{\beta}f$, we cannot transform a $|t-r|$ decay in a $t+r$ one as it is done, in view of support consideration, for the massive case with compactly supported initial data. We are then led to carefully study the null structure of the equations, and in particular of the non linearities such as
\begin{equation}\label{intrononlin}
 v^{\mu} {\mathcal{L}_{Z^{\gamma}}(F)_{\mu}}^{ i} \partial_{v^i} \widehat{Z}^{\beta} f, 
\end{equation}
with $Z$ a Killing vector field and $\widehat{Z}$ its complete lift\footnote{The expression of the complete lift of a vector field of the Minkowski space is presented in Definition \ref{Deflift}.}. The problem is that, for $g$ a solution to $|v|\partial_t+v^{i} \partial_{i} g=0$, $\partial_v g$ essentially behaves as $(1+t+r) \partial_{t,x} g$ and the electromagnetic field, as a solution of a wave equation, only decays with a rate of $(1+t+r)^{-1}$ in the $t+r$ direction. However, from \cite{CK}, we know that certain null components of the Maxwell field are expected to behave better than others. The same is true for the null components of the velocity vector $v$ as it is suggested by \eqref{extradecayintro}. Moreover, we also know from \cite{dim4} that $v^{\underline{L}}$ allows us to take advantage of the $t-r$ decay as it permits to estimate spacetime integrals by using a null foliation. Finally, the radial component of $(0,\partial_{v^1} \widehat{Z}^{\beta} f,\partial_{v^2} \widehat{Z}^{\beta} f,\partial_{v^3} \widehat{Z}^{\beta} f)$ costs a power of $t-r$ instead of $t+r$. The null structure of \eqref{intrononlin} is then studied in Lemma \ref{calculF} and we can observe that each term contains at least one good component. 

Another problem, specific to massless particles, arises from small velocities. We already observed in Section $8$ of \cite{dim4} that the velocity part $V$ of the characteristics of
\begin{equation}\label{characintro}
 \partial_t +\frac{v^i}{|v|} \partial_i f+\left( F_{0i} +\frac{v^j}{|v|} F_{ji} \right) \partial_{v^i} f=0
 \end{equation}
can reach $0$ in finite time. The consequence is that if $f$ does not initially vanish for small velocities, the Vlasov-Maxwell system could not admit a local classical solution. This issue is reflected in the energy estimates through, schematically,
$$ \left\| \widehat{Z} f \right\|_{L^1_{x,v}} \hspace{-0.5mm} (t) \hspace{2mm} \leq \hspace{2mm}  2\left\| \widehat{Z} f \right\|_{L^1_{x,v}} \hspace{-0.5mm} (0)+\int_{0}^t \int_{x \in \R^3} \int_{v \in \R^3} |\psi(t,x,v)| \frac{|\widehat{Z} f|}{|v|} dv dx ds, $$
where $\psi$ is a homogeneous function of degree $0$ in $v$. One cannot hope to close such an estimate using say the Gr\"onwall inequality due to the factor of $\frac{1}{|v|}$ appearing in the error term on the right hand side. In \cite{dim4}, we take advantage of the strong decay rate of the electromagnetic field, given by the high dimensions, to prove that the velocity support of $f$ remains bounded away from $0$ if initially true. The slow decay of $F$ in dimension $3$ forces us to exploit the null structure of the equations satisfied by the characteristics of \eqref{characintro} in order to recover this result. The strong decay rate satisfied by the radial component of the electric field $\rho (F)$ plays a fundamental role here. We point out that this difficulty is not present in the Einstein-Vlasov system as the Vlasov equation can be written, for a metric $g$ and defined in terms of the cotangent variables, as
$$v_{\mu} g^{\mu \nu} \partial_{x^{\nu}} f-\frac{1}{2} v_{\mu} v_{\nu} \partial_i g^{\mu \nu} \partial_{v_i} f=0.$$
One can observe that the homogeneity in $v$ of the non linearity of the Vlasov equation is the same than the one of $|v| \partial_t+v^i \partial_i$, so that the velocity part of the characteristics cannot reach $0$ in finite time time. Local existence, for the massless Einstein-Vlasov system and for sufficiently regular initial data, has been proven by Svedberg in \cite{Svedberg}.

\subsection{Structure of the paper}

Section \ref{sec2} presents the notations used in this article, basic results on the electromagnetic field and its null decomposition. The commutation vector fields are introduced in Subsection \ref{subsecvector} and the source terms of the commuted equations are descibed in Subsection \ref{subseccomu}. Subsection \ref{sectionweights} contains fundamental properties on the null components of the velocity vector. In Section \ref{sec4}, we introduce the norms used to study the Vlasov-Maxwell system and we present energy estimates in order to control them. We then exploit these energy norms to obtain pointwise decay estimates on both fields through Klainerman-Sobolev type inequalities. Lemma \ref{calculF}, proved in Section \ref{sec5}, is of fundamental importance in this work since it depicts the null structure of the non linearities of the transport equations. In section \ref{sec6}, we set up the bootstrap assumptions, discuss their immediate consequences and describe the main steps of the proof of Theorem \ref{theorem}. Sections \ref{sec7} to \ref{sec9} concern respectively the improvement of the bounds on the distribution function, the proof of $L^2$ estimates for the velocity averages of its higher order derivatives and the improvement of the estimates on the electromagnetic field energies. In Appendix \ref{appendixA}, we prove that the Vlasov field vanishes for small velocities. In Appendix \ref{AppendixB} we expose how to bound the energy norms of $f$ and $F$ in terms of weighted $L^1$ and $L^2$ norms of the initial data. We prove in Appendix \ref{AppendixC} that the derivatives of $F$, for $(f,F)$ a sufficiently regular solution to the Vlasov-Maxwell system, are automatically chargeless. Finally, Appendix \ref{appendixD} contains the proof of certain results concerning the null decomposition of the electromagnetic field.

\subsection{Acknowledgements}

This article forms part of my Ph.D. thesis and I am grateful to my advisor Jacques Smulevici for his support and fruitful discussions. Part of this work was funded by the European Research Council under the European Union's Horizon 2020 research and innovation program (project GEOWAKI, grant agreement 714408).

\section{Notations and preliminaries}\label{sec2}

\subsection{Basic notations}

In this paper we work on the $3+1$ dimensional Minkowski spacetime $(\R^{3+1},\eta)$. We will use two sets of coordinates, the Cartesian $(t,x^1,x^2,x^3)$, in which $\eta=diag(-1,1,1,1)$, and null coordinates $(\underline{u},u,\omega_1,\omega_2)$, where
$$\underline{u}=t+r, \hspace{10mm} u=t-r$$
and $(\omega_1,\omega_2)$ are spherical variables, which are spherical coordinates on the spheres $(t,r)=constant$. These coordinates are defined globally on $\R^{3+1}$ apart from the usual degeneration of spherical coordinates and at $r=0$. We will also use the following classical weights,
$$\tau_+:= \sqrt{1+\underline{u}^2} \hspace{8mm} \text{and} \hspace{8mm} \tau_-:= \sqrt{1+u^2}.$$
We denote by $(e_1,e_2)$ an orthonormal basis on the spheres and by $\slashed{\nabla}$ (respectively $\slashed{div}$) the intrinsic covariant differentiation (respectively divergence operator) on the spheres $(t,r)=constant$. Capital Latin indices (such as $A$ or $B$) will always correspond to spherical variables. The null derivatives are defined by
$$L=\partial_t+\partial_r \hspace{5mm} \text{and} \hspace{5mm} \underline{L}=\partial_t-\partial_r, \hspace{8mm} \text{so that} \hspace{4mm} L(\underline{u})=2, \hspace{3mm} L(u)=0, \hspace{3mm} \underline{L}( \underline{u})=0 \hspace{3mm} \text{and} \hspace{3mm} \underline{L}(u)=2.$$
The velocity vector $(v^{\mu})_{0 \leq \mu \leq 3}$ is parametrized by $(v^i)_{1 \leq i \leq 3}$ and $v^0=|v|$ since we study massless particles. We introduce $T$, the operator defined, for all sufficiently regular function $f : [0,T[ \times \R^3_x \times \left( \R^3_v \setminus \{ 0 \} \right)$, by
$$T : f \mapsto v^{\mu} \partial_{\mu} f.$$
We will use the notation $\nabla_v g := (0,\partial_{v^1}g, \partial_{v^2}g,\partial_{v^3}g)$ so that \eqref{VM1} can be rewritten $$T_F(f) := v^{\mu} \partial_{\mu} f +F \left( v, \nabla_v f \right) =0.$$

\begin{Rq}
As we study massless particles, the functions considered in this paper will not be defined for $v=0$. However, for simplicity and since $\{v=0\}$ has Lebesgue measure $0$, we will consider integrals over $\R^3_v$. Moreover, the distribution function $f$ will be supported away from $v=0$ during the proof of Theorem \ref{theorem}.
\end{Rq}
We will use the notation $D_1 \lesssim D_2$ for an inequality such as $ D_1 \leq C D_2$, where $C>0$ is a positive constant independent of the solutions but which could depend on $N \in \mathbb{N}$, the maximal order of commutation. Finally we will raise and lower indices using the Minkowski metric $\eta$. For instance, $v_{\mu} = v^{\nu} \eta_{\nu \mu}$ so that $v_0=-v^0$ and $v_i=v^i$ for all $1 \leq i \leq 3$.

\subsection{The problem of the small velocities}\label{subsecsmallvelo}

For technical reasons, we will use all along this paper a fixed cutoff function $\chi$ such that $ \chi =1$ on $[1,+\infty[$ and $\chi =0$ on $]-\infty, \frac{1}{2}]$. We introduce the operator
\begin{equation}\label{TFtheta}
T_F^{\chi} : g \mapsto  v^{\mu} \partial_{\mu} g + \chi(|v|) F \left( v, \nabla g \right) .
\end{equation}
As mentioned earlier, we proved in Section $8$ of \cite{dim4} that because of the small velocities, there exist initial data sets for which the Vlasov-Maxwell system does not admit a local classical solution. The main idea of the proof consists in studying characteristics such that their velocity part reaches $0$ in finite time. This is why we suppose in Theorem \ref{theorem} that the Vlasov field vanishes initially for small velocities and one step of the proof will be to verify that this property remains true for all $t \in \R_+$. To circumvent difficulties related to characteristics reaching $v=0$, we will rather first define $(f,F)$ as the solution to $\eqref{VM2}-\eqref{VM3}$ and $T_F^{\chi}(f)=0$. Notice that none of the characteristics of the operator $T^{\chi}_F$ reaches $v=0$. Indeed, if $(X,V)$ is one of them, we have
$$\frac{d V^j}{dt}(s)= \chi \left( |V|(s) \right) \frac{V^{\mu}(s)}{|V|(s)} {F_{\mu}}^j(s,X(s)) .$$
Consequently, if $|V(s)| < \frac{1}{2}$, then $V(t)=V(s)$ for all $t \geq s$. The goal will then to prove that if $f(0,.,.)$ vanishes for all $|v| \leq 3$, so does $f(t,.,.)$ for all $|v| \leq 1$, implying that $T_F(f)=0$ and that $(f,F)$ is a solution to the Vlasov-Maxwell system \eqref{VM1}-\eqref{VM3}. 
\subsection{Basic tools for the study of the electromagnetic field}\label{basicelec}

As we describe the electromagnetic field in geometric form, it will be represented throughout this article by a $2$-form. Let $F$ be a $2$-form defined on $[0,T[ \times \R^3_x$. Its null decomposition $(\alpha(F), \underline{\alpha}(F), \rho(F), \sigma (F))$, introduced by \cite{CK}, is defined by
$$\alpha_A(F) = F_{AL}, \hspace{8mm} \underline{\alpha}_A(F)= F_{A \underline{L}}, \hspace{8mm} \rho(F)= \frac{1}{2} F_{L \underline{L} } \hspace{8mm} \text{and} \hspace{8mm} \sigma(F) =F_{12}.$$
The Hodge dual ${}^* \! F$ of $F$ is the $2$-form given by
$${}^* \! F_{\mu \nu} = \frac{1}{2} F^{\lambda \sigma} \varepsilon_{ \lambda \sigma \mu \nu},$$
where $\varepsilon_{ \lambda \sigma \mu \nu}$ are the components of the Levi-Civita symbol, and its energy-momentum tensor is
$$T[F]_{\mu \nu} :=   F_{\mu \beta} {F_{\nu}}^{\beta}- \frac{1}{4}\eta_{\mu \nu} F_{\rho \sigma} F^{\rho \sigma}.$$
Note that $T[F]_{\mu \nu}$ is symmetric and traceless, i.e. $T[F]_{\mu \nu}=T[F]_{\nu \mu}$ and ${T[F]_{\mu}}^{\mu}=0$. This last point is specific to the dimension $3$ and engenders additional difficulties in the analysis of the Maxwell equations in high dimensions (see Section $3.3.2$ of \cite{dim4} for more details). We have an alternative form of the Maxwell equations.

\begin{Lem}\label{maxwellbis}
Let $G$ be a $2$-form and $J$ be a $1$-form both sufficiently regular. Then,
\begin{align*}
	\left\{	
	\begin{aligned}	
		\nabla^{\mu} G_{\mu \nu} & =  J_{\nu} \\	
		\nabla^{\mu} {}^* \! G_{\mu \nu} & = 0	
	\end{aligned}	
	\right.	
	& \hspace{2cm} \Leftrightarrow & 	
	\left\{
	\begin{aligned}	
		\nabla_{[ \lambda} G_{\mu \nu ]}&=0 \\	
		\nabla_{[ \lambda} {}^* \! G_{\mu \nu ]} &= \varepsilon_{ \lambda \mu \nu \kappa} J^{\kappa},
	\end{aligned}
	\right.
\end{align*}
where $\nabla_{[ \lambda} H_{\mu \nu ]}:= \nabla_{ \lambda} H_{\mu \nu }+\nabla_{ \mu} H_{ \nu \lambda }+\nabla_{ \nu} H_{ \lambda \mu  }$.

\end{Lem}
\begin{proof}
Consider for instance $\nabla^{\mu} G_{\mu 1} =J_1$ and $\nabla^i {}^* G_{i0} =0$. As 
\begin{eqnarray}
\nonumber G^{01} & = & {}^* \! G_{23} \varepsilon_{ 01 23} \hspace{2mm} = \hspace{2mm} {}^* \! G_{23}, \hspace{9.5mm} G^{21} \hspace{2mm} = \hspace{2mm} {}^* \! G_{03} \varepsilon_{ 21 03} \hspace{2mm} = \hspace{2mm} {}^* \! G_{30}, \hspace{17mm} G^{31} \hspace{2mm} = \hspace{2mm} {}^* \! G_{02} \varepsilon_{ 31 02} \hspace{2mm} = \hspace{2mm} G_{02}, \\ 
\nonumber  {}^* \! G_{10} & =  & G_{23} \varepsilon_{ 2 3 1 0} \hspace{2mm} = \hspace{2mm} -G_{23}, \hspace{8mm} {}^* \! G_{20} \hspace{2mm} = \hspace{2mm} G_{31} \varepsilon_{  3 1 2 0} \hspace{2mm} = \hspace{2mm} -G_{31}  \hspace{5.5mm} \text{and} \hspace{5.5mm} {}^* \! G_{30} \hspace{2mm} = \hspace{2mm} G_{12} \varepsilon_{ 1 2 3 0} \hspace{2mm} = -\hspace{2mm} G_{12},
\end{eqnarray}
we have
$$ \nabla^{\mu} G_{\mu 1} =J_1 \Leftrightarrow \nabla_0 {}^* \! G_{23}+\nabla_2 {}^* \! G_{30}+\nabla_3 {}^* \! G_{02}=J_1  \hspace{8mm} \text{and} \hspace{8mm} \nabla^i {}^* G_{i0} =0 \Leftrightarrow \nabla_1 G_{23}+\nabla_2 G_{31}+\nabla_3 G_{12} =0.$$
The equivalence of the two systems can be obtained by similar computations. 
\end{proof}
We can then compute the divergence of the energy momentum tensor of an electromagnetic field.

\begin{Cor}\label{tensordiv}
Let $G$ and $J$ be as in the previous lemma. Then, $\nabla^{\mu} T[G]_{\mu \nu}=G_{\nu \lambda} J^{\lambda}$.
\end{Cor}

\begin{proof}
Using the previous lemma, we have
$$G_{\mu \rho} \nabla^{\mu} {G_{\nu}}^{\rho} \hspace{1mm} = \hspace{1mm}  G^{\mu \rho} \nabla_{\mu} G_{\nu \rho}  \hspace{1mm} = \hspace{1mm} \frac{1}{2} G^{\mu \rho} (\nabla_{\mu} G_{\nu \rho}-\nabla_{\rho} G_{\nu \mu}) \hspace{1mm} = \hspace{1mm} \frac{1}{2} G^{\mu \rho} \nabla_{\nu} G_{\mu \rho} \hspace{1mm} = \hspace{1mm} \frac{1}{4} \nabla_{\nu} (G^{\mu \rho} G_{\mu \rho}).$$
Hence,
$$\nabla^{\mu} T[G]_{\mu \nu} \hspace{1mm} = \hspace{1mm} \nabla^{\mu} (G_{\mu \rho}){G_{\nu}}^{\rho}+\frac{1}{4} \nabla_{\nu} (G^{\mu \rho} G_{\mu \rho})-\frac{1}{4}\eta_{\mu \nu} \nabla^{\mu} (G^{\sigma \rho} G_{\sigma \rho})\hspace{1mm} = \hspace{1mm} G_{\nu \rho} J^{\rho}.$$
\end{proof}

Finally, the null components of the energy-momentum tensor of a $2$-form $G$ are given by
\begin{equation}\label{tensorcompo}
T[G]_{L L}=|\alpha(G)|^2, \hspace{10mm} T[G]_{\underline{L} \hspace{0.5mm} \underline{L}}=|\underline{\alpha}(G)|^2 \hspace{10mm} \text{and} \hspace{10mm} T[G]_{L \underline{L}}=|\rho(G)|^2+|\sigma(G)|^2.
\end{equation}

\subsection{The vector fields of the Poincar\'e group and their complete lifts}\label{subsecvector}

We present in this section the commutation vector fields for the Maxwell equations and those for the relativistic transport operator. Let $\Po$ be the generators of the Poincar\'e algebra, i.e. the set containing
\begin{flalign*}
& \hspace{1cm} \text{the translations\footnotemark} \hspace{18mm} \partial_{\mu}, \hspace{2mm} 0 \leq \mu \leq 3, & \\
& \hspace{1cm} \text{the rotations} \hspace{24mm} \Omega_{ij}=x^i\partial_{j}-x^j \partial_i, \hspace{2mm} 1 \leq i < j \leq 3, & \\
& \hspace{1cm} \text{the hyperbolic rotations} \hspace{7mm} \Omega_{0k}=t\partial_{k}+x^k \partial_t, \hspace{2mm} 1 \leq k \leq 3. 
\end{flalign*}
\footnotetext{In this article, we will denote $\partial_{x^i}$, for $1 \leq i \leq 3$, by $\partial_{i}$ and sometimes $\partial_t$ by $\partial_0$.}
We also consider $\mathbb{K}:= \Po \cup \{ S \}$, where $S=x^{\mu} \partial_{\mu}$ is the scaling vector field and $\Or := \{ \Omega_{12}, \hspace{1mm} \Omega_{13}, \hspace{1mm} \Omega_{23} \}$, the set of the rotational vector fields. The vector fields of $\mathbb{K}$ are well known for commuting with the wave and the Maxwell equations (see Proposition \ref{comumax1} below). However, to commute the operator $T=v^{\mu} \partial_{\mu}$, one should consider, as in \cite{FJS}, the complete lifts of the vector fields of $\Po$.

\begin{Def}\label{Deflift}
Let $V$ be a vector field of the form $V^{\beta} \partial_{\beta}$. Then, the complete lift $\widehat{V}$ of $V$ is defined by
$$\widehat{V}=V^{\beta} \partial_{\beta}+v^{\gamma} \frac{\partial V^i}{\partial x^{\gamma}} \partial_{v^i}.$$
We then have $\widehat{\partial}_{\mu}=\partial_{\mu}$ for all $0 \leq \mu \leq 3$, $$\widehat{\Omega}_{ij}=x^i \partial_j-x^j \partial_i+v^i \partial_{v^j}-v^j \partial_{v^i}, \hspace{2mm} \text{for} \hspace{2mm} 1 \leq i < j \leq 3, \hspace{8mm} \text{and} \hspace{8mm} \widehat{\Omega}_{0k} = t\partial_k+x^k \partial_t+v^0 \partial_{v^k}, \hspace{2mm} \text{for} \hspace{2mm} 1 \leq k \leq 3.$$
\end{Def}

One can check that $[T,\widehat{Z}]=0$ for all $Z \in \Po$. As we also have $[T,S]=T$, we consider
$$\K := \{ \widehat{Z} \hspace{1mm} / \hspace{1mm} Z \in \Po \} \cup \{ S \}$$
and we will, for simplicity, denote by $\widehat{Z}$ an arbitrary vector field of $\K$, even if $S$ is not a complete lift. These vector fields and the averaging in $v$ almost commute in the following sense.

\begin{Lem}\label{lift}
Let $f : [0,T[ \times \mathbb{R}^3_x \times \left( \R^3_v \setminus \{ 0 \} \right) \rightarrow \mathbb{R} $ be a sufficiently regular function. We have, almost everywhere,
$$\forall \hspace{0.5mm} Z \in \mathbb{K}, \hspace{6mm} \left|Z\left( \int_{v \in \R^3 } |f| dv \right) \right| \leq \int_{v \in \R^3} |f| dv+ \sum_{  \widehat{Z} \in \K } \int_{v \in \R^3 } | \widehat{Z} f | dv .$$
\end{Lem}
\begin{proof}
Let us consider, for instance, the case where $Z=\Omega_{12}=x^1 \partial_2-x^2 \partial_1$. Then, integrating by parts in $v$, we have almost everywhere
\begin{eqnarray}
\nonumber \left| \Omega_{12}\left( \int_{v \in \R^3 } |f| dv \right) \right| & = & \left| \int_{v \in \R^3 } \widehat{\Omega}_{12} \left( |f|  \right)  dv -\int_{v \in \R^3 } \Big( v^1 \partial_{v^2} \left( |f|  \right)-v^2 \partial_{v^1} \left( |f|  \right)   \Big) dv \right| \\ \nonumber
& = & \left| \int_{v \in \R^3 } \frac{ f}{| f|} \widehat{\Omega}_{12} \left( f \right)  dv +0 \right| \hspace{2mm} \leq \hspace{2mm} \int_{v \in \R^3 } \left| \widehat{\Omega}_{12} \left(  f \right) \right| dv.
\end{eqnarray}
\end{proof}
The vector space generated by each of the sets defined in this section is an algebra.

\begin{Lem}
Let $\mathbb{L}$ be either $\K$, $\mathbb{K}$, $\Po$ or $\Or$. Then for all $(Z_1,Z_2) \in \mathbb{L}^2$, $[Z_1,Z_2]$ is a linear combination of vector fields of $\mathbb{L}$.

\end{Lem}

We consider an ordering on each of the sets $\mathbb{O}$, $\mathbb{P}$, $\mathbb{K}$ and $\widehat{\mathbb{P}}_0$. We take orderings such that, if $\mathbb{P}= \{ Z^i / \hspace{2mm} 1 \leq i \leq |\mathbb{P}| \}$, then $\mathbb{K}= \{ Z^i / \hspace{2mm} 1 \leq i \leq |\mathbb{K}| \}$, with $Z^{|\mathbb{K}|}=S$, and
$$ \K= \left\{ \widehat{Z}^i / \hspace{2mm} 1 \leq i \leq |\K| \right\}, \hspace{2mm} \text{with} \hspace{2mm} \left( \widehat{Z}^i \right)_{ 1 \leq i \leq |\Po|}=\left( \widehat{Z^i} \right)_{ 1 \leq i \leq |\Po|} \hspace{2mm} \text{and} \hspace{2mm} \widehat{Z}^{|\K|}=S  .$$
If $\mathbb{L}$ denotes $\mathbb{O}$, $\mathbb{P}$, $\mathbb{K}$ or $\widehat{\mathbb{P}}_0$, and  $\beta \in \{1, ..., |\mathbb{L}| \}^q$, with $q \in \mathbb{N}^*$, we will denote the differential operator $\Gamma^{\beta_1}...\Gamma^{\beta_r} \in \mathbb{L}^{|\beta|}$ by $\Gamma^{\beta}$. For a vector field $Y$, we will denote by $\mathcal{L}_Y$ the Lie derivative with respect to $Y$ and if $Z^{\gamma} \in \mathbb{K}^{q}$, we will write $\mathcal{L}_{Z^{\gamma}}$ for $\mathcal{L}_{Z^{\gamma_1}}...\mathcal{L}_{Z^{\gamma_q}}$. 

Let us recall, by the following classical result, that the derivatives tangential to the cone behave better than others.

\begin{Lem}\label{goodderiv}
The following relations hold,
$$(t-r)\underline{L}=S-\frac{x^i}{r}\Omega_{0i}, \hspace{6mm} (t+r)L=S+\frac{x^i}{r}\Omega_{0i} \hspace{6mm} \text{and} \hspace{6mm} re_A=\sum_{1 \leq i < j \leq 3} C^{i,j}_A \Omega_{ij},$$
where the $C^{i,j}_A$ are uniformly bounded and depend only on spherical variables. We also have
$$(t-r)\partial_t =\frac{t}{t+r}S-\frac{x^i}{t+r}\Omega_{0i} \hspace{8mm} \text{and} \hspace{8mm} (t-r) \partial_i = \frac{t}{t+r} \Omega_{0i}- \frac{x^i}{t+r}S- \frac{x^j}{t+r} \Omega_{ij}.$$
\end{Lem}
Finally, we introduce the vector field
$$\overline{K}_0:= \frac{1}{2}\tau_+^2 L+ \frac{1}{2}\tau_-^2 \underline{L},$$
which will be used as a multiplier.

\subsection{Commutation of the Vlasov-Maxwell system}\label{subseccomu}

Let us start by proving the following result. For convenience, we extend the Kronecker symbol to vector fields, i.e. $\delta_{X, Y}=1$ if $X=Y$ and $\delta_{X, Y}=0$ otherwise.
\begin{Lem}\label{comumax1}
Let $G$ be a $2$-form and $g$ a function, both sufficiently regular. For all $\widehat{Z} \in \K$,
$$ \widehat{Z} \left( G \left( v, \nabla_v g \right) \right) = \mathcal{L}_Z(G) \left( v, \nabla_v g \right)+G \left( v, \nabla_v \widehat{Z} g \right)-2 \delta_{\widehat{Z},S} G \left( v, \nabla_v g \right).$$
If $\nabla^{\mu} G_{\mu \nu} = J(g)_{\nu}$ and $\nabla^{\mu} {}^* \! G_{\mu \nu} = 0$, then
$$\forall \hspace{0.5mm} Z \in \mathbb{K}, \hspace{6mm} \nabla^{\mu} \mathcal{L}_Z(G)_{\mu \nu} = J(\widehat{Z}g)_{\nu}+3 \delta_{Z,S} J(g)_{\nu} \hspace{15mm} \text{and} \hspace{15mm} \nabla^{\mu} {}^* \! \mathcal{L}_Z(G)_{\mu \nu} = 0.$$
\end{Lem}
\begin{proof}
Let $\widehat{Z} \in \K$ and define $Z_v:=\widehat{Z}-Z$. Then,
$$ \widehat{Z} \left( G \left( v , \nabla_v g \right) \right) \hspace{2mm} = \hspace{2mm} \mathcal{L}_Z(G) \left( v, \nabla_v g \right)+G \left( [Z,v], \nabla_v g \right)+G \left( v,[Z,\nabla_v g] \right)+G \left( Z_v(v), \nabla_v g\right)+G \left( v, Z_v \left(\nabla_v g \right) \right) .$$
Note now that
\begin{itemize}
\item $S_v=0$ and $[S,v]=-v$,
\item $[Z,v]=-Z_v(v)$ if $Z \in \mathbb{P}$.
\end{itemize}
The first identity is then implied by
\begin{itemize}
\item $[\partial, \nabla_v g]=\nabla_v \partial(g)$ and $[S, \nabla_v g ]= \nabla_v S(g)-\nabla_v g$.
\item $[Z, \nabla_v g]+Z_v \left( \nabla_v g \right)= \nabla_v \widehat{Z}(g)$, if $Z \in \mathbb{O}$.
\item $[Z, \nabla_v g]+Z_v \left( \nabla_v g \right)= \nabla_v \widehat{Z}(g)-\frac{v}{v^0}\partial_{v^i} g $ and $G(v,v)=0$, if $Z = \Omega_{0i}$.
\end{itemize}
Recall now that if\footnote{This can be obtained by straightforward computations in cartesian coordinates. We also refer to Proposition $3.3$ of \cite{CK}.} $Z \in \mathbb{K}$,
$$ \nabla^{\mu} \mathcal{L}_{Z}(G)_{\mu \nu} = \mathcal{L}_{Z} (J(g))_{\nu} +2\delta_{Z,S} J(g)_{\nu}  \hspace{10mm} \text{and} \hspace{10mm} \nabla^{\mu} {}^* \! \mathcal{L}_{Z}(G)_{\mu \nu} = 0.$$
One then only have to notice that 
$$ \mathcal{L}_{S} (J(g))=J(S g)+J(g) \hspace{8mm} \text{and} \hspace{8mm} \forall \hspace{0.5mm} Z \in \mathbb{P}, \hspace{5mm} \mathcal{L}_{Z} (J(g)) = J( \widehat{Z} g).$$
This follows from $ \mathcal{L}_{Z} (J(g))_{\nu} = Z (J(g)_{\nu})+ \partial_{\nu} \left( Z^{\lambda} \right) J(g)_{\lambda}$ and integration by parts in $v$. For instance,
\begin{eqnarray}
\nonumber  \mathcal{L}_{\Omega_{12}} (J(g))_{\nu} & = & \int_{v \in \R^3} \frac{v_{\nu}}{v^0}  \left( \widehat{\Omega}_{12}-v^1 \partial_{v^2}+v^2 \partial_{v^1} \right) g dv +\delta_{1,\nu} \int_{v \in \R^3} \frac{v_2}{v^0} g dv -\delta_{2, \nu}  \int_{v \in \R^3} \frac{v_1}{v^0} g dv \\ \nonumber
& = & J \left( \widehat{\Omega}_{12} g \right)+\delta_{2,\nu} \int_{v \in \R^3} \frac{v^1}{v^0} g dv -\delta_{1, \nu}  \int_{v \in \R^3} \frac{v^1}{v^0} g dv+\delta_{1,\nu} \int_{v \in \R^3} \frac{v_2}{v^0} g dv -\delta_{2, \nu}  \int_{v \in \R^3} \frac{v_1}{v^0} g dv.
\end{eqnarray}
\end{proof}
Iterating Lemma \ref{comumax1}, we can describe the form of the source terms of the commuted Vlasov-Maxwell equations.
\begin{Pro}\label{Maxcom}
Let $(f,F)$ be a sufficiently regular solution to the Vlasov-Maxwell system \eqref{VM1}-\eqref{VM3} and $Z^{\beta} \in \mathbb{K}^{|\beta|}$. There exist integers $n^{\beta}_{\gamma, \kappa}$ and $m^{\beta}_{\xi}$ such that
\begin{eqnarray}
\nonumber T_F \left( \widehat{Z}^{\beta} f \right) & = & \sum_{\begin{subarray}{} |\gamma|+|\kappa| \leq |\beta| \\ \hspace{1mm} |\kappa| \leq |\beta|-1 \end{subarray}} n^{\beta}_{\gamma , \kappa} \mathcal{L}_{Z^{\gamma}}(F) \left( v , \nabla_v \widehat{Z}^{\kappa} (f) \right) , \\ \nonumber
\nabla^{\mu} \mathcal{L}_{Z^{\beta}}(F)_{\mu \nu} & = & \sum_{|\xi| \leq |\beta|} m^{\beta}_{\xi} J \left( \widehat{Z}^{\xi} f \right)_{\nu}, \\ \nonumber
\nabla^{\mu} {}^* \! \mathcal{L}_{Z^{\beta}}(F)_{\mu \nu} & = & 0 .
\end{eqnarray}
\end{Pro}
The main observation is that the structure of the non-linearity $F(v,\nabla_v f)$ is conserved after commutation, which is important since if the source terms of the Vlasov equation behaved as $v^0 |F| |\partial_v f|$, we would not be able to close the energy estimates for the Vlasov field. The other conserved structure is $J(f)$, which is also crucial since a source term behaving as $\int_v |f|dv$ would prevent us to close the energy estimates for the electromagnetic field.

\subsection{Weights preserved by the flow and null components of the velocity vector}\label{sectionweights}

We designate the null components of the velocity vector by $(v^L,v^{\underline{L}},v^{e_1},v^{e_2})$, so that
$$v=v^L L+ v^{\underline{L}} \underline{L}+v^{e_A}e_A, \hspace{6mm} v^L=\frac{v^0+v^r}{2} \hspace{6mm} \text{and} \hspace{6mm} v^{\underline{L}}=\frac{v^0-v^r}{2}.$$
For simplicity we will write $v^A$ instead of $v^{e_A}$. We introduce, as in \cite{FJS}, the following set of weights
$$ \mathbf{k}_0 := \left\{\frac{v^{\mu}}{v^0} \hspace{1mm} \Big/ \hspace{1mm} 0 \leq \mu \leq 3 \right\} \cup \left\{ x^{\mu}\frac{v^{\nu}}{v^0}-x^{\nu}\frac{v^{\mu}}{v^0} \hspace{1mm} \Big/ \hspace{1mm} \mu \neq \nu \right\} \cup \left\{ x^{\mu} \frac{v_{\mu}}{v^0} \right\}$$
and we will denote $x^{\mu}\frac{v^{\nu}}{v^0}-x^{\nu}\frac{v^{\mu}}{v^0}$ by $z_{\mu \nu}$ and $x^{\mu} \frac{v_{\mu}}{v^0}$ by $s_0$. They are preserved by the flow of $T$ and by the action of $\K$. More precisely, we have the following result.
\begin{Lem}\label{weights}
Let $z \in \mathbf{k}_0$ and $\widehat{Z} \in \K$. Then,
$$ T(z)=0, \hspace{10mm} \widehat{Z}(v^0z) \in v^0 \mathbf{k}_0 \cup \{ 0 \} \hspace{10mm} \text{and} \hspace{10mm} \left| \widehat{Z} (z) \right| \leq \sum_{w \in \V} |w|.$$
\end{Lem}
\begin{proof}
The first property ensues from straightforward computations. For the second one, let us consider for instance $tv^1-x^1v^0$, $x^1v^2-x^2v^1$, $\widehat{\Omega}_{12}$ and $\widehat{\Omega}_{02}$. We have
\begin{eqnarray}
\nonumber \widehat{\Omega}_{12}(tv^1-x^1v^0) & = & -tv^2-x^2v^0, \hspace{24mm} \widehat{\Omega}_{12}(x^1v^2-x^2v^1 ) \hspace{2mm} = \hspace{2mm} 0,\\ \nonumber
\widehat{\Omega}_{02}(tv^1-x^1v^0) & = & x^2v^1-x^1v^2 \hspace{10mm} \text{and} \hspace{10mm} \widehat{\Omega}_{02}(x^1v^2-x^2v^1 )  \hspace{2mm} = \hspace{2mm} x^1v^0-tv^1.
\end{eqnarray}
The other cases are similar and the third property follows directly from the second one.
\end{proof}
The following inequalities, which should be compared to those of Lemma \ref{goodderiv}, suggest how we will use these weights.
\begin{Lem}\label{weights1}
We have,
$$ \frac{v^L}{v^0} \lesssim \frac{1}{\tau_-} \sum_{z \in \V} |z|, \hspace{8mm}  \frac{v^{\underline{L}}}{v^0}+\frac{|v^A|}{v^0} \lesssim \frac{1}{\tau_+} \sum_{z \in \V} |z| \hspace{8mm} \text{and} \hspace{8mm} |v^A| \lesssim \sqrt{v^0v^{\underline{L}}}.$$
\end{Lem} 
\begin{proof}
Note first that
$$2(t-r)\frac{v^L}{v^0} = -s_0-\frac{x^i}{r}z_{0i}, \hspace{10mm} 2(t+r)\frac{v^{\underline{L}}}{v^0} = -s_0+\frac{x^i}{r}z_{0i} \hspace{10mm} \text{and} \hspace{10mm} rv_A = v^0 C_A^{i,j} z_{ij},$$
where $C^{i,j}_A$ are bounded functions depending only on the spherical variables such as $re_A= C^{i,j}_A \Omega_{i,j}$. This gives the first two estimates. For the last one, use also that $4r^2v^Lv^{\underline{L}} =\sum_{k < l} |v^0z_{kl}|^2 $, which comes from
\begin{eqnarray}
\nonumber 4r^2v^Lv^{\underline{L}} & = & (rv^0)^2-\left(x^i v_i \right)^2 \hspace{2mm} = \hspace{2mm}  \sum_{i=1}^3(r^2 -|x^i|^2)|v_i|^2-2\sum_{1 \leq k < l \leq 3} x^kx^lv_kv_l, \\ \nonumber 
 \sum_{1 \leq k < l \leq 3} |v^0 z_{kl}|^2 & = & \sum_{1 \leq k < l \leq 3} |x^k|^2 |v_l|^2+|x^l|^2 |v_k|^2-2x^kx^lv_kv_l \hspace{2mm} = \hspace{2mm} \sum_{i=1}^3 \sum_{j \neq i} |x^j|^2 |v_i|^2 -2\sum_{1 \leq k < l \leq n} x^kx^lv_kv_l.
\end{eqnarray}
\end{proof}
\begin{Rq}
There are certain differences to the massive case, where $v^0 = \sqrt{m^2+|v|^2}$ and $m >0$.
\begin{itemize}
\item The inequality $1 \lesssim v^0 v^{\underline{L}}$ does not hold.
\item As $x^{i} v_{i}-tv^0$ does not commute with the massive relativistic transport operator, we rather consider the set of weights $\mathbf{k}_1 := \V \setminus \{s_0 \}$ in this context. Then, the estimate $\tau_-v^L+\tau_+ v^{\underline{L}} \lesssim v^0 \sum_{z \in \mathbf{k}_1} |z|$ is merely satisfied in the exterior of the light cone. 
\end{itemize}
\end{Rq}

\subsection{Various subsets of Minkowski spacetime}\label{subsets}

We introduce here several subsets of Minkowski space depending on $t \in \mathbb{R}_+$, $r \in \R_+$, $u \in \mathbb{R}$. Let $\Sp_{t,r}$, $\Sigma_t$, $C_u(t)$ and $V_u(t)$, be the sets defined as
\begin{eqnarray}
\nonumber \Sp_{t,r} & := & \{ (s,y) \in \R_+ \times \R^3 \hspace{1mm} / \hspace{1mm} (s,|y|)=(t,r) \}, \hspace{1.5cm} C_u(t) \hspace{2mm} := \hspace{2mm} \{(s,y)  \in \mathbb{R}_+ \times \mathbb{R}^3 / \hspace{1mm} \hspace{1mm} s \leq t, \hspace{1mm} s-|y|=u \},  \\ \nonumber
 \Sigma_t & := & \{ (s,y) \in \R_+ \times \R^3 \hspace{1mm} / \hspace{1mm} s=t \}, \hspace{29.5mm} V_u(t) \hspace{2mm} := \hspace{2mm} \{ (s,y) \in \mathbb{R}_+ \times \mathbb{R}^3 / \hspace{1mm} s \leq t, \hspace{1mm} s-|y| \leq u \}. 
\end{eqnarray}
The volume form on $C_u(t)$ is given by $dC_u(t)=\sqrt{2}^{-1}r^{2}d\underline{u}d \mathbb{S}^{2}$, where $ d \mathbb{S}^{2}$ is the standard metric on the $2$ dimensional unit sphere. In view of applying the divergence theorem, we also introduce
$$\Sigma_t^u \hspace{2mm} := \hspace{2mm} \{ (s,y) \in \R_+ \times \R^3 \hspace{1mm} / \hspace{1mm} s=t, \hspace{1mm} |y| \geq s-u \}.$$

\begin{tikzpicture}
\draw [-{Straight Barb[angle'=60,scale=3.5]}] (0,-0.3)--(0,5);
\fill[color=gray!35] (2,0)--(5,3)--(9.8,3)--(9.8,0)--(1,0);
\node[align=center,font=\bfseries, yshift=-2em] (title) 
    at (current bounding box.south)
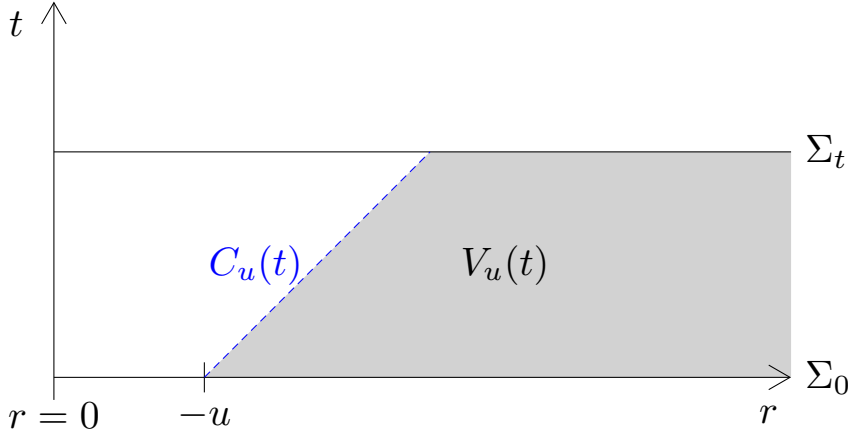
    {The sets $\Sigma_t$, $C_u(t)$ and $V_u(t)$};
\draw (0,3)--(9.8,3) node[scale=1.5,right]{$\Sigma_t$};
\draw (2,0.2)--(2,-0.2);
\draw [-{Straight Barb[angle'=60,scale=3.5]}] (0,0)--(9.8,0) node[scale=1.5,right]{$\Sigma_0$};
\draw[densely dashed] (2,0)--(5,3) node[scale=1.5,left, midway] {$C_u(t)$};
\draw (6,1.5) node[ color=black!100, scale=1.5] {$V_u(t)$}; 
\draw (0,-0.5) node[scale=1.5]{$r=0$};
\draw (2,-0.5) node[scale=1.5]{$-u$};
\draw (-0.5,4.7) node[scale=1.5]{$t$};
\draw (9.5,-0.5) node[scale=1.5]{$r$};   
\end{tikzpicture}

We also introduce a dyadic partition of $\R_+$ by considering the sequence $(t_i)_{i \in \mathbb{N}}$ and the functions $(T_i(t))_{i \in \mathbb{N}}$ defined by
$$t_0=0, \hspace{5mm} t_i = 2^i \hspace{5mm} \text{if} \hspace{5mm} i \geq 1, \hspace{5mm} \text{and} \hspace{5mm} T_{i}(t)= t \mathds{1}_{t \leq t_i}(t)+t_i \mathds{1}_{t > t_i}(t).$$
We then define the truncated cones $C^i_u(t)$ adapted to this partition by
$$C_u^i(t) := \left\{ (s,y) \in \R_+ \times \R^3 \hspace{2mm} / \hspace{2mm} t_i \leq s \leq T_{i+1}(t), \hspace{2mm} s-|y| =u \right\}= \left\{ (s,y) \in C_u(t) \hspace{2mm} / \hspace{2mm} t_i \leq s \leq T_{i+1}(t) \right\}.$$
The following lemma will be used several times during this paper. It depicts that we can foliate $[0,t] \times \R^3$ by $(\Sigma_s)_{0 \leq s \leq t}$, $(C_u(T))_{u \leq t}$ or $(C^i_u(T))_{u \leq t, i \in \mathbb{N}}$.
\begin{Lem}\label{foliationexpli}
Let $t>0$ and $g \in L^1([0,t] \times \R^3)$. Then
$$\int_0^t \int_{\Sigma_s} g dx ds \hspace{2mm} = \hspace{2mm} \int_{u=-\infty}^t \int_{C_u(t)} g dC_u(t) \frac{du}{\sqrt{2}} \hspace{2mm} = \hspace{2mm} \sum_{i=0}^{+ \infty} \int_{u=-\infty}^t \int_{C^i_u(t)} g dC^i_u(t) \frac{du}{\sqrt{2}}.$$
\end{Lem}
Note that the sum over $i$ is in fact finite. The second foliation is useful to take advantage of decay in the $t-r$ direction since $\| \tau_-^{-1} \|_{L^{\infty}(C_u(t))} = \tau_-^{-1}$, whereas $\| \tau_-^{-1} \|_{L^{\infty}(\Sigma_s)} = 1$. The last foliation will be used to take advantage of time decay on $C_u(t)$ as we merely have $\|\tau_+^{-1}\|_{L^{\infty}(C_u(t))} = \tau_-^{-1}$, whereas $\|\tau_+^{-1}\|_{L^{\infty}(C^i_u(t))} \leq (1+t_i)^{-1} \leq 3(1+t_{i+1})^{-1}$.

\section{Energy and pointwise decay estimates}\label{sec4}

In this section, we recall classical energy estimates for both the electromagnetic field and the Vlasov field and how to obtain pointwise decay estimates from them.

\subsection{Energy estimates}\label{energy}

For the Vlasov field, we will use the following energy estimate.

\begin{Pro}\label{energyf}
Let $H : [0,T[ \times \R^3_x \times \left( \R^3_v \setminus \{ 0 \} \right) \rightarrow \R $ and $g_0 : \R^3_x \times \left( \R^3_v \setminus \{ 0 \} \right) \rightarrow \R$ be two sufficiently regular functions and $F$ a sufficiently regular $2$-form. Then, $g$, the unique classical solution of 
\begin{eqnarray}
\nonumber T_F(g)&=&H \\ \nonumber
g(0,.,.)&=&g_0,
\end{eqnarray}
satisfies, for all $t \in [0,T[$, the following estimates,
$$ \left\| \int_{v \in \R^3} |g|dv \right\|_{L^1(\Sigma_t)}+ \sup_{u \leq t} \left\| \int_{v \in \R^3} \frac{v^{\underline{L}}}{v^0} |g| dv \right\|_{L^1(C_u(t))} \hspace{2mm} \leq \hspace{2mm}   2 \left\| \int_{v \in \R^3} |g_0| dv \right\|_{L^1(\Sigma_0)} + 2\int_0^t  \int_{\Sigma_s} \int_{v \in \R^3}|H| \frac{dv}{v^0} dxds.$$
\end{Pro}

\begin{proof}
Note first that as $T(|g|)=\frac{g}{|g|}H-\frac{g}{|g|} F ( v, \nabla_v g)$ and since $F$ is a $2$-form, integration by parts in $v$ gives us
$$ \partial_{\mu} \int_v |g| \frac{v^{\mu}}{v^0}dv = \int_v \left( \frac{g}{|g|}\frac{H}{v^0}-\frac{g}{|g|} F\left( \frac{v}{v^0}, \nabla_v g \right) \right) dv = \int_v \left( \frac{g}{|g|}\frac{H}{v^0}- \frac{v^{j}v^i}{(v^0)^3}F_{j i}|g| \right) dv= \int_v  \frac{g}{|g|}H\frac{dv}{v^0} .$$
Hence, the divergence theorem applied to $\int_v |g| \frac{v^{\mu}}{v^0}dv$ in the regions $[0,t] \times \R^3_x$ and $V_u(t)$, for all $u \leq t$, gives
\begin{eqnarray}
\nonumber \int_{\Sigma_t} \int_v |g| dv dx  & \leq &  \int_{\Sigma_0} \int_v |g| dv dx+\int_0^t \int_{\Sigma_s} \left| \int_v  \frac{g}{|g|}H \frac{dv}{v^0} \right| dx ds, \\ \nonumber
\int_{\Sigma^u_t} \int_v  |g| dvdx+\sqrt{2} \int_{C_u(t)} \int_v \frac{v^{\underline{L}}}{v^0} |g| dv dC_u(t)  & \leq &  \int_{\Sigma^u_0} \int_v  |g| dvdx+\int_0^t \int_{\Sigma^u_s} \left| \int_v  \frac{g}{|g|}H \frac{dv}{v^0} \right| dx ds ,
\end{eqnarray}
which implies the result.
\end{proof}
We then define, for $(Q,q) \in \mathbb{N}^2$,
\begin{eqnarray}\label{defE1}
 \E[g](t) & := & \left\| \int_{v \in \R^3} |g|dv \right\|_{L^1(\Sigma_t)}+ \sup_{u \leq t} \left\| \int_{v \in \R^3} \frac{v^{\underline{L}}}{v^0} |g| dv \right\|_{L^1(C_u(t))}, \\ 
  \E^{q}_{Q}[g](t) & := & \sum_{ \begin{subarray}{} \widehat{Z}^{\beta} \in \K^{|\beta|} \\ \hspace{0.5mm} |\beta| \leq Q  \end{subarray}} \sum_{z \in \V}  \E \left[  z^q \widehat{Z}^{\beta} g \right](t). \label{defE2}
\end{eqnarray}
We now introduce the energy norms, related to the electromagnetic field, used in this paper. We consider, for the remaining of this section, $G$ a sufficiently regular $2$-form defined on $[0,T[ \times \R^3$ and we denote by $(\alpha,\underline{\alpha},\rho,\sigma)$ its null decomposition. We moreover suppose that $G$ satisfies
\begin{eqnarray}
\nonumber \nabla^{\mu} G_{\mu \nu} & = & J_{\nu} \\ \nonumber
\nabla^{\mu} {}^* \! G_{\mu \nu} & = & 0,
\end{eqnarray}
with $J$ be a sufficiently regular $1$-form defined on $[0,T[ \times \R^3$.
\begin{Def}\label{defMax1}
Let $k \in \mathbb{N}$. We define, for $t \in [0,T[$,
\begin{eqnarray}
\nonumber \mathcal{E}^{\overline{K}_0}[G](t) & := & 4\int_{\Sigma_t} T[G]_{0 \nu} \overline{K}_0^{\nu} dx+ 2 \sup_{u \leq t } \int_{C_u(t)} T[G]_{L \nu}\overline{K}^{\nu}_0dC_u(t) \\ \nonumber 
 & = & \int_{\Sigma_t} \tau_+^2|\alpha|^2+\tau_-^2|\underline{\alpha}|^2+(\tau_+^2+\tau_-^2)(|\rho|^2+|\sigma|^2)dx+\sup_{u \leq t}\int_{C_u(t)} \tau_+^2|\alpha|^2+\tau_-^2(|\rho|^2+|\sigma|^2)d C_u(t), \\
 \nonumber \mathcal{E}^{\partial_t,k}[G](t) & = & \int_{\Sigma_t} \tau_-^{2} \log^{-k}(1+\tau_-) \left( |\alpha|^2+|\underline{\alpha}|^2+2|\rho|^2+2|\sigma|^2 \right)dx .
\end{eqnarray}
For $N \in \mathbb{N}^*$, we also introduce
$$  \mathcal{E}_N[G] :=  \sum_{\begin{subarray}{l}  Z^{\gamma} \in \mathbb{K}^{|\gamma|} \\ \hspace{0.5mm} |\gamma| \leq N \end{subarray}}  \mathcal{E}^{\overline{K}_0}[\mathcal{L}_{Z^{\gamma}}(G)] \hspace{12mm} \text{and} \hspace{12mm} \mathcal{E}^{k}_N[G] :=  \sum_{\begin{subarray}{l}  Z^{\gamma} \in \mathbb{K}^{|\gamma|} \\ \hspace{0.5mm} |\gamma| \leq N \end{subarray}}  \mathcal{E}^{\partial_t,k}[\mathcal{L}_{Z^{\gamma}}(G)].$$
\end{Def}
\begin{Rq}
During the proof of Theorem \ref{theorem}, we will have a small growth on $\mathcal{E}_N[G]$ and not on $\mathcal{E}^{k}_N[G]$. The second energy norm will then permit us to obtain the optimal decay rate in the $t+r$ direction on $\underline{\alpha}$, which will be crucial for closing the energy estimates for the Vlasov field.
\end{Rq}
The following energy estimates hold.
\begin{Pro}\label{energyMax1}
We have, with $\overline{C} >0$ a constant depending on $k$, for all $ t \in [0,T[$,
\begin{eqnarray}
 \nonumber \mathcal{E}^{\overline{K}_0}[G](t) & \leq & 2 \mathcal{E}^{\overline{K}_0}[G](0)+8 \int_0^t \int_{\Sigma_s} |G_{\mu \nu} J^{\mu} \overline{K}_0^{\nu}| dx ds, \\
 \nonumber \mathcal{E}^{\partial_t, k}[G](t) & \leq & \overline{C} \mathcal{E}^{\overline{K}_0}[G](0)+\overline{C} \int_0^t \int_{\Sigma_s} \frac{\tau_-^2}{\log^k(1+\tau_-)}|G_{\mu 0} J^{\mu}| dx ds+\overline{C}\sup_{0 \leq s \leq t} \left( \log^{1-k}(2+s) \mathcal{E}^{\overline{K}_0}[G](s) \right). 
\end{eqnarray}
\end{Pro}
\begin{proof}
Denoting $ T[G]$ by $T$ and using Corollary \ref{tensordiv}, we have, as $\nabla^{\mu} \overline{K}_0^{\nu}+\nabla^{\nu} \overline{K}_0^{\mu}=4t\eta^{\mu \nu}$ and ${T_{\mu}}^{\mu}=0$,
$$  \nabla^{\mu} \left( T_{\mu \nu} \overline{K}_0^{\nu} \right)  \hspace{1mm} = \hspace{1mm} \nabla^{\mu} T_{\mu \nu}\overline{K}_0^{\nu}+T_{\mu \nu} \nabla^{\mu} \overline{K}_0^{\nu}  \hspace{1mm} = \hspace{1mm} G_{\nu \lambda} J^{\lambda} \overline{K}_0^{\nu} +\frac{1}{2} T_{\mu \nu} \left( \nabla^{\mu} \overline{K}_0^{\nu}+\nabla^{\nu} \overline{K}_0^{\mu} \right) \hspace{1mm} = \hspace{1mm} G_{\nu \lambda} J^{\lambda} \overline{K}_0^{\nu}. $$
Applying the divergence theorem in $[0,t] \times \R^3$ and in $V_{u}(t)$, for all $u \leq t$, yields
\begin{eqnarray}\label{eq:2}
 \int_{\Sigma_t}  T_{0 \nu} \overline{K}_0^{\nu} dx & = & \int_{\Sigma_0}  T_{0 \nu} \overline{K}_0^{\nu}  dx-\int_0^t \int_{\Sigma_s }   G_{\mu \nu} J^{\mu} \overline{K}_0^{\nu}  dx ds, \\ 
 \int_{\Sigma_t^u}  T_{0 \nu} \overline{K}_0^{\nu} dx + \frac{1}{\sqrt{2}} \int_{C_{u}(t)} T_{L \nu} \overline{K}_0^{\nu}dC_{u}(t) & = & \int_{\Sigma^u_0}  T_{0 \nu} \overline{K}_0^{\nu}  dx-\int_0^t \int_{\Sigma^u_s }   G_{\mu \nu} J^{\mu} \overline{K}_0^{\nu}  dx ds. \label{eq:22}
 \end{eqnarray}
Notice, using \eqref{tensorcompo} and $2 \overline{K}_0=\tau_+^2L+\tau_-^2 \underline{L}$, that
$$ 4T_{0 \nu} \overline{K}_0^{\nu} = \tau_+^2 |\alpha|+\tau_-^2|\underline{\alpha}|+(\tau_+^2+\tau_-^2)(|\rho|+|\sigma|) \hspace{10mm} \text{and} \hspace{10mm} 2 T_{L \nu} \overline{K}_0^{\nu} = \tau_+^2 |\alpha|^2+\tau_-^2|\rho|^2+\tau_-^2|\sigma|^2.$$
It then only remains, to obtain the first estimate, to take the supremum over all $u \leq t$ in \eqref{eq:22} and to combine it with \eqref{eq:2}. For the other one, note first using Corollary \ref{tensordiv} and \eqref{tensorcompo} that
\begin{eqnarray}
\nonumber \left| \nabla^{\mu}  \left( \tau_-^2 \log^{-k}(1 \hspace{-0.3mm} + \hspace{-0.3mm} \tau_-) T_{\mu 0} \right) \hspace{-0.2mm} \right| \hspace{-0.5mm}  & = & \hspace{-0.5mm} \left| \tau_-^2 \log^{-k}(1 \hspace{-0.3mm} + \hspace{-0.3mm} \tau_-) \nabla^{\mu} T_{\mu 0}-\frac{1}{2}\underline{L} \left( \tau_-^2 \log^{-k}(1 \hspace{-0.3mm} + \hspace{-0.3mm} \tau_-) \right) T_{L 0} \right|  \\ \nonumber
& = & \hspace{-0.5mm} \left| \tau_-^2 \log^{-k}(1 \hspace{-0.3mm} + \hspace{-0.3mm} \tau_-)  \nabla^{\mu} T_{\mu 0} -u \log^{-k}(1 \hspace{-0.3mm} + \hspace{-0.3mm} \tau_-) \hspace{-0.2mm} \left(2-\frac{k \tau_-}{1 \hspace{-0.3mm} + \hspace{-0.3mm} \tau_-} \log^{-1}(1 \hspace{-0.3mm} + \hspace{-0.3mm} \tau_-) \right) \hspace{-0.2mm} T_{L 0} \right| \\ \nonumber
& \lesssim & \hspace{-0.5mm} \tau_-^2 \log^{-k}(1+\tau_-) \left| G_{0 \lambda} J^{\lambda} \right|+\frac{\tau_+^2}{\tau_+ \log^{k+1}(1 \hspace{-0.3mm} + \hspace{-0.3mm} \tau_+)} \left( \left| \alpha  \right|^2+\left| \rho \right|^2+\left| \sigma  \right|^2 \right).
\end{eqnarray}
Consequently, applying the divergence theorem in $[0,t] \times \R^3$, we obtain
$$ \int_{\Sigma_t} \frac{\tau_-^2}{ \log^{k}(1+\tau_-)} T_{00}dx  \lesssim \int_{\Sigma_0} (1+r)^2 T_{00}dx+\int_0^t \int_{\Sigma_s } \frac{\tau_-^2}{\log^{k}(1+\tau_-)}  \left|   G_{0 \nu} J^{\nu} \right| dx ds+\int_0^t \frac{\mathcal{E}^{\overline{K}_0}[G](s)}{(1+s) \log^{k+1}(2+s)} ds.$$
The result then follows from $4T_{00} =  |\alpha|^2+|\underline{\alpha}|^2+2|\rho|^2+2|\sigma|^2$ and $\int_0^{+\infty} (1+s)^{-1} \log^{-2}(2+s) ds < + \infty$.
\end{proof}
\subsection{Decay estimates}
\subsubsection{Decay estimates for velocity averages}
We prove in this subsection an $L^{\infty}-L^1$ and an $L^2-L^1$ Klainerman-Sobolev inequality for velocity averages. The $L^{\infty}-L^1$ one was originally proved in \cite{FJS} (see Theorem $6$) and we propose here a shorter proof. Let us start with the following lemma.
\begin{Lem}\label{Sobsphere}
Let $g : \mathbb{S}^2 \times \left( \R^3_v \setminus \{ 0 \} \right) \rightarrow \R$ a sufficiently regular function. Then, with $\Omega^{\beta} \in \Or^{|\beta|}$,
$$\left\| \int_{v \in \R^3}  |g| dv \right\|_{L^{\infty}(\mathbb{S}^2)} \lesssim   \sum_{ |\beta| \leq 2 } \left\| \int_{v \in \R^3}   \left| \widehat{\Omega}^{\beta} g \right| dv \right\|_{L^1(\mathbb{S}^2)} \hspace{-1mm}, \hspace{1cm}  \left\| \int_{v \in \R^3}  |g| dv \right\|_{L^{2}(\mathbb{S}^2)} \lesssim   \sum_{ |\beta| \leq 1  } \left\| \int_{v \in \R^3}   \left| \widehat{\Omega}^{\beta} g \right| dv \right\|_{L^1(\mathbb{S}^2)} \hspace{-1mm} .$$
\end{Lem}
\begin{proof}
Let $\omega \in \mathbb{S}^2$ and $(\theta, \varphi ) \in ]0, \pi [ \times ]0, 2 \pi [$ be a system of spherical coordinates such that $\omega_1:=\theta(\omega)=\frac{\pi}{2}$ and $\omega_2 := \varphi ( \omega )= \pi$. Using a one dimensional Sobolev inequality, that $|\partial_{\theta} u| \lesssim \sum_{\Omega \in \Or} |\Omega u |$ and Lemma \ref{lift}, we have,
\begin{align*}
 \int_v  |g|(\omega_1,\omega_2,v) dv \hspace{1mm} & \lesssim \hspace{1mm} \int_{|\theta - \frac{\pi}{2}|\leq \frac{\pi}{4}} \left( \left| \int_v |g|(\theta,\omega_2,v) dv \right|+\left| \partial_{\theta} \int_v  |g|(\theta,\omega_2,v) dv \right| \right) d \theta \\ & \lesssim \hspace{1mm} \sum_{ \begin{subarray}{} \Omega^{\kappa} \in \Or^{|\kappa|} \\ \hspace{1mm} |\kappa| \leq 1 \end{subarray} } \int_{|\theta - \frac{\pi}{2}|\leq \frac{\pi}{4}} \int_v  | \widehat{\Omega}^{\kappa} g|(\theta,\omega_2,v) dv \sin (\theta ) d \theta.
 \end{align*}
We obtain similarly that 
\begin{align*}
\int_{|\theta - \frac{\pi}{2}|\leq \frac{\pi}{4}}  \int_v  | \widehat{\Omega}^{\kappa} g|(\theta,\omega_2,v) dv \sin ( \theta ) d \theta \hspace{1mm} & \lesssim \hspace{1mm}  \sum_{ \begin{subarray}{} \Omega^{\gamma} \in \Or^{|\gamma|} \\ \hspace{1mm} |\gamma| \leq 1 \end{subarray}} \int_{\varphi =0}^{2 \pi} \int_{|\theta - \frac{\pi}{2}|\leq \frac{\pi}{4}} \int_v  | \widehat{\Omega}^{\gamma} \widehat{\Omega}^{\kappa} g|(\theta,\varphi,v) dv  \sin ( \theta )  d \theta d \varphi \\ & \lesssim \hspace{1mm} \sum_{ \begin{subarray}{} \Omega^{\beta} \in \Or^{|\beta|} \\ \hspace{1mm} |\beta| \leq 2 \end{subarray} } \left\| \int_{v}   \left| \widehat{\Omega}^{\beta} g \right| dv \right\|_{L^1(\mathbb{S}^2)},
\end{align*}
which implies the first inequality. For the other one, by a standard $L^2-L^1$ Sobolev inequality, one has
$$ \left\| \int_{v \in \R^3}  |g| dv \right\|_{L^{2}(\mathbb{S}^2)} \lesssim \left\| \int_{v \in \R^3}  |g| dv \right\|_{L^{1}(\mathbb{S}^2)}+\left\| \partial_{\theta} \int_{v \in \R^3}  |g| dv \right\|_{L^1(\mathbb{S}^2)}+\left\| \partial_{\varphi} \int_{v \in \R^3}  |g| dv \right\|_{L^1(\mathbb{S}^2)} .$$
It then remains to apply Lemma \ref{lift} again.
\end{proof}
\begin{Pro}\label{KS1}
Let $f:[0,T[ \times \R^3 \times \left( \R^3_v \setminus \{ 0 \} \right)$ be a sufficiently regular function, $z \in \V$ and $j \in \mathbb{N}$. Then, 
$$\forall \hspace{0.5mm} (t,x) \in [0,T[ \times \R^3, \hspace{1cm} \int_{v \in \R^3}  |z^jf|(t,x,v) dv \hspace{2mm} \lesssim \hspace{2mm} \frac{(j+1)^3}{\tau_+^2 \tau_-} \sum_{\begin{subarray}{l} \widehat{Z}^{\beta} \in \K^{|\beta|} \\ \hspace{1mm} |\beta| \leq 3 \end{subarray}} \sum_{w \in \V} \left\| \int_{v \in \R^3} \left| w^j \widehat{Z}^{\beta} f \right| dv \right\|_{ L^1(\Sigma_t)}.$$
\end{Pro}
\begin{proof}
Note first that if $j \geq 1$, the inequality follows from the case $j=0$ as, using Lemma \ref{weights},
$$ \left| \widehat{Z}^{\beta} \left( z^j f \right) \right| \lesssim j^{|\beta|} \sum_{w \in \V} \sum_{|\kappa| \leq |\beta|} \left| w^j \widehat{Z}^{\kappa} f \right|.$$
Suppose now that $j=0$ and consider $(t,x)=(t,|x| \omega) \in [0,T[ \times \R^3$.
\begin{itemize}
\item If $1+t \leq 2|x|$, one has, using Lemmas \ref{goodderiv} and \ref{lift},
\begin{eqnarray}
\nonumber |x|^2 \tau_- \int_v |f|(t,|x|\omega,v) dv & = & -|x|^2 \int_{r=|x|}^{+ \infty} \partial_r \left( \tau_- \int_v |f|(t,r\omega,v) dv \right) dr\\ \nonumber
& \lesssim & |x|^2 \sum_{Z \in \mathbb{K}} \int_{r=|x|}^{+ \infty} \left( \int_v  |f|(t,r\omega,v) dv + \left| Z \left( \int_v  |f|(t,r\omega,v) dv \right) \right| \right) dr \\ \nonumber
& \lesssim & \sum_{ \widehat{Z} \in \K} \int_{r=|x|}^{+ \infty} \left( \int_v  |f|(t,r\omega,v) dv+ \int_v  |\widehat{Z} f|(t,r\omega,v) dv \right) r^2 dr.
\end{eqnarray}
It then remains to apply Lemma \ref{Sobsphere} and to remark that $\tau_+ \lesssim r$ in the region considered.
\item Otherwise $1+t \geq 2|x|$, so that, with $\tau :=1+t$,
$$\forall \hspace{0.5mm} |y| \leq \frac{1}{4}, \hspace{8mm} \tau \leq 10(1+|t-|x+\tau y||).$$
Thus, for all sufficiently regular functions $h$, $1 \leq i \leq 3$ and almost all $|y| \leq \frac{1}{4}$, we have, using Lemmas \ref{goodderiv} and then \ref{lift},
\begin{eqnarray}
\nonumber
\left| \partial_{y^i} \left( \int_v |h|(t,x+ \tau y,v) dv \right) \right|  & = &  \left| \tau \int_v  \Big(  \partial_i |h| \Big) (t,x+ \tau y,v) dv \right| \\ \nonumber
& \lesssim &   \left| (1+|t-|x+ \tau y|| ) \int_v  \Big(  \partial_i |h| \Big) (t,x+\tau y,v) dv \right| \\ \nonumber & \lesssim &  \sum_{Z \in \mathbb{K}} \left|  \int_v   \Big(  Z |h| \Big)(t,x+ \tau y,v) dv \right| \\  & \lesssim &  \sum_{\begin{subarray}{} \widehat{Z}^{\kappa} \in \K^{|\kappa|} \\ \hspace{1mm} |\kappa| \leq 1 \end{subarray}} \int_v  \left| \widehat{Z}^{\kappa} h \right| (t,x+\tau y,v) dv . \label{eq:KSproof}
\end{eqnarray}
Hence, using alternatively three times a one dimensional Sobolev inequality and then \eqref{eq:KSproof}, we get,
\begin{eqnarray}
\nonumber \int_v |f|(t,x,v) dv   & \lesssim &  \sum_{n=0}^1  \int_{|y^1| \leq \frac{1}{4 \sqrt{3}}} \left| \left(\partial_{y^1} \right)^n \left( \int_v |f|(t,x^1+\tau y^1,x^2,x^3,v) dv \right)  \right| dy^1 \\ \nonumber
& \lesssim &  \sum_{|\kappa| \leq 1} \int_{|y^1| \leq \frac{1}{4 \sqrt{3}}} \int_v |\widehat{Z}^{\kappa} f|(t,x^1+\tau y^1,x^2,x^3,v) dv dy^1 \\ \nonumber
& \lesssim & \sum_{n=0}^1  \sum_{ |\kappa| \leq 1 }  \int_{|y^1| \leq \frac{1}{4 \sqrt{3}}} \int_{|y^2| \leq \frac{1}{4 \sqrt{3}}}  \left| \left(\partial_{y^2} \right)^n \left( \int_v |\widehat{Z}^{\kappa} f|(t,x+\tau (y^1,y^2,0),v) dv \right)  \right| dy^2dy^1 \\ \nonumber
& \lesssim &  \sum_{|\kappa| \leq 2} \int_{|y^1| \leq \frac{1}{4 \sqrt{3}}} \int_{|y^2| \leq \frac{1}{4 \sqrt{3}}} \int_v |\widehat{Z}^{\kappa} f|(t,x+ \tau (y^1,y^2,0),v) dv dy^2 dy^1 \\ \nonumber
& \lesssim & \sum_{n=0}^1  \sum_{ |\kappa| \leq 2 }  \int_{|y| \leq \frac{1}{4}}   \left| \left(\partial_{y^3} \right)^n  \left( \int_v |\widehat{Z}^{\kappa} f|(t,x+\tau y,v) dv \right)  \right| dy \\ \nonumber
& \lesssim &  \sum_{|\kappa| \leq 3} \int_{|y| \leq \frac{1}{4 }} \int_v |\widehat{Z}^{\kappa} f|(t,x+\tau y,v) dv dy.
\end{eqnarray}
The result then follows from the change of variables $z=\tau y$ and that $\tau_- \leq \tau_+ \lesssim \tau$ in the region studied.
\end{itemize}
\end{proof}
We now turn on the $L^2-L^1$ Klainerman-Sobolev inequality.
\begin{Pro}\label{KS2}
Let $f:[0,T[ \times \R^3 \times \left( \R^3_v \setminus \{ 0 \} \right)$ be a sufficiently regular function, $z \in \V$ and $j \in \mathbb{N}$. Then, 
$$\forall \hspace{0.5mm} t \in [0,T[, \hspace{1cm} \left\| \tau_+ \tau_-^{\frac{1}{2}} \int_{v \in \R^3}  |z^jf| dv \right\|_{L^2(\Sigma_t)} \lesssim (j+1)^2 \sum_{\begin{subarray}{l} \widehat{Z}^{\beta} \in \K^{|\beta|} \\ \hspace{1mm} |\beta| \leq 2 \end{subarray}} \sum_{w \in \V} \left\| \int_{v \in \R^3} \left| w^j \widehat{Z}^{\beta} f \right| dv \right\|_{ L^1(\Sigma_t)}.$$
\end{Pro}
\begin{proof} As previously, we can restrict the proof to the case $j=0$. We introduce $\delta = \frac{1}{4}$ for convenience and we suppose first that $t \geq 1$. The idea is classical and consists in splitting $\Sigma_t$ into the three domains, $|x| \leq \frac{t}{2}$, $|x| \geq \frac{3}{2}t$ and $\frac{1}{2}t \leq |x| \leq \frac{3}{2}t$.

$\bullet$ \textit{ Step $1$, the interior region.} Applying a local two-dimensional $L^2-L^1$ Sobolev inequality to the function $x \mapsto \int_v |f|(t,tx,v) dv$, we get
$$ \int_{|x| \leq \frac{1}{2}} \left| \int_v |f|(t,tx,v)dv \right|^2 dx_1 dx_2 dx_3  \lesssim  \sum_{q=0}^1 \int_{|x_3| \leq \frac{1}{2}} \left| \int_{x_1^2+x_2^2 \leq \frac{1}{4}-x_3^2+\delta^2} \int_v \left((t \partial_{x_1,x_2})^q |f| \right)(t,tx,v)dv  dx_1 d x_2 \right|^2 d x_3.$$
As $t-|tx| \geq \frac{1}{4}t$ on the domain of integration since $|x| \leq \frac{1}{2}+\delta \leq \frac{3}{4}$, Lemmas \ref{goodderiv} and \ref{lift} gives us
$$ \int_{|x| \leq \frac{1}{2}} \left| \int_v |f|(t,tx,v)dv \right|^2 dx_1 dx_2 dx_3  \lesssim  \sum_{|\beta| \leq 1} \int_{|x_3| \leq \frac{1}{2}} \left| \int_{x_1^2+x_2^2 \leq \frac{1}{4}-x_3^2+\delta^2} \int_v \left|\widehat{Z}^{\beta} f \right|(t,tx,v)dv  dx_1 d x_2 \right|^2 d x_3.$$
Now, one can obtain similarly, using a one-dimensional $L^2-L^1$ Sobolev inequality in the variable $x_3$, that
$$ \int_{|x| \leq \frac{1}{2}} \left| \int_v |f|(t,tx,v)dv \right|^2 dx_1 dx_2 dx_3  \lesssim  \sum_{|\beta| \leq 2} \left| \int_{|x_3| \leq \frac{1}{2}+\delta}  \int_{x_1^2+x_2^2 \leq \frac{1}{4}-x_3^2+\delta^2} \int_v \left|\widehat{Z}^{\beta} f \right|(t,tx,v)dv  dx_1 d x_2 d x_3 \right|^2.$$
Since $\tau_+^2 \tau_- \lesssim t^3$ if $|x| \leq \frac{1}{2}t$, we finally obtain, by the change of variables $y=tx$,
\begin{equation}\label{larget}
\left\| \tau_+ \tau_-^{\frac{1}{2}} \int_{v \in \R^3}  |f|(t,y,v) dv \right\|_{L^2(|y| \leq \frac{1}{2}t)} \lesssim t^3 \left\|  \int_{v \in \R^3}  |f|(t,tx,v) dv \right\|_{L^2(|x| \leq \frac{1}{2})} \lesssim \sum_{|\beta| \leq 2} \left\| \int_v \left|\widehat{Z}^{\beta} f \right| dv \right\|_{L^1(\Sigma_t)}.
\end{equation}
$\bullet$ \textit{Step $2$, the exterior region.} Let us introduce, for $i \in \mathbb{N}$, the following sets\footnote{The constant hidden in $\lesssim$ in the upcoming computations will not depend on $i$.}
$$ X_i := \left\{ y \times \Sigma_t \hspace{1mm} / \hspace{1mm} 3 t \times 2^{i-1} \leq |y| < 3t \times 2^i \right\}, \hspace{0.9cm} \text{and} \hspace{0.9cm} Y_i := \left\{ y \times \Sigma_t \hspace{1mm} / \hspace{1mm} 5 t \times 2^{i} \leq 4 |y| < 13t \times 2^{i} \right\}.$$
In the domain considered here, where $|x| \geq \frac{3}{2}t$, we have $\tau_+ \lesssim |x|$ but we cannot follow exactly what we have done for the interior region as we cannot view $|x|$ as a parameter. However, as for $i \in \mathbb{N}$, $2^i t \sim \tau_+$ on $X_i$ and
$$\forall \hspace{0.5mm} \frac{3}{2}-\delta \leq |x| \leq 3+\delta, \hspace{1cm} |2^i t x|-t \geq \left( 2^i\frac{5}{4}-1 \right)t \geq \frac{1}{4} \times 2^i t,$$ we can apply similar operations to $x \mapsto \int_v |f|(t,2^itx,v) dv$ as to $x \mapsto \int_v |f|(t,tx,v) dv$ previously and obtain
\begin{eqnarray}
\nonumber \int_{\frac{3}{2} \leq |x| \leq 3} \left| \int_v |f|(t,2^i tx,v)dv \right|^2 \hspace{-0.5mm} dx \hspace{-1mm} & \lesssim & \hspace{-1mm} \sum_{|\beta| \leq 2} \left| \int_{ |x_3| \leq 3+\delta}  \int_{\frac{9}{4}-x_3^2-\delta^2 \leq x_1^2+x_2^2 \leq 9-x_3^2+\delta^2} \int_v \left|\widehat{Z}^{\beta} g \right|(t,2^itx,v)dv  dx \right|^2 \\ \nonumber
& \lesssim & \hspace{-1mm} \sum_{|\beta| \leq 2} \left| \int_{\frac{5}{4} \leq |x| \leq \frac{13}{4}} \int_v \left|\widehat{Z}^{\beta} g \right|(t,2^itx,v)dv  dx \right|^2.
\end{eqnarray}
As $\tau_+^2 \tau_- \lesssim 2^{3i} t^3$ on $X_i$, we finally obtain by the change of variables $y=2^i t x$, 
$$\left\| \tau_+ \tau_-^{\frac{1}{2}} \int_{v \in \R^3}  |f| dv \right\|_{L^2(X_i)} \lesssim 2^{3i} t^3 \left\|  \int_{v \in \R^3}  |f|(t,2^itx,v) dv \right\|_{L^2( \frac{3}{2} \leq |x| \leq 3)} \lesssim \sum_{|\beta| \leq 2} \left\| \int_v \left|\widehat{Z}^{\beta} f \right| dv \right\|_{L^1(Y_i)}.$$
As $Y_i \cap Y_j = \varnothing$ if $|i-j| \geq 2$, $X_i \cap X_j = \varnothing$ if $i \neq j$ and since $\{ y \in \Sigma_t \hspace{1mm} / \hspace{1mm} |y| \geq \frac{3}{2}t \}=\cup_{i=0}^{+\infty} X_i$, we get
\begin{equation}\label{largex}
\left\| \tau_+ \tau_-^{\frac{1}{2}} \int_{v \in \R^3}  |f|(t,y,v) dv \right\|_{L^2(|y| \geq \frac{3}{2}t)} \hspace{2mm} \lesssim \hspace{2mm}  \sum_{|\beta| \leq 2} \left\| \int_v \left|\widehat{Z}^{\beta} f \right| dv \right\|_{L^1(\Sigma_t)}.
\end{equation}

$\bullet$ \textit{Step $3$, the remaining domain.} We now focus on the region $\frac{1}{2}t \leq |x| \leq \frac{3}{2}t$. We will obtain the $\tau_+$ integrated decay with the rotational vector fields through Sobolev inequalities on the spheres. To obtain the $\sqrt{\tau_-}$ decay, note first that $|u| \leq \frac{1}{2}t$ in this region (recall that $u=t-|x|$). The idea to capture the decay in $u$ will then be to divide the domain in the disjoint union of the sets
$$ V_i := \{ y \in \Sigma_t \hspace{1mm} / \hspace{1mm} 2^{-i-1}t < |t-|y|| \leq 2^{-i} t \}, \hspace{5mm} i \in \mathbb{N}^*.$$
Let $\omega \in \mathbb{S}^2$. Applying a $L^2-L^1$ Sobolev inequality to $g: s \mapsto  \int_v|f|(t,t(1-2^{-i}s)\omega,v)dv$, we obtain
$$ \int_{ \frac{1}{2} \leq |s| \leq 1} \left| \int_v |f|(t,t(1-2^{-i}s)\omega,v) dv \right|^2 ds \hspace{2mm} \lesssim \hspace{2mm} \sum_{j=0}^1 \left| \int_{ \frac{1}{4} \leq |s| \leq \frac{5}{4} } \left| \int_v \left( t2^{-i} \partial_r|f| \right)^j(t,t(1-2^{-i}s)\omega,v) dv \right| ds \right|^2 . $$
Since $\frac{1}{4}2^{-i}t \leq |t-|t-2^{-i}ts||$ for all $i \in \mathbb{N}^*$ and $\frac{1}{4} \leq |s| \leq \frac{5}{4}$, we get, using Lemmas \ref{goodderiv} and \ref{lift} that
$$ \int_{ \frac{1}{2} \leq |s| \leq 1} \left| \int_v |f|(t,t(1-2^{-i}s)\omega,v) dv \right|^2 ds \hspace{2mm} \lesssim \hspace{2mm} \sum_{|\kappa| \leq 1} \left| \int_{ \frac{1}{4} \leq |s| \leq \frac{5}{4} } \left| \int_v | \widehat{Z}^{\kappa} f |(t,t(1-2^{-i}s)\omega,v) dv \right| ds \right|^2 . $$
The change of variables $r=t(1-2^{-i}s)$ gives
$$ t2^{-i} \int_{ 2^{-i-1}t \leq |t-r| \leq 2^{-i}t} \left| \int_v |f|(t,r\omega,v) dv \right|^2 dr \hspace{2mm} \lesssim \hspace{2mm} \sum_{|\kappa| \leq 1} \left| \int_{ 2^{-i-2}t \leq |t-r| \leq 5 \times 2^{-i-2}t } \left| \int_v | \widehat{Z}^{\kappa} f |(t,r\omega,v) dv \right| dr \right|^2. $$
As previously with the domains $X_i$ and $Y_i$, we take the sum over $i \in \mathbb{N}^*$ and we get
$$  \int_{ \frac{1}{2}t \leq r \leq \frac{3}{2}t} |t-r| \left|  \int_v |f|(t,r\omega,v) dv \right|^2 dr \hspace{2mm} \lesssim \hspace{2mm} \sum_{|\kappa| \leq 1} \left| \int_{ \frac{3}{8}t \leq r \leq \frac{13}{8}t } \left| \int_v | \widehat{Z}^{\kappa} f|(t,r\omega,v) dv \right| dr \right|^2. $$
By simpler operations, one can also obtain that
$$  \int_{ t-\frac{1}{2} \leq r \leq t+\frac{1}{2}}  \left|  \int_v |f|(t,r\omega,v) dv \right|^2 dr \hspace{2mm} \lesssim \hspace{2mm} \sum_{|\kappa| \leq 1} \left| \int_{ t-\frac{5}{8} \leq r \leq t+\frac{5}{8} } \left| \int_v | \widehat{Z}^{\kappa} f|(t,r\omega,v) dv \right| dr \right|^2. $$
Integrating each side of these inequalities over $\mathbb{S}^2$ and applying Proposition \ref{Sobsphere} to the right hand sides, we get
$$  \int_{ \frac{1}{2}t \leq r \leq \frac{3}{2}t} \int_{\omega \in \mathbb{S}^2} \tau_- \left|\int_v |f|(t,r\omega,v dv \right|^2 d \mathbb{S}^2 dr \hspace{2mm} \lesssim \hspace{2mm} \sum_{|\kappa| \leq 2} \left| \int_{ \frac{3}{8}t \leq r \leq \frac{13}{8}t } \int_{\omega \in \mathbb{S}^2} \left| \int_v | \widehat{Z}^{\kappa} f|(t,r\omega,v) dv \right| d \mathbb{S}^2 dr \right|^2. $$
Finally, multiply both side of the inequality by $t^2$ and use $\tau_+ \lesssim t \leq 2r $ on the domain of integration in order to obtain
\begin{eqnarray}\label{reqt}
\int_{ \frac{1}{2}t \leq r \leq \frac{3}{2}t} \int_{\omega \in \mathbb{S}^2} \hspace{-0.3mm} \tau_+^2 \tau_- \left| \int_v |f|(t,r\omega,v dv \right|^2 \hspace{-0.3mm} d \mathbb{S}^2 r^2 dr \lesssim \sum_{|\kappa| \leq 2} \left| \int_{ \frac{3}{8}t \leq r \leq \frac{13}{8}t } \int_{\omega \in \mathbb{S}^2} \left| \int_v | \widehat{Z}^{\kappa} f|(t,r\omega,v) dv \right| d \mathbb{S}^2 r^2 dr \right|^2 \hspace{-0.3mm}.
\end{eqnarray}
The result then follows from \eqref{larget}, \eqref{largex} and \eqref{reqt}. The case $t \leq 1$ can be treated similarly, repeating the arguments of Steps $1$ and $2$ since in that case $\tau_+^2 \tau_- \lesssim (1+r)^3$ and
$$ \Sigma_t = \{ y \in \Sigma_t \hspace{1mm} / \hspace{1mm}   |y| \leq 2^{-1} \} \cup \left( \cup_{i=0}^{+ \infty} \{ y \in \Sigma_t \hspace{1mm} / \hspace{1mm}   2^{i-1} \leq |y| < 2^i \} \right).$$
\end{proof}
\subsubsection{Pointwise Decay estimates for the electromagnetic field}

In this section, we follow mostly \cite{CK}. We first present certain identities and inequalities between quantities linked to the null decomposition of a $2$-form (see Section \ref{basicelec} for its definition), then we recall Sobolev inequalities and, finally, we prove the desired pointwise decay estimates for the electromagnetic field.

 For the remaining part of this section, we consider $G$ a $2$-form and $J$ a $1$-form, both sufficiently regular and defined on $[0,T[ \times \R^3$, such that
\begin{eqnarray}
\nonumber \nabla^{\mu} G_{\mu \nu} & =& J_{\nu}, \\ \nonumber
\nabla^{\mu} {}^* \! G_{ \mu \nu } & = & 0.
\end{eqnarray}
Aside from Lemma \ref{alphaem} and the estimate on $\alpha(G)$ in Proposition \ref{decayMaxell}, all the results of this subsection apply to a general $2$-form.

$\bullet$ \textbf{Preparatory results.} \newline
To lighten the presentation, we prove the three upcoming lemmas in Appendix \ref{appendixD}. 
 
\begin{Lem}\label{randrotcom}
Let $\Omega \in \Or$. Then, the operators $\mathcal{L}_{\Omega}$ and $\nabla_{\partial_r}$ commute with the null decomposition of $G$ as well as with each other, i.e., denoting by $\zeta$ any of the null component $\alpha$, $\underline{\alpha}$, $\rho$ or $\sigma$,
$$ [\mathcal{L}_{\Omega}, \nabla_{\partial_r}] G=0, \hspace{1.2cm} \mathcal{L}_{\Omega}(\zeta(G))= \zeta ( \mathcal{L}_{\Omega}(G) ) \hspace{1.2cm} \text{and} \hspace{1.2cm} \nabla_{\partial_r}(\zeta(G))= \zeta ( \nabla_{\partial_r}(G) ) .$$
Similar results hold for $\mathcal{L}_{\Omega}$ and $\nabla_{\partial_t}$, $\nabla_L$ or $\nabla_{\underline{L}}$. For instance, $\nabla_{L}(\zeta(G))= \zeta ( \nabla_{L}(G) )$.
\end{Lem}
We now give a more precise version of Lemma $3.3$ of \cite{CK}.
\begin{Lem}\label{null}
Denoting by $\zeta$ any of the null component $\alpha$, $\underline{\alpha}$, $\rho$ or $\sigma$, we have
$$\tau_- \left| \nabla_{\underline{L}} \zeta (G)\right|+\tau_+ \left| \nabla_{L} \zeta (G) \right| \lesssim \sum_{|\gamma| \leq 1} \left|  \zeta \left( \mathcal{L}_{Z^{\gamma}}(G) \right) \right| \hspace{5mm} \text{and} \hspace{5mm} (1+r)\left| \slashed{\nabla} \zeta (G) \right| \lesssim |\zeta(G)|+\sum_{ \Omega \in \Or} \left|  \zeta \left( \mathcal{L}_{\Omega}(G) \right) \right| .$$
\end{Lem}
The following equation will be useful in order to obtain a strong decay estimate on $\alpha(G)$.
\begin{Lem}\label{alphaem}
Denoting by $(\alpha, \underline{\alpha}, \rho, \sigma)$ the null decomposition of $G$, we have
\begin{equation}\label{eq:alphaem}
\nabla_{\underline{L}} \alpha_A-\frac{\alpha_A}{r}+\slashed{\nabla}_{e_A} \rho+\varepsilon_{BA} \slashed{\nabla}_{e_B} \sigma=J_A .
\end{equation}
\end{Lem}
The following result will allow us to treat part of the interior of the light cone.
\begin{Lem}\label{decayint}
Let $U$ be a smooth tensor field defined on $[0,T[ \times \mathbb{R}^3$. Then,
$$\forall \hspace{0.5mm} t \in [0,T[, \hspace{8mm} \sup_{|x| \leq 1+\frac{t}{2}} |U(t,x)| \lesssim \frac{1}{(1+t)^{\frac{5}{2}}} \sum_{|\gamma| \leq 2}  \| \tau_- \mathcal{L}_{Z^{\gamma}}(U)(t,y) \|_{L^2 \left( |y| \leq 2+\frac{3}{4}t \right)}.$$
\end{Lem}
\begin{proof}
As $|\mathcal{L}_{Z^{\gamma}}(U)| \lesssim \sum_{|\beta| \leq |\gamma|} \sum_{\mu, \nu} | Z^{\beta} (U_{\mu \nu})|$, it suffices to prove the result for each component of the tensor and we can assume that $U$ is a scalar function. Let $(t,x) \in [0,T[ \times \R^3$ such that $|x| \leq 1+ \frac{1}{2}t$. Apply a standard $L^2$ Sobolev inequality to $V: y \mapsto U(t,x+\frac{1+t}{4}y)$ and then make a change of variables to get
$$|U(t,x)|=|V(0)| \lesssim \sum_{|\beta| \leq 2} \| \partial_x^{\beta} V \|_{L^2_y(|y| \leq 1)} \lesssim \left( \frac{1+t}{4} \right)^{-\frac{3}{2}} \sum_{|\beta| \leq 2} \left( \frac{1+t}{4} \right)^{|\beta|} \| \partial_x^{\beta} U(t,.) \|_{L^2_y(|y-x| \leq \frac{1+t}{4})}.$$
Observe now that $|y-x| \leq \frac{1+t}{4}$ implies $|y| \leq 2+\frac{3}{4}t$ and that $1+t \lesssim \tau_-$ on that domain. Consequently, using Lemma \ref{goodderiv} and that $[Z, \partial]$, for $Z \in \mathbb{K}$, is either $0$ or a translation, we have
$$( 1+t )^{|\beta|+1} \| \partial_x^{\beta} U(t,.) \|_{L^2_y(|y-x| \leq \frac{1+t}{4})} \hspace{2mm} \lesssim \hspace{2mm} \| \tau_-^{|\beta|+1} \partial_x^{\beta} U(t,.) \|_{L^2_y(|y| \leq 2+\frac{3}{4}t)} \hspace{2mm} \lesssim  \hspace{2mm} \sum_{|\gamma| \leq |\beta|} \| \tau_- Z^{\gamma} U(t,.) \|_{L^2_y(|y| \leq 2+\frac{3}{4}t)}.$$ 
\end{proof}
We refer to Lemma $2.3$ of \cite{CK} for a proof of the following two Sobolev inequalities, which will permit us to deal with the remaining region.
\begin{Lem}\label{Sob}
Let $U$ be a sufficiently regular tensor field defined on $\R^3$ and denote $\sum_{|\beta| \leq k} \left|\mathcal{L}_{\Omega^{\beta}}(U) \right|^2$, where $\Omega^{\beta} \in \Or^{|\beta|}$, by $|U|^2_{\mathbb{O},k}$. There exists a uniform constant $C>0$, independent of $U$, such that
\begin{eqnarray}
\nonumber \forall \hspace{0.5mm} t \in \R_+, \hspace{2mm} \forall \hspace{0.5mm} |x| \geq \frac{1}{2}t+1, \hspace{8mm} |U(x)| & \leq & \frac{C}{|x|\tau_-^{\frac{1}{2}}} \left( \int_{|y| \geq \frac{1}{2}t+1} |U(y)|^2_{\mathbb{O},2}+\tau_-^2|\nabla_{\partial_r} U(y) |^2_{\mathbb{O},1} dy \right)^{\frac{1}{2}}, \\ \nonumber
\forall \hspace{0.5mm} x \neq 0, \hspace{8mm} |U(x)| & \leq & \frac{C}{|x|^{\frac{3}{2}}} \left( \int_{|y| \geq |x|} |U(y)|^2_{\mathbb{O},2}+|y|^2|\nabla_{\partial_r} U(y) |^2_{\mathbb{O},1} dy \right)^{\frac{1}{2}}.\end{eqnarray}
\end{Lem}

$\bullet$ \textbf{Decay estimates for $G$.} \newline 
We are now ready to prove the pointwise decay estimates on the electromagnetic field.
\begin{Pro}\label{decayMaxell}
Let $k \in \mathbb{N}^*$. Then, we have for all $(t,x) \in [0,T[ \times \R^3$,
\begin{eqnarray}
\nonumber |\rho|(t,x)+ |\sigma|(t,x) & \lesssim & \frac{ \sqrt{\mathcal{E}_2[G](t)}}{\tau_+^{2}\tau_-^{\frac{1}{2}}}, \hspace{40mm} |\underline{\alpha}|(t,x) \hspace{2mm} \lesssim \hspace{2mm} \sqrt{\mathcal{E}^{k}_2[G](t)} \frac{ \log^{\frac{k}{2}}(1+\tau_-) }{\tau_+ \tau_-^{\frac{3}{2}}}, \\ \nonumber
  |\alpha|(t,x) & \lesssim & \frac{\sqrt{ \mathcal{E}_2[G](t) }+\sum_{|\kappa| \leq 1} \|r^2 \mathcal{L}_{Z^{\kappa}}(J)_A\|_{L^2(\Sigma_t)}}{\tau_+^{\frac{5}{2}}} . 
  \end{eqnarray}
\end{Pro}
\begin{proof}
We fix for this proof $(t,x) \in [0,T[ \times \R^3$. If $|x| \leq 1+\frac{1}{2}t$, the result follows from Proposition \ref{decayint}. We then suppose $|x| \geq 1+\frac{t}{2}$. During this proof, $\Omega^{\beta}$ will always denote a combination of rotational vector fields, i.e. $\Omega^{\beta} \in \Or^{|\beta|}$. Let $\zeta$ be either $\alpha$, $ \rho$ or $ \sigma$. As $\nabla_{\partial_r}$ and $\mathcal{L}_{\Omega}$ commute with the null decomposition (see Lemma \ref{randrotcom}), Lemma \ref{Sob} gives us
$$r^4 \tau_- |\zeta|^2 \lesssim \int_{ |y| \geq \frac{t}{2}+1} |r \zeta |^2_{\mathbb{O},2}+\tau_-^2|\nabla_{\partial_r} (r \zeta) |_{\mathbb{O},1}^2 dy \lesssim \sum_{ |\beta| \leq 1 } \sum_{|\gamma| \leq 2} \int_{ |y| \geq \frac{t}{2}+1} r^2| \zeta ( \mathcal{L}_{Z^{\gamma}} (G) |^2+\tau_-^2r^2|  \zeta ( \mathcal{L}_{\Omega^{\beta}} (\nabla_{\partial_r} G)) |^2 dy.$$
As $\nabla_{\partial_r}$ commute with $\mathcal{L}_{\Omega}$ as well as with the null decomposition (see Lemma \ref{randrotcom}), we have, using $2\partial_r= L-\underline{L}$ and Lemma \ref{null},
\begin{equation}\label{zetaeq2}
 |  \zeta ( \mathcal{L}_{\Omega} (\nabla_{\partial_r} G)) |+|  \zeta ( \nabla_{\partial_r} G) | \lesssim |  \nabla_L \zeta ( \mathcal{L}_{\Omega} (G) |+|  \nabla_{\underline{L}} \zeta ( \mathcal{L}_{\Omega} (G) |+| \nabla_L \zeta (  G) |+| \nabla_{\underline{L}} \zeta (  G) | \lesssim \frac{1}{\tau_-}\sum_{ |\gamma| \leq 2} | \zeta ( \mathcal{L}_{Z^{\gamma}} (G) |.
 \end{equation}
Since $\tau_+ \lesssim r \leq \tau_+$ in the region considered, we finally obtain
$$\tau_+^4 \tau_- |\zeta|^2 \hspace{2mm} \lesssim \hspace{2mm} \sum_{|\gamma| \leq 2} \int_{ |y| \geq \frac{t}{2}+1} \tau_+^2| \zeta ( \mathcal{L}_{Z^{\gamma}} (G) |^2 dx \hspace{2mm} \lesssim \hspace{2mm} \mathcal{E}_2[G](t).$$
We improve now the estimate on $\alpha$. As $\nabla^{\mu} \mathcal{L}_{\Omega} (G)_{\mu \nu} = \mathcal{L}_{\Omega}(J)_{\nu}$ and $\nabla^{\mu} {}^* \! \mathcal{L}_{\Omega} (G)_{\mu \nu} = 0$ for all $\Omega \in \Or$, we have according to \eqref{eq:alphaem} that
$$\forall \hspace{0.5mm} |\beta| \leq 1, \hspace{15mm} \nabla_{\underline{L}} \alpha(\mathcal{L}_{\Omega^{\beta}} (G))_A=\frac{1}{r}\alpha(\mathcal{L}_{\Omega^{\beta}} (G))_A-\slashed{\nabla}_{e_A}\rho (\mathcal{L}_{\Omega^{\beta}} (G)) +\varepsilon_{AB} \slashed{\nabla}_{e_B} \sigma (\mathcal{L}_{\Omega^{\beta}} (G))+\mathcal{L}_{\Omega^{\beta}}(J)_A.$$
Consequently, we get using Lemma \ref{null} that for all $\Omega \in \Or$,
\begin{equation}\label{alphaeq2}
  |  \alpha ( \nabla_{\partial_r} G) |+|  \alpha ( \mathcal{L}_{\Omega} (\nabla_{\partial_r} G)) | \lesssim \left|J_A \right| +\left| \mathcal{L}_{\Omega} (J)_A \right|+ \frac{1}{r}\sum_{ |\gamma| \leq 2} | \alpha ( \mathcal{L}_{Z^{\gamma}} (G) |+| \rho ( \mathcal{L}_{Z^{\gamma}} (G) |+| \sigma ( \mathcal{L}_{Z^{\gamma}} (G) |.
 \end{equation}
Applying the second inequality of Lemma \ref{Sob} and using this time \eqref{alphaeq2} instead of \eqref{zetaeq2}, we get
$$\tau_+^5 |\alpha|^2 \hspace{2mm} \lesssim \hspace{2mm} |x|^5 |\alpha|^2 \hspace{2mm} \lesssim \hspace{2mm} \int_{ |y| \geq |x|} |r \alpha|^2_{\mathbb{O},2}+r^2|\nabla_{\partial_r} ( r \alpha) |_{\mathbb{O},1}^2 dy \hspace{2mm} \lesssim \hspace{2mm} \mathcal{E}_2[G](t)+\sum_{|\kappa| \leq 1} \|r^2 \mathcal{L}_{Z^{\kappa}}(J)_A\|^2_{L^2(\Sigma_t)}.$$
Applying the first inequality of Lemma \ref{Sob} to $\tau_- \log^{-\frac{k}{2}}(1+\tau_-) \underline{\alpha}$ and using the same arguments as previously, one has
$$ \frac{r^2\tau_-^3}{\log^k(1+\tau_-)} |\underline{\alpha} |^2 \hspace{2mm} \lesssim \hspace{2mm}   \int_{ |y| \geq \frac{t}{2}+1} \left| \frac{\tau_-}{\log^{\frac{k}{2}}(1+\tau_-)} \underline{\alpha} \right|^2_{\mathbb{O},2}+\tau_-^2 \left| \nabla_{\partial_r} \left( \frac{\tau_-}{\log^{\frac{k}{2}}(1+\tau_-)} \underline{\alpha} \right) \right|_{\mathbb{O},1}^2 dy  \hspace{2mm} \lesssim \hspace{2mm}   \mathcal{E}^{k}_2[G](t).$$
\end{proof}
\section{The null structure of the non-linearity $\mathcal{L}_{Z^{\gamma}}(F) \left( v, \nabla_v \widehat{Z}^{\beta} f \right)$}\label{sec5}

In order to take advantage of the null structure of the Vlasov equation, we will expand quantities such as $\mathcal{L}_{Z^{\gamma}}(F) \left( v,\nabla_v g \right)$, with $g$ a regular function, in null coordinates. We then use the following lemma.
\begin{Lem}\label{calculF}
Let $G$ be a sufficiently regular $2$-form, $(\alpha, \underline{\alpha}, \rho, \sigma)$ its null components and $g$ a sufficiently regular function. Then,
$$\left| G \left( v, \nabla_v g \right) \right| \hspace{2mm} \lesssim \hspace{2mm}  \left( \tau_- |\rho|+\tau_+|\alpha|+\tau_+\frac{|v^A|}{v^0}|\sigma| +\tau_-\frac{|v^A|}{v^0} |\underline{\alpha}| +\tau_+\frac{v^{\underline{L}}}{v^0} |\underline{\alpha}| \right) \sum_{\widehat{Z} \in \K} \left| \widehat{Z} g \right|  .$$
\end{Lem}
\begin{proof}
Expanding $G(v, \nabla_v g )$ with null components, we obtain
\begin{eqnarray}
\nonumber G(v, \nabla_v g ) & = & 2 \rho\left( v^L \left( \nabla_v g \right)^{\underline{L}}-v^{\underline{L}} \left( \nabla_v g \right)^L \right)+v^B \varepsilon_{BA} \sigma \left( \nabla_v g \right)^A-v^L \alpha_A \left( \nabla_v g \right)^A+v^A \alpha_A \left( \nabla_v g \right)^L \\ 
& & -v^{\underline{L}} \underline{\alpha}_A \left( \nabla_v g \right)^A+v^A \underline{\alpha}_A \left( \nabla_v g \right)^{\underline{L}}. \label{eq:calculF}
\end{eqnarray}
We bound the angular components of $\nabla_v g$ by merely using $v^0 \partial_{v^i} = \widehat{\Omega}_{0i}-t \partial_i-x^i \partial_t$. The radial component\footnote{Note that $\left( \nabla_v g \right)^L=-\left( \nabla_v g \right)^{\underline{L}}=\left( \nabla_v g \right)^r$.} has a better behavior since
\begin{equation}\label{radialcompo}
v^0\left( \nabla_v g \right)^r = \frac{x^i}{r} v^0\partial_{v^i} g=\frac{x^i}{r} \widehat{\Omega}_{0i}g-Sg+(t-r) \underline{L} g.
\end{equation}
\end{proof}
\begin{Rq}
Let us explain how this lemma reflects the null structure of the system. For this, we write $D_1 \prec D_2$ if $D_2$ is expected to behave better than $D_1$. Recall that we have the following hierarchies between the null components of $G$, $v$ and $\nabla_v g$.
\begin{itemize}
\item $\underline{\alpha} \prec \rho \sim \sigma \prec \alpha$,
\item $v^L \prec v^A \prec v^{\underline{L}}$,
\item $\left( \nabla_v g \right)^A \prec \left( \nabla_v g \right)^r$.
\end{itemize}
We can then notice that $\underline{\alpha}$ is hit by $v^{\underline{L}}$ or $v^A \left( \nabla_v g \right)^r$, $\rho$ by $\left( \nabla_v g \right)^r$ and $\sigma$ by $v^A$.
\end{Rq}

\section{Bootstrap assumptions and strategy of the proof}\label{sec6}

Let $N \geq 10$ and $(f^0,F^0)$ be an initial data set satisfying the assumptions of Theorem \ref{theorem}. Then, by a local well-posedness argument, there exists a unique maximal solution $(f,F)$ arising from this data to the system\footnote{We refer to Subsection \ref{subsecsmallvelo} for the reasons which bring us to define $(f,F)$ as a solution to these equations rather than the Vlasov-Maxwell system.}
 \begin{eqnarray}
\nonumber T^{\chi}_F(f) & = & 0, \\ \nonumber
\nabla^{\mu} F_{\mu \nu} & = & J(f)_{\nu}, \\ \nonumber
\nabla^{\mu} {}^* \! F_{\mu \nu} & = & 0.
\end{eqnarray}
Applying Proposition \ref{PropB} and considering possibly $\epsilon_1=C_1 \epsilon$, with $C_1$ a constant depending only on $N$, we can suppose without loss of generality that $\E_N^2[f](0) \leq \epsilon$ and $\mathcal{E}_N[F](0) \leq \epsilon$. Let $T^* >0$ such that $[0,T^*[$ is the maximal domain of $(f,F)$ and $T \in ]0,T^*[$ be the largest time such that\footnote{Notice that such a $T>0$ exists by a standard continuity argument.}, for all $t \in [0,T]$,
\begin{eqnarray}\label{bootf1}
\E^{2}_{N-2}[f](t) & \leq  & 4 \epsilon, \\ 
\E^{1}_N[f](t) & \leq  & 4 \epsilon \log (3+t), \label{bootf2} \\
\sum_{|\beta| = N-1} \left\| r^2 \int_v \frac{v^A}{v^0} \widehat{Z}^{\beta} f dv \right\|_{L^2(\Sigma_t)} & \leq & \sqrt{\epsilon} \log(3+t), \label{bootL2} \\
 \mathcal{E}_{N}[F](t) & \leq & 4 \epsilon \log^4 (3+t), \label{bootF1} \\
 \mathcal{E}^{5}_N[F](t) & \leq & 2 \underline{C} \epsilon , \label{bootF2} 
\end{eqnarray}
where $\underline{C}>0$ is a positive constant which will be specified later. The third bootstrap assumption is here for convenience, we could avoid it but it would complicate the proof. Before presenting our strategy, let us write the immediate consequences of these bootstrap assumptions. Using the Klainerman-Sobolev inequality of Proposition \ref{KS1} and the bootstrap assumption \eqref{bootf1}, one has
\begin{equation}\label{decayf}
\forall \hspace{0.5mm} (t,x) \in [0,T[ \times \R^3, \hspace{3mm} z \in \V, \hspace{3mm} |\beta| \leq N-5, \hspace{15mm} \int_v \left| z^2 \widehat{Z}^{\beta} f \right| dv \lesssim  \frac{\epsilon}{\tau_+^{2} \tau_-}.
\end{equation}
Applying the Klainerman-Sobolev inequality of Proposition \ref{KS2}, Lemma \ref{weights1} and using \eqref{bootf1} and \eqref{bootf2}, we get
\begin{equation}\label{decayf2}
\forall \hspace{0.5mm} t \in [0,T[, \hspace{6mm} \sum_{\begin{subarray}{} |\beta| \leq N-4 \\ \hspace{2mm} z \in \V \end{subarray}}  \left\| \tau_+ \sqrt{\tau_-} \int_v \left| z^2 \widehat{Z}^{\beta} f \right| dv \right\|_{L^2(\Sigma_t)} \lesssim  \epsilon,  \hspace{6mm} \sum_{|\beta| \leq N-2}  \left\| r^2 \int_v \frac{v^A}{v^0}  \widehat{Z}^{\beta} f  dv \right\|_{L^2(\Sigma_t)} \lesssim  \epsilon \log(3+t) .
\end{equation}
By Proposition \ref{decayMaxell}, commutation formula of Proposition \ref{Maxcom}, the bootstrap assumptions \eqref{bootL2}, \eqref{bootF1}, \eqref{bootF2} and the estimate \eqref{decayf2}, we obtain that, for all $(t,x) \in [0,T[ \times \R^3$ and $|\gamma| \leq N-2$,
\begin{eqnarray}\label{decayF1}
\left| \rho \left( \mathcal{L}_{Z^{\gamma}}(F) \right) \right|(t,x) & \lesssim &  \sqrt{\epsilon}\frac{\log^2(3+t)}{\tau_+^{2} \tau_-^{\frac{1}{2}}}, \hspace{20mm} \left| \alpha \left( \mathcal{L}_{Z^{\gamma}}(F) \right) \right|(t,x) \hspace{2mm} \lesssim \hspace{2mm} \sqrt{\epsilon}\frac{\log^2(3+t)}{\tau_+^{\frac{5}{2}} }, \\
\left| \sigma \left( \mathcal{L}_{Z^{\gamma}}(F) \right) \right|(t,x) & \lesssim &  \sqrt{\epsilon}\frac{\log^2(3+t)}{\tau_+^{2} \tau_-^{\frac{1}{2}}} , \hspace{20mm}  \left| \underline{\alpha} \left( \mathcal{L}_{Z^{\gamma}}(F) \right) \right|(t,x) \hspace{2mm} \lesssim \hspace{2mm} \sqrt{\epsilon}\frac{\log^{\frac{5}{2}}(1+\tau_-)}{\tau_+ \tau_-^{\frac{3}{2}}} . \label{decayF2}
\end{eqnarray}
Applying Proposition \ref{velocity}, one obtains that $f$ vanishes for small velocities, i.e.
\begin{equation}\label{eq:velocity}
\forall \hspace{1mm} t \in [0,T[, \hspace{2mm} x \in \R^3, \hspace{2mm} 0 < |v| \leq 1, \hspace{1.5cm} f(t,x,v)=0.
\end{equation}
In view of the support of $\chi$, we then obtain that $T_F(f)=0$ on $[0,T[$, so that $(f,F)$ is the unique classical solution to the Vlasov-Maxwell system \eqref{VM1}-\eqref{VM3} on $[0,T[$. The remainder of the article is devoted to the improvement of the bootstrap assumptions \eqref{bootf1}-\eqref{bootF2}, which will imply Theorem \ref{theorem} as it will prove that $T=T^*$ and then $T^* = + \infty$. The proof is divided in three parts.
\begin{enumerate}
\item First, we improve the bootstrap assumptions \eqref{bootf1} and \eqref{bootf2} by using Proposition \ref{energyf}. To bound the spacetime integrals arising from this energy estimate, we make crucial use of the null structure of the non linearity $\mathcal{L}_{Z^{\gamma}}(F)(v,\nabla_v \widehat{Z}^{\beta} f )$ as well as \eqref{decayF1}, \eqref{decayF2} and \eqref{decayf}.
\item Then, in of view of improving \eqref{bootL2}, \eqref{bootF1} and \eqref{bootF2}, the next step consists in proving $L^2$ estimates on quantities such as $\int_v | z \widehat{Z}^{\beta} f | dv$. To treat the higher order derivatives, we rewrite all transport equations as an inhomogeneous system of Vlasov equations. To handle the homogenous part, we take advantage of the smallness assumption on the $N+3$ derivatives of $f$ at $t=0$, \eqref{decayF1} and \eqref{decayF2}. The inhomogenous part $G$ will be schematically decomposed as a product $K Y$, with $\int_v |Y| dv$ a decaying function and $|K|^2Y$ an integrable function in $(x,v)$.
\item Finally, we improve the bounds on the energy norms of the electromagnetic field through Proposition \ref{energyMax1}. The null structure of the source terms of the Maxwell equations is fundamental for us here.
\end{enumerate}

\section{Improvement of the energy bound on the particle density}\label{sec7}

The purpose of this section is to improve the bootstrap assumptions \eqref{bootf1} and \eqref{bootf2}. Note first that
$$ \forall \hspace{0.5mm} z \in \V, \hspace{1.5mm} q \in \{1,2 \}, \hspace{1.5mm} |\beta| \leq N, \hspace{4mm} T_F(z^q \widehat{Z}^{\beta} f )= T_F(z^q) \widehat{Z}^{\beta} f + z^qT_F( \widehat{Z}^{\beta} f )=q z^{q-1} F(v,\nabla_v z )  \widehat{Z}^{\beta} f + z^q T_F( \widehat{Z}^{\beta} f )$$
and recall from \eqref{eq:velocity} that $T_F(f)=0$ on $[0,T[$. Thus, in view of the energy estimate of Proposition \ref{energyf}, $\E^{2}_{N-2}[f](0) \leq \epsilon$ and commutation formula of Proposition \ref{Maxcom}, $\E^{2}_{N-2}[f] \leq 3 \epsilon $ on $[0,T[$ ensues, if $\epsilon$ is small enough, from the following proposition. 
\begin{Pro}\label{Prop1}
Let $z \in \V$, $|\zeta| \leq N-2$, $|\gamma| \leq N-2$ and $|\xi| \leq N-3$. Then,
\begin{eqnarray}
\nonumber I_{\zeta}^{z,2} \hspace{2mm} & := &  \hspace{2mm} \int_0^t \int_{\Sigma_s} \int_v \left| z F \left( v,\nabla_v z \right)  \widehat{Z}^{\zeta} f  \right| \frac{dv}{v^0} dx ds  \hspace{2mm} \lesssim \hspace{2mm} \epsilon^{\frac{3}{2}}, \\ \nonumber
K^{z,2}_{\gamma, \xi} \hspace{2mm} & := &  \hspace{2mm} \int_0^t \int_{\Sigma_s} \int_v \left| z^2 \mathcal{L}_{Z^{\gamma}}(F) \left( v,\nabla_v \widehat{Z}^{\xi} f \right) \right| \frac{dv}{v^0} dx ds  \hspace{2mm} \lesssim \hspace{2mm}  \epsilon^{\frac{3}{2}}.
\end{eqnarray}
\end{Pro}
Similarly, the following result implies, if $\epsilon$ is small enough, that $\E^1_N[f](t) \leq 3 \epsilon \log (3+t)$ for all $t \in [0,T[$.
\begin{Pro}\label{Pro2}
Let $z \in \V$, $|\zeta| \leq N$, $\gamma$ and $\xi$ such that $|\gamma|+|\xi| \leq N$ and $|\xi| \leq N-1$. We have,
\begin{eqnarray}
\nonumber I^{z,1}_{\zeta} & := &\int_0^t \int_{\Sigma_s} \int_v \left| F \left( v,\nabla_v z \right) \widehat{Z}^{\zeta} f  \right| \frac{dv}{v^0} dx ds \hspace{2mm} \lesssim \hspace{2mm} \epsilon^{\frac{3}{2}} \log (3+t) , \\ \nonumber
K^{z,1}_{\gamma, \xi} & := & \int_0^t \int_{\Sigma_s} \int_v \left| z \mathcal{L}_{Z^{\gamma}}(F) \left( v,\nabla_v \widehat{Z}^{\xi} f \right) \right| \frac{dv}{v^0} dx ds \hspace{2mm} \lesssim \hspace{2mm} \epsilon^{\frac{3}{2}} \log (3+t).
\end{eqnarray}
\end{Pro}
The proofs are based on the analysis, through Lemma \ref{calculF}, of quantities such as $\mathcal{L}_{Z^{\gamma}}(F) \left( v, \nabla_v \widehat{Z}^{\beta} f \right)$. We then prove the following preparatory lemma.
\begin{Lem}\label{prepalem}
Let $|\gamma| \leq N-2$ and $h : [0,T[ \times \R^3_x \times (\R^3_v \setminus \{ 0\} )$ be a sufficiently regular function. Then,
$$\left| \mathcal{L}_{Z^{\gamma}}(F) \left( v, \nabla_v h \right) \right| \hspace{2mm} \lesssim \hspace{2mm}   \left( \frac{\sqrt{\epsilon}}{\tau_+^{\frac{5}{4}}}+\frac{ \sqrt{\epsilon} v^{\underline{L}}}{ \tau_-^{\frac{5}{4}}v^0} \right)  \sum_{\widehat{Z} \in \K} \left| \widehat{Z} h \right|  \hspace{9.5mm} \text{and} \hspace{9.5mm}  \left| F \left( v, \nabla_v z \right) \right| \hspace{2mm} \lesssim \hspace{2mm} \left( \frac{\sqrt{\epsilon}}{\tau_+^{\frac{5}{4}}}+\frac{ \sqrt{\epsilon} v^{\underline{L}}}{ \tau_-^{\frac{5}{4}}v^0} \right) \sum_{w \in \V} |w|. $$
\end{Lem}
\begin{proof}
Let $(\alpha, \underline{\alpha}, \rho, \sigma)$ be the null decomposition of $\mathcal{L}_{Z^{\gamma}}(F)$. Using Lemma \ref{calculF}, we have
$$ \left| \mathcal{L}_{Z^{\gamma}}(F) \left( v, \nabla_v h \right) \right| \hspace{1mm} \lesssim \hspace{1mm} \sum_{\widehat{Z} \in \K} \left( \tau_- \left| \rho \right|+\tau_+\left| \alpha  \right| + \tau_+ \frac{|v^A|}{v^0} \left| \sigma  \right|+ \tau_- \frac{|v^A|}{v^0} \left| \underline{\alpha} \right|+ \tau_+ \frac{v^{\underline{L}}}{v^0} \left| \underline{\alpha} \right| \right)      \left| \widehat{Z} h \right|.$$
According to the pointwise estimates \eqref{decayF1}, \eqref{decayF2} and the inequality $|v^A| \lesssim \sqrt{v^0 v^{\underline{L}}}$ (see Lemma \ref{weights1}), one has
\begin{equation}\label{eq:one} \tau_- \left| \rho \right|+\tau_+\left| \alpha  \right|  \lesssim  \frac{\sqrt{\epsilon}}{\tau_+^{\frac{5}{4}}}, \hspace{10mm} \tau_+\frac{v^{\underline{L}}}{v^0}|\underline{\alpha}| \lesssim \frac{\sqrt{\epsilon}v^{\underline{L}}}{\tau_-^{\frac{5}{4}}v^0}, \hspace{10mm} \frac{|v^A|}{v^0} (\tau_+ |\sigma|+\tau_- |\underline{\alpha} |) \lesssim \frac{\sqrt{\epsilon} \sqrt{ v^{\underline{L}}}}{\tau_+^{\frac{3}{4}} \tau_-^{\frac{1}{2}}\sqrt{v^0}} \lesssim \frac{\sqrt{\epsilon}}{\tau_+^{\frac{5}{4}} }+\frac{ \sqrt{\epsilon} v^{\underline{L}}}{ \tau_-^{\frac{5}{4}}v^0},
\end{equation}
which implies the first inequality. The second one follows directly since, by Lemma \ref{weights}, $\sum_{\widehat{Z} \in \K} | \widehat{Z}( z) | \lesssim \sum_{w \in \V} |w|$. 
\end{proof}
The remainder of the section is devoted to the proof of Propositions \ref{Prop1} and \ref{Pro2}.
\subsection{Proof of Proposition \ref{Prop1}}\label{secPro1}

Let $z \in \V$, $|\zeta| \leq N-2$, $|\gamma| \leq N-2$ and $|\xi| \leq N-3$. Using successively Lemma \ref{prepalem}, that $1 \leq v^0$ on the support of $f$ (see \eqref{eq:velocity}), that $[0,t] \times \R^3$ can be foliated by $(C_u(t))_{u \leq t}$ (see Lemma \ref{foliationexpli}) and $\E^2_{N-2}[f] \leq 4 \epsilon$, which comes from the bootstrap assumption \eqref{bootf1}, we have,
\begin{eqnarray}
\nonumber I^{z,2}_{\zeta}+K^{z,2}_{\gamma, \xi} \hspace{-0.5mm} & \lesssim & \sum_{|\beta| \leq N-2} \sum_{w \in \V} \int_0^t \int_{\Sigma_s} \int_v \frac{\sqrt{\epsilon}}{\tau_+^{\frac{5}{4}}} \left| w^2 \widehat{Z}^{\beta} f  \right| \frac{dv}{v^0} dx ds+\int_0^t \int_{\Sigma_s} \int_v \frac{\sqrt{\epsilon}}{\tau_-^{\frac{5}{4}}} \frac{v^{\underline{L}}}{v^0}\left| w^2 \widehat{Z}^{\beta} f  \right| \frac{dv}{v^0} dx ds \\ \nonumber
& \lesssim & \sum_{\begin{subarray}{} |\beta| \leq N-2 \\ \hspace{1mm} w \in \V \end{subarray}} \int_0^t \frac{\sqrt{\epsilon}}{(1+s)^{\frac{5}{4}}}  \int_{\Sigma_s} \int_v \left| w^2 \widehat{Z}^{\beta} f  \right| dv dx ds+\int_{u=-\infty}^t \int_{C_u(t)} \int_v \frac{\sqrt{\epsilon}}{\tau_-^{\frac{5}{4}}} \frac{v^{\underline{L}}}{v^0}\left| w^2 \widehat{Z}^{\beta} f  \right| dv dC_u(t) du \\ \nonumber 
& \lesssim & \sqrt{\epsilon} \sum_{|\beta| \leq N-2} \sum_{w \in \V} \int_0^t \frac{\E [ w^2 \widehat{Z}^{\beta} f ](s)}{(1+s)^{\frac{5}{4}}} ds+\int_{u=-\infty}^t \frac{1}{\tau_-^{\frac{5}{4}}} \int_{C_u(t)} \int_v \frac{v^{\underline{L}}}{v^0}\left| w^2 \widehat{Z}^{\beta} f  \right| dv dC_u(t) du \\ \nonumber
& \lesssim & \sqrt{\epsilon}  \int_0^t \frac{\E^2_{N-2}[f](s)}{(1+s)^{\frac{5}{4}}} ds+\sqrt{\epsilon}\int_{u=-\infty}^t \frac{1}{\tau_-^{\frac{5}{4}}} \E^2_{N-2}[f](t) du \hspace{2mm} \lesssim \hspace{2mm} \epsilon^{\frac{3}{2}} \int_0^{+ \infty} \frac{ds}{(1+s)^{\frac{5}{4}}}+\epsilon^{\frac{3}{2}}\int_{u=-\infty}^{+\infty} \frac{du}{\tau_-^{\frac{5}{4}}}. 
\end{eqnarray}
We then deduce that $I^{z,2}_{\zeta}+K^{z,2}_{\gamma, \xi} \lesssim \epsilon^{\frac{3}{2}}$, which concludes the proof of Proposition \ref{Prop1}.
\subsection{Proof of Proposition \ref{Pro2}}\label{secpointF}

We fix $z \in \V$, $\zeta$, $\gamma$ and $\xi$ satisfying $|\zeta| \leq N$, $|\gamma|+|\xi| \leq N$ and $|\xi| \leq N-1$. Suppose first that $|\gamma| \leq N-2$. Hence, following the computations of Subsection \ref{secPro1} and using that $\E^1_N[f](t) \leq 4 \log (3+t)$ by the bootstrap assumption \eqref{bootf2}, we get
\begin{eqnarray}
\nonumber I^{z,1}_{\zeta}+K^{z,1}_{\gamma, \xi}  & \lesssim & \sqrt{\epsilon} \sum_{|\beta| \leq N} \sum_{w \in \V} \int_0^t \frac{\E [ w \widehat{Z}^{\beta} f ](s)}{(1+s)^{\frac{5}{4}}} ds+\int_{u=-\infty}^t \frac{1}{\tau_-^{\frac{5}{4}}} \int_{C_u(t)} \int_v \frac{v^{\underline{L}}}{v^0}\left| w \widehat{Z}^{\beta} f  \right| dv dC_u(t) du \\ \nonumber
& \lesssim & \sqrt{\epsilon}  \int_0^t \frac{\E^1_{N}[f](s)}{(1+s)^{\frac{5}{4}}} ds+\sqrt{\epsilon}\int_{u=-\infty}^t \frac{1}{\tau_-^{\frac{5}{4}}} \E^1_{N}[f](t) du \\ \nonumber
& \lesssim &  \epsilon^{\frac{3}{2}} \int_0^{+ \infty} \frac{\log(3+s)}{(1+s)^{\frac{5}{4}}} ds+\epsilon^{\frac{3}{2}}\log (3+t) \int_{u=-\infty}^{+\infty} \frac{du}{\tau_-^{\frac{5}{4}}} \hspace{2mm} \lesssim \hspace{2mm} \epsilon^{\frac{3}{2}} \log (3+t). 
\end{eqnarray}
We now consider the cases where $|\gamma| \geq N-1$, so that $|\xi| \leq 1$. Let us denote the null decomposition of $\mathcal{L}_{Z^{\gamma}}(F)$ by $(\alpha, \underline{\alpha}, \rho , \sigma)$. Using Lemma \ref{calculF} and that $1 \leq v^0$ on the support of $f$, we are led to bound, for all $|\beta|=|\kappa|+1 \leq 2$, the following integrals,
\begin{eqnarray}
\nonumber I_F & := & \int_0^t \int_{\Sigma_s}\int_v \left( \tau_- \left| \rho \right|+\tau_+\left| \alpha  \right|+\tau_+|\sigma|  \frac{|v^{A}|}{v^0} \right) \left| z \widehat{Z}^{\beta} f \right|  dv dx ds,  \\ \nonumber
I_{\underline{\alpha}} & := & \int_0^t \int_{\Sigma_s} \int_v  \left| \underline{\alpha} \right|   \frac{\tau_-|v^A| + \tau_+v^{\underline{L}}}{v^0}\left| z \widehat{Z}^{\beta} f \right|  dv dx ds.
\end{eqnarray}
Using $\tau_+ v^{\underline{L}}+\tau_+|v^A| \lesssim v^0  \sum_{w \in \V}|w|$ (which comes from Lemmas \ref{weights1}), the Cauchy-Schwarz inequality, the bootstrap assumption \eqref{bootF2} and the estimate \eqref{decayf2}, we get
$$ I_{\underline{\alpha}}  \hspace{2mm} \lesssim \hspace{2mm} \sum_{w \in \V}  \int_{0}^{t} \left\| |\underline{\alpha}| \right\|_{L^2(\Sigma_s)} \left\|  \int_v  \left| (w^2+z^2) \widehat{Z}^{\beta} f \right| dv \right\|_{L^2(\Sigma_s)} ds \hspace{2mm} \lesssim \hspace{2mm} \int_{0}^{t} \sqrt{\mathcal{E}^{5}_N[F](s)} \frac{ds}{1+s}  \hspace{2mm} \lesssim \hspace{2mm} \epsilon^{\frac{3}{2}} \log(3+t).$$
For $I_F$, in order to apply Lemma \ref{foliationexpli}, notice first that we have by the estimate \eqref{decayf}, for $u \leq t$ and $i \in \mathbb{N}$, 
$$\left\| \int_v \left| w^2 \widehat{Z}^{\beta} f \right| dv \right\|^2_{L^2(C^i_u(t))} \lesssim  \int_{C_u^i(t)} \frac{\epsilon^2}{\tau_+^4 \tau_-^2} d C_u^i(t) \lesssim \frac{\epsilon^2}{\tau_-^{\frac{9}{4}}(1+t_i)^{\frac{1}{4}}} \int_{\underline{u}=2t_i-u}^{2t_{i+1}-u} \frac{r^2}{\tau_+^{\frac{7}{2}}} d \underline{u} \lesssim \frac{\epsilon^2}{\tau_-^{\frac{9}{4}}(1+2^i)^{\frac{1}{4}}} .$$
Hence, using $\tau_+|v^A| \lesssim v^0 \sum_{w \in \V} |w|$, the Cauchy-Schwarz inequality and the bootstrap assumption \eqref{bootF1}, we obtain
\begin{eqnarray}
\nonumber  I_{F} & = & \sum_{w \in \V} \sum_{i=0}^{+\infty} \int_{u=-\infty}^t  \int_{C^i_u(t)} \left( \tau_- \left| \rho \right|+\tau_+\left| \alpha  \right|+|\sigma| \right) \int_v \left| w^2 \widehat{Z}^{\beta} f \right| dv dC_u(t)^i du \\ \nonumber
& \lesssim & \sum_{w \in \V} \sum_{i=0}^{+\infty} \int_{u=-\infty}^t  \left\| \tau_- \left| \rho \right|+\tau_+\left| \alpha  \right|+|\sigma| \right\|_{L^2(C_u^i(t))} \left\| \int_v \left| w^2 \widehat{Z}^{\beta} f \right| dv \right\|_{L^2(C^i_u(t))} du \\ \nonumber
& \lesssim &  \sum_{i=0}^{+\infty} \int_{u=-\infty}^t  \sqrt{\mathcal{E}_N[F]( T_{i+1}(t))} \frac{\epsilon}{\tau_-^{\frac{9}{8}}(1+2^i)^{\frac{1}{8}}}  du \hspace{2mm} \lesssim \hspace{2mm} \epsilon^{\frac{3}{2}} \sum_{i=0}^{+\infty} \frac{\log^2(3+2^{i+1})}{(1+2^i)^{\frac{1}{8}}} \int_{u=-\infty}^{+\infty} \frac{du}{\tau_-^{\frac{9}{8}}} \hspace{2mm} \lesssim \hspace{2mm} \epsilon^{\frac{3}{2}}.
\end{eqnarray}
This concludes the proof and the improvement of the bootstrap assumptions \eqref{bootf1} and \eqref{bootf2}.
\section{$L^2$ estimates on the velocity averages of the Vlasov field}\label{sec8}

In view of the energy estimate of Proposition \ref{energyMax1}, we have to prove $L^2_x$ estimates on quantities such as $\int_v | z \widehat{Z}^{\beta} f | dv$, for $|\beta| \leq N$. If $|\beta| \leq N-2$, we can use a Klainerman-Sobolev inequality to obtain a sufficient decay rate (see Proposition \ref{L2esti} below). The main part of this section then consists in deriving such estimates for $|\beta| \geq N-1$. For this purpose, we follow the strategy used in \cite{FJS} (Section $4.5.7$) and adapted in \cite{dim4} for the Vlasov-Maxwell system. Contrary to \cite{dim4}, we will have to keep more of the null structure of the system. This will force us to add a new hierarchy on the functions studied here. Let us first rewrite the system and then we will explain how we will proceed. Let $I_1$ and $I_2$ be the following ordered sets,
\begin{flalign*}
& \hspace{0.5cm} I_1 := \{ \beta \hspace{2mm} \text{multi-index} \hspace{1mm} / \hspace{1mm} N-5 \leq |\beta| \leq N \} = \{ \beta_{1,1},...,\beta_{1,|I_1|} \}, & \\
& \hspace{0.5cm} I_2 := \{ \beta \hspace{2mm} \text{multi-index} \hspace{1mm} / \hspace{1mm}  |\beta| \leq N-5 \} =\{ \beta_{2,1},...,\beta_{2,|I_2|} \} .& 
\end{flalign*}

\begin{Rq}
Contrary to \cite{dim4}, we have $I_1 \cap I_2 \neq \varnothing$.
\end{Rq}
We also consider, for $N-5 \leq k \leq N$, $I^k_1 := \{ \beta \in I_1 \hspace{1mm} / \hspace{1mm} |\beta| =k \}$, and two vector valued fields $R$ and $W$ of respective length $|I_1|$ and $|I_2|$ such that
$$ R_i= \widehat{Z}^{\beta_{1,i}}f \hspace{10mm} \text{and} \hspace{10mm} W_i = \widehat{Z}^{\beta_{2,i}}f.$$
We will sometimes abusively write $i \in I^k_1$ instead of $\beta_{1,i} \in I^k_1$. Let us denote by $\Vv$ the module over the ring $C^0 \left( [0,T[ \times \R^3_x \times \left( \R^3_v \setminus \{0 \} \right) \right)$ generated by $( \partial_{v^l})_{1 \leq l \leq 3}$. We now rewrite the Vlasov equations satisfied by $R$ and $W$.
\begin{Lem}\label{L2bilan}
There exists three matrix-valued functions $A : [0,T[ \times \R^3 \times \left( \R^3_v \setminus \{0 \} \right) \rightarrow \mathfrak M_{|I_1|}(\Vv)$, $D : [0,T[ \times \R^3 \times \left( \R^3_v \setminus \{0 \} \right) \rightarrow  \mathfrak M_{|I_2|}(\Vv)$ and $B : [0,T[ \times \R^3 \times \left( \R^3_v \setminus \{0 \} \right) \rightarrow  \mathfrak M_{|I_1|,|I_2|}(\Vv)$ such that
$$T_F(R)+AR=B W  \hspace{10mm} \text{and} \hspace{10mm} T_F(W)=DW.$$
Moreover, if $1 \leq i \leq |I_1|$, $A$ and $B$ are such that $T_F(R_i)$ is a linear combination of 
\begin{flalign*}
& \hspace{1.2cm} \mathcal{L}_{Z^{\gamma}}(F) \left( v , \nabla_v R_j \right), \hspace{8.5mm} \text{with} \hspace{10mm} |\beta_{1,j}| < |\beta_{1,i}| \hspace{9.5mm} \text{and} \hspace{8mm} |\gamma|  +  |\beta_{1,j}| \leq |\beta_{1,i}|, & \\
& \hspace{1.2cm} \mathcal{L}_{Z^{\xi}}(F) \left( v , \nabla_v W_j \right), \hspace{8mm} \text{with} \hspace{10mm} |\beta_{2,j}| \leq N-6 \hspace{8mm} \text{and} \hspace{8mm} |\xi| \leq N. &
\end{flalign*}
Similarly, if $1 \leq i \leq I_2$, $D$ is such that $T_F(W_i)$ is a linear combination of
\begin{flalign*}
& \hspace{1.2cm} \mathcal{L}_{Z^{\gamma}}(F) \left( v , \nabla_v W_j \right), \hspace{8mm} \text{with} \hspace{10mm} |\beta_{2,j}| \leq N-6 \hspace{8mm} \text{and} \hspace{8mm} |\gamma| \leq N-5. & 
\end{flalign*}
Note also, using \eqref{decayf}, that
\begin{flalign*}
& \hspace{1.2cm}\forall \hspace{0.5mm} (t,x) \in [0,T[ \times \R^3, \hspace{3mm} z \in \V, \hspace{3mm} 1 \leq q \leq |I_2|, \hspace{12mm} \int_v |z^2W_q| dv \lesssim \frac{\epsilon}{\tau_+^{2}\tau_-}.&
\end{flalign*}
\end{Lem}
\begin{Rq}\label{Rqbilan}
Notice that if $\beta_{1,i} \in I^{N-5}_1$, then $A_i^q=0$ for all $1 \leq q \leq |I_1|$. Note also that if $p \geq 1$ and $\beta_{1,i} \in I^{N-5+p}$, we have $|\gamma| \leq p$. 
\end{Rq}
\begin{proof}
One only has to apply the commutation formula of Proposition \ref{Maxcom} to $\widehat{Z}^{\beta_{1,i}} f$ or $\widehat{Z}^{\beta_{2,i}} f$ and to replace each quantity such as $\widehat{Z}^{\kappa} f$, for $|\kappa| \neq N-5$, by the corresponding component of $R$ or $W$. If $|\kappa| = N-5$, we replace it by the corresponding component of $R$.
\end{proof}
The goal is to obtain an $L^2$-estimate on $R$. For this, let us split it in $R:=H+G$, where
$$\left\{
    \begin{array}{ll}
         T^{\chi}_F(H)+ AH=0 \hspace{2mm}, \hspace{2mm} H(0,.,.)=R(0,.,.),\\
        T^{\chi}_F(G)+ AG= BW \hspace{2mm}, \hspace{2mm} G(0,.,.)=0
    \end{array}
\right.$$
and then prove $L^2$ estimates on the velocity averages of $H$ and $G$. To do it, we will schematically establish that $G=KW$, with $K$ a matrix such that $\E[KKW]$ do not growth too fast, and then use the pointwise decay estimates on $\int_v |z^2W|dv$ to obtain the expected decay rate on $\| \int_v |G| dv \|_{L^2_x}$. For $\| \int_v |H| dv \|_{L^2_x}$, we will make crucial use of Klainerman-Sobolev inequalities so that we will need to commute the transport equation satisfied by $H$ and prove $L^1$-bounds such as we did in the proof of Proposition \ref{Prop1}. Contrary to what we did in \cite{dim4}, we keep the $v$ derivatives in order to take advantage of the good behavior of radial component of $\nabla_v g$. This is why we put the derivatives of order $N-5$ in both $R$ and $W$.
\begin{Rq}
If we proceed as in \cite{dim4}, we would not be able to use the estimate $\left( \nabla_v g \right)^r \sim \tau_- \widehat{Z}g$ and an analogous result to Lemma \ref{calculF} would give the term $\tau_+|\underline{\alpha}| \frac{|v^A|}{v^0}|\widehat{Z} g|$. In our case (the three dimensional one), a lack of decay in the $t+r$ direction prevents us to deal with it.
\end{Rq}
\subsection{The homogeneous system}\label{subsecH}

In order to obtain $L^{\infty}$, and then $L^2$, estimates on $\int_v |H| dv$, we will have to commute at least three times the transport equation satisfied by each component of $H$. However, if $\beta_{1,i} \in I^k_1$, with $k \geq N-4$, we need to control the $L^1$ norm of $\widehat{Z}^{\kappa} H_j$, with $|\kappa|=4$ and $j \in I^{k-1}_1$, to bound $\| \widehat{Z}^{\xi} H_i \|_{L^1_{x,v}}$, with $|\xi| = 3$. We then consider the following energy norm
$$ \E_H \hspace{2mm} := \hspace{2mm} \sum_{z \in \V} \hspace{1mm} \sum_{k=0}^5 \hspace{1mm} \sum_{ |\beta| \leq 3+k} \hspace{1mm} \sum_{i \in I^{N-k}_1} \hspace{1mm} \E[z^2\widehat{Z}^{\beta} H_i].$$
We have the following commutation formula.
\begin{Lem}\label{comL2hom}
Let $0 \leq k \leq 5$, $|\beta| \leq 3+k$ and $i \in I^{N-k}_1$. Then, if $H$ vanishes for all $|v| \leq 1$, $T_F(\widehat{Z}^{\beta} H_i)$ can be written as a linear combination of terms of the form
$$ \mathcal{L}_{Z^{\gamma}}(F) \left( v, \nabla_v \widehat{Z}^{\xi} H_j \right) , \hspace{5mm} \text{with} \hspace{5mm} |\gamma| \leq 8 \leq N-2, \hspace{5mm} |\xi| \leq |\beta|, \hspace{5mm} |\beta_{1,j}| \leq |\beta_{1,i}|, \hspace{5mm} |\xi|+|\beta_{1,j}| < |\beta|+|\beta_{1,i}|.$$
\end{Lem}
\begin{proof}
If $H$ vanishes for all $|v| \leq 1$, we have $T_F(H)+AH=0$. Hence, according to Proposition \ref{Maxcom}, the source terms which arise from the commutator $[T_F,\widehat{Z}^{\beta}]$ are such as those described in this lemma, with $j=i$. The other ones come from $\widehat{Z}^{\beta} \left( T_F \left(H_i \right) \right)$ (use Lemma \ref{L2bilan}, Remark \ref{Rqbilan} and Lemma \ref{comumax1} to check that they are of the researched form).
\end{proof}
As $H(0,.,.)=R(0,.,.)$, it then follows that $H(0,.v)=0$ for all $|v| \leq 3$ and, applying Proposition\footnote{Strictly speaking, one cannot simply apply Proposition \ref{ProB} since $T_F(H_i) \neq 0$ if $ i \in I^{k}_1$ and $k \geq N-4$. However, in view of Proposition \ref{comL2hom}, one can easily adapt it to our context.} \ref{ProB}, that there exists $C_H>0$ such that $\E_H(0) \leq C_H \epsilon$. Consequently, using Corollary \ref{velocity2} and following the proof of Proposition \ref{Prop1}, one can prove that, for $\epsilon$ small enough, 
$$ \forall \hspace{0.5mm} t \in [0,T[, \hspace{8mm} \E_H (t) \leq 3C_H \epsilon \hspace{12mm} \text{and} \hspace{12mm} \forall \hspace{0.5mm} (t,x) \in [0,T[ \times \R^3, \hspace{1mm} 0 < |v| \leq 1, \hspace{8mm} H(t,x,v)=0.$$
By Proposition \ref{KS1}, we then obtain, for $0 \leq k \leq 5$,
\begin{equation}\label{decayH}
 \forall \hspace{0.5mm} (t,x) \in [0,T[ \times \R^3, \hspace{3mm} z \in \V, \hspace{3mm} 1 \leq j \leq |I^{N-k}_1|, \hspace{3mm} |\beta| \leq k,  \hspace{15mm} \int_v |z^2\widehat{Z}^{\beta} H_j| dv \lesssim  \frac{\epsilon}{\tau_+^2 \tau_-}.
 \end{equation}
\begin{Rq}\label{rqFhighderiv}
Proceeding as in Subsection $17.2$ of \cite{FJS2}, we could avoid any hypothesis on the higher order derivatives of $F^0$.
\end{Rq}

\subsection{The inhomogenous system}\label{subsecG}

Start by noticing that G vanishes for all $|v| \leq 1$ since $G=R-H$. We then deduce from $\chi(|v|) =1$ for all $|v| \geq 1$ that $G$ satisfies $T_F(G)+AG=BW$. To derive an $L^2$ estimate on $G$, we cannot commute the transport equation because $B$ contains top order derivatives of the electromagnetic field. Instead, we follow the methodology of \cite{FJS} (see Subsection $4.5.7$). We kept the $v$ derivatives of $G$ in the matrix $A$ so that we could use the null structure in a better way. In order to obtain $L^1$-bounds on quantities introduced below, we now need to rewrite these $v$ derivatives. This is the purpose of the following lemma.

\begin{Lem}\label{bilaninho}
There exists $p \geq 1$, a vector valued field $Y$ of length $p$, which vanishes for $|v| \leq 1$, and three matrix-valued functions $\overline{A} : [0,T[ \times \R^3 \times \left( \R^3_v \setminus \{0 \} \right) \rightarrow \mathfrak M_{|I_1|}(\R)$, $\overline{B} : [0,T[ \times \R^3 \times \left( \R^3_v \setminus \{0 \} \right) \rightarrow \mathfrak M_{|I_1|,p}(\R)$, $\overline{D} : [0,T[ \times \R^3 \times \left( \R^3_v \setminus \{0 \} \right) \rightarrow \mathfrak M_{p}(\R)$ such that
$$T_F(G)+\overline{A}G= \overline{B} Y, \hspace{10mm} T_F(Y)= \overline{D} Y \hspace{10mm} \text{and} \hspace{10mm} \sum_{z \in \V} \int_v  |z^2Y| dv \lesssim  \frac{\epsilon}{\tau_+^2 \tau_-}.$$
Moreover, $\overline{A}$ and $\overline{B}$ are such that, if $i \in \llbracket 1, |I_1| \rrbracket$, $T_F(G_i)$ can be bounded by a linear combination of terms of the form,
\begin{eqnarray}
\nonumber & & \left( \tau_-\left| \rho \left( \mathcal{L}_{Z^{\gamma}}(F) \right) \right|+\tau_+\left| \alpha \left( \mathcal{L}_{Z^{\gamma}}(F) \right) \right| +\tau_+ \frac{|v^A|}{v^0} \left| \sigma \left( \mathcal{L}_{Z^{\gamma}}(F) \right) \right|+ \frac{\tau_- |v^A|+ \tau_+ v^{\underline{L}}}{v^0} \left| \underline{\alpha} \left( \mathcal{L}_{Z^{\gamma}}(F) \right) \right| \right) |G_j| \hspace{5mm} \text{and} \\ \nonumber
& & \left( \tau_-\left| \rho \left( \mathcal{L}_{Z^{\xi}}(F) \right) \right|+\tau_+\left| \alpha \left( \mathcal{L}_{Z^{\xi}}(F) \right) \right|+\left| \sigma \left( \mathcal{L}_{Z^{\xi}}(F) \right) \right|+\left| \underline{\alpha} \left( \mathcal{L}_{Z^{\xi}}(F) \right) \right| \right) |z Y_q|,
\end{eqnarray}
where $j \in \llbracket 1, |I_1| \rrbracket$, $|\gamma| \leq 5$, $q \in \llbracket 1, p \rrbracket$, $|\xi| \leq N$ and $z \in V$. Similarly, $\overline{D}$ is such that, if $i \in \llbracket 1, p \rrbracket$, $T_F(Y_i)$ can be bounded by a linear combination of terms of the form,
$$\left( \tau_-\left| \rho \left( \mathcal{L}_{Z^{\gamma}}(F) \right) \right|+\tau_+\left| \alpha \left( \mathcal{L}_{Z^{\gamma}}(F) \right) \right| +\tau_+ \frac{|v^A|}{v^0} \left| \sigma \left( \mathcal{L}_{Z^{\gamma}}(F) \right) \right|+ \frac{\tau_- |v^A|+ \tau_+ v^{\underline{L}}}{v^0} \left| \underline{\alpha} \left( \mathcal{L}_{Z^{\gamma}}(F) \right) \right| \right) |Y_j|,$$
where $j \in \llbracket 1, p \rrbracket$ and $|\gamma| \leq N-5$.
\end{Lem}
\begin{proof}
The strategy of the proof is the following. If $\partial_{v^k} G_j$ appears in $T_F(G)+AG=BW$, then, by Lemma \ref{L2bilan}, $j \in I^k_1$, with $N-5 \leq k \leq N-1$. We then transform it with $v^0 \partial_{v^k} = \widehat{\Omega}_{0k}-x^k \partial_t-t \partial_k$ and express it, with controllable error terms, as a combination of $(G_l)_{l \in I^{k+1}_1}$. The other manipulations are similar to those made in Section \ref{sec7} when we applied Lemma \ref{calculF}. Let us denote, for $j \in I_1 \setminus I_1^N$ and $\widehat{Z} \in \K$, by $j_{\widehat{Z}}$ the index such that $R_{j_{\widehat{Z}}}= \widehat{Z} \widehat{Z}^{\beta_{1,j}} f = \widehat{Z} R_j $. Hence, by \eqref{decayH} and since $R=H+G$, we have, for all $ j \in I_1 \setminus I^N_1$,
\begin{equation}\label{error}
\forall \hspace{0.5mm} (t,x) \in [0,T[ \times \R^3, \hspace{2mm} (z, \widehat{Z}) \in \V \times \K, \hspace{6mm} \int_v |z|^2|G_{j_{\widehat{Z}}}-\widehat{Z} G_j|dv = \int_v |z|^2|H_{j_{\widehat{Z}}}-\widehat{Z} H_j|dv \lesssim  \frac{\epsilon}{\tau_+^2 \tau_-}.
\end{equation}
Let $p^0 := |I_2|+|I_1 \setminus I^N_1|$ and $Y^0$ a vector valued field\footnote{$Y^0$ will be a subvector of the vector $Y$ of the lemma.} of length $p^0$ containing each component of $W$ and each $G_{j_{\widehat{Z}}}-\widehat{Z} G_j$, for $j \in I_1 \setminus I^N_1$. We order the components of $Y^0$ such as $Y^0_{j_{\widehat{Z}}}=G_{j_{\widehat{Z}}}-\widehat{Z} G_j$. In view of \eqref{error} and Lemma \ref{L2bilan}, $\int_v  |z^2Y^0| dv$ satisfies the desired pointwise decay estimate on $\int_v  |z^2Y| dv$. We now fix $i \in I_1$. Applying Lemma \ref{L2bilan}, one can see that $T_F(G_i)$ can be written as a linear combination of the following terms.
\begin{itemize}
\item Those coming from $BW$,
$$ \mathcal{L}_{Z^{\xi}}(F) \left( v , \nabla_v W_j \right), \hspace{5mm} \text{with} \hspace{5mm} |\beta_{2,j}| \leq N-6 \hspace{5mm} \text{and} \hspace{5mm} |\xi| \leq N,$$
leading, by Lemma \ref{calculF} and $\tau_+v^{\underline{L}}+\tau_+ |v^A| \lesssim v^0 \sum_{w \in \V} |w|$ (see Lemma \ref{weights1}), to the announced terms involving $Y$.
\item Those coming from $AW$, 
$$ \mathcal{L}_{Z^{\gamma}}(F) \left( v , \nabla_v G_j \right), \hspace{5mm} \text{with} \hspace{5mm} |\beta_{1,j}| < |\beta_{1,i}|  \hspace{5mm} \text{and} \hspace{5mm} |\gamma|+|\beta_{1,j}| \leq |\beta_{1,i}|.$$
Then, expand $\mathcal{L}_{Z^{\gamma}}(F) \left( v , \nabla_v G_j \right)$ in null components using formula \eqref{eq:calculF}. We now rewrite the angular components of $\nabla_v G_j$ using $v^0 \partial_{v^k}= \widehat{\Omega}_{0k}-x^k \partial_t-t \partial_k$, so that
$$ v^0\partial_{v^k} G_j = G_{j_{\widehat{\Omega}_{0k}}}-x^k G_{j_{\partial_t}}-t G_{j_{\partial_k}}-Y^0_{j_{\widehat{\Omega}_{0k}}} +x^k Y^0_{j_{\partial_t}}+t Y^0_{j_{\partial_k}}.$$
For the radial component, use \eqref{radialcompo} to obtain
$$v^0\left( \nabla_v G_j \right)^r = \frac{x^q}{r} \left( G_{j_{\widehat{\Omega}_{0q}}}-Y^0_{j_{\widehat{\Omega}_{0q}}} \right)-G_{j_S}+Y^0_{j_S}+(t-r) \left( G_{j_{\partial_t}}-Y^0_{j_{\partial_t}}-\frac{x^q}{r} G_{j_{\partial_q}}+\frac{x^q}{r} Y^0_{j_{\partial_q}} \right).$$
\end{itemize}
This concludes the construction of $\overline{A}$, $\overline{B}$. To obtain an equation for $Y^0$, we will see that we need to consider a bigger vector than $Y^0$. Let $i \in \llbracket 1 , p^0 \rrbracket$. If $Y^0_i=W_q$, with $q \in I_2$, we can build the line $i$ of $\overline{D}$ using Lemmas \ref{L2bilan} and \ref{calculF}. Otherwise, $Y^0_i = \widehat{Z} H_j-H_{j_{\widehat{Z}}}$ and by Lemma \ref{comL2hom} we see that functions such as $\partial_v \widehat{Z} H_r$, with $|\beta_{1,r}| < |\beta_{1,j}|$, appear in certain source terms of $T_F(Y^0_i)$. We then consider the vector valued field $Y$ containing $Y^0$ and all the quantities $\widehat{Z}^{\kappa} H_j$ such as $\beta_{1,j} \in I^{N-5+k}_1$ and $|\kappa|+k \leq 5$. It remains to use \eqref{decayH} and Lemmas \ref{comL2hom}, \ref{calculF}.
\end{proof}
Consider now $K$ satisfying $T^{\chi}_{ F}(K)+\chi \overline{A}K+\chi K \overline{D}= \chi\overline{B}$ and $K(0,.,.)=0$. Hence, $KY=G$ since they both initially vanish and $T_F(KY)+\overline{A}KY=\overline{B}Y$ in view of the velocity support of $Y$. The goal now is to control the energy
$$ \E_G := \sum_{i=0}^{|I_1|} \sum_{j=0}^p \sum_{q=0}^p  \E \left[ \left| K_i^j \right|^2 Y_q \right].$$
We will then be naturally led to use that
\begin{equation}\label{sourceinho}
T_F\left( |K^j_i|^2 Y_q\right) = |K^j_i |^2 \overline{D}^r_q Y_r-2\left(\overline{A}^r_i K^j_r +K^r_i \overline{D}^j_r \right) K^j_i Y_q+2 \overline{B}^j_iK^j_iY_q.
\end{equation}
\begin{Pro}\label{inhoL1}
If $\epsilon$ is small enough, we have $\E_G(t) \lesssim \epsilon \log^2 (3+t)$ for all $t \in [0,T[$.
\end{Pro}
\begin{proof}
Let $T_0 \in [0,T[$ the largest time such that $\E_G(t) \lesssim \epsilon \log^2 (3+t)$ for all $t \in [0,T_0[$. By continuity, $T_0 >0$. The remaining of the proof consists in improving this bootstrap assumption, which would imply the result. The computations will be similar as those of the proof of Proposition \ref{Pro2}. Let $i \in \llbracket 1, |I_1| \rrbracket$ and $(j,q) \in \llbracket 1, p \rrbracket^2$. According to the energy estimate of Proposition \ref{energyf} and \eqref{sourceinho}, it suffices to prove that 
\begin{equation}\label{Aterms}
I_{\overline{A},\overline{D}} \hspace{1.5mm} := \hspace{1.5mm} \int_0^t \int_{\Sigma_s} \int_v \left| |K^j_i |^2 \overline{D}^r_q Y_r-2\left(\overline{A}^r_i K^j_r +K^r_i \overline{D}^j_r \right) K^j_i Y_q \right| \frac{dv}{v^0}dx ds \hspace{1.5mm} \lesssim \hspace{1.5mm} \epsilon^{\frac{3}{2}} \log^2 (3+t) ,
\end{equation}
\begin{equation}\label{Bterms}
I_{\overline{B}} \hspace{1.5mm} := \hspace{1.5mm} \int_0^t \int_{\Sigma_s} \int_v \left| \overline{B}^j_iK^j_iY_q \right| \frac{dv}{v^0}dx ds \hspace{1.5mm} \lesssim \hspace{1.5mm} \epsilon^{\frac{3}{2}} \log^2 (3+t).
\end{equation}
According to Proposition \ref{bilaninho} and \eqref{eq:one}, one has, using $\E_G (t) \lesssim \epsilon \log^2 (3+t)$ and $1 \leq v^0$ on the support of $Y$,
$$ I_{\underline{A}, \overline{D}} \hspace{0.3mm} \lesssim  \hspace{0.3mm} \int_0^t \int_{\Sigma_s} \int_v \left( \frac{\sqrt{\epsilon}}{\tau_+^{\frac{5}{4}}}+\frac{\sqrt{\epsilon} v^{\underline{L}}}{ \tau_-^{\frac{5}{4}} v^0} \right) |K|^2 |Y| dv dx ds  \hspace{0.3mm} \lesssim  \hspace{0.3mm} \sqrt{\epsilon} \int_0^t \frac{\E_G(s)}{(1+s)^{\frac{5}{4}}} ds + \sqrt{\epsilon} \int_{u=-\infty}^t \frac{\E_G(t)}{\tau_-^{\frac{5}{4}}} du \hspace{0.3mm} \lesssim \hspace{0.3mm} \epsilon^{\frac{3}{2}}\log^2 (3+t)  .$$
We now turn on \eqref{Bterms}, where the electromagnetic field is differentiated too many times to be estimated pointwise. According to Proposition \ref{bilaninho}, $1 \leq v^0$ on the support of $Y$ and using the Cauchy-Schwarz inequality in $v$, we can bound $\int_v | \overline{B}^j_iK^j_iY_q | \frac{dv}{v^0}$ by a linear combination of terms of the form
$$\bullet \hspace{1mm} \Phi_F^{\xi} := \left| \int_v  |z^2Y| dv \int_v | K|^2 |Y|  dv \right|^{\frac{1}{2}}\left( \tau_-\left| \rho_{\xi} \right|+\tau_+\left| \alpha_{\xi}  \right|+\left| \sigma_{\xi} \right| \right), \hspace{1cm} \bullet \hspace{1mm} \Phi_{\underline{\alpha}}^{\xi}:= \left| \int_v  |z^2Y| dv \int_v | K|^2 |Y|  dv \right|^{\frac{1}{2}} \left| \underline{\alpha}_{\xi}  \right|, $$
where $|\xi| \leq N$ and $(\alpha_{\xi},\underline{\alpha}_{\xi}, \rho_{\xi}, \sigma_{\xi})$ is the null decomposition of $\mathcal{L}_{Z^{\xi}}(F)$. Now, fix $|\xi| \leq N$. Using the Cauchy-Schwarz inequality in $x$, the bootstrap assumption \eqref{bootF2} and $\E_G(t) \leq \epsilon \log^2(3+t)$, we have
\begin{eqnarray}
\nonumber \int_0^t \int_{\Sigma_s} \Phi^{\xi}_{\underline{\alpha}} dx ds & \lesssim & \int_0^t \| \underline{\alpha}_{\xi} \|_{L^2(\Sigma_s)} \left\|  \left| \int_v  |z^2Y| dv \int_v | K|^2 |Y|  dv \right|^{\frac{1}{2}} \right\|_{L^2(\Sigma_s)} ds \\ \nonumber
& \lesssim & \int_0^t \sqrt{\mathcal{E}^5_N[F](s)} \left\| \int_v  |z^2Y| dv \right\|_{L^{\infty}(\Sigma_s)}^{\frac{1}{2}} \left\|  \int_v | K|^2 |Y|  dv \right\|_{L^1(\Sigma_s)}^{\frac{1}{2}}  ds \\ \nonumber
& \lesssim & \epsilon^{\frac{3}{2}} \int_0^t \frac{\sqrt{\E_G(s)}}{1+s} ds \hspace{2mm} \lesssim \hspace{2mm} \epsilon^{\frac{3}{2}} \log^2(3+t).
\end{eqnarray}
By the inequality $2ab \leq a^2+b^2$ and $\tau_+^2 \tau_- \int_v |z^2 Y| dv \lesssim \epsilon$, one has
$$ \Phi_F^{\xi} \leq \frac{\epsilon}{\tau_+^{\frac{5}{4}}} \int_v | K|^2 |Y|  dv + \frac{\epsilon}{ \tau_+^{\frac{3}{4}} \tau_-} \left( \tau_-\left| \rho_{\xi} \right|+\tau_+\left| \alpha_{\xi}  \right|+\left| \sigma_{\xi} \right| \right)^2,$$ 
so that, by Lemma \ref{foliationexpli} and the bootstrap assumption \eqref{bootF1}
\begin{eqnarray}
\nonumber \int_0^t \int_{\Sigma_s} \Phi^{\xi}_F dx ds & \lesssim & \epsilon \int_0^t \frac{\E_G(s)}{(1+s)^{\frac{5}{4}}} ds+\sum_{i =0}^{+\infty} \int_{u=-\infty}^t \frac{\epsilon}{\tau_-^{\frac{5}{4}} } \int_{C_u^i(t)} \frac{1}{\sqrt{\tau_+}} \left( \tau_-\left| \rho_{\xi} \right|+\tau_+\left| \alpha_{\xi}  \right|+\left| \sigma_{\xi} \right| \right)^2 dC_u^i(t) du \\ \nonumber
& \lesssim & \epsilon^{\frac{3}{2}}+ \epsilon \sum_{i =0}^{+\infty} \int_{-\infty}^t \frac{\mathcal{E}_N[F](T_{i+1}(t))}{ \tau_-^{\frac{5}{4}} \sqrt{1+t_i}}  du \hspace{2mm} \lesssim \hspace{2mm} \epsilon^{\frac{3}{2}}+ \epsilon^2 \sum_{i =0}^{+\infty} \frac{\log^4(1+2^{i+1})}{\sqrt{1+2^i}} \int_{-\infty}^{+\infty} \frac{du}{ \tau_-^{\frac{5}{4}} }  \hspace{2mm} \lesssim \hspace{2mm} \epsilon^{\frac{3}{2}}.
\end{eqnarray}
This concludes the improvement of the bootstrap assumption on $\E_G$ and then the proof.
\end{proof}
\subsection{The $L^2$ estimates}

In order to improve the bound on the electromagnetic field energy, we will use the following estimates.
\begin{Pro}\label{L2esti}
Let $z \in \V$ and $|\beta| \leq N$. Then,
$$\forall \hspace{0.5mm} t \in [0,T[, \hspace{15mm} \left\| \tau_+ \sqrt{\tau_-} \int_v \left| z \widehat{Z}^{\beta} f \right| dv \right\|_{L^2(\Sigma_t)} \lesssim \epsilon \log (3+t).$$
The logarithmical growth can be removed for $|\beta| \leq N-4$.
\end{Pro}
\begin{proof}
The cases $|\beta| \leq N-4$ ensue from \eqref{decayf2}. Suppose now that $|\beta| \geq N-3$, so that there exists $j \in \llbracket 1 , |I_1| \rrbracket$ such that $ \widehat{Z}^{\beta} f = H_j+G_j$. It then suffices to prove that both $H_j$ and $G_j$ satisfy such $L^2$-estimates. For $H_j$, one only has to use $\E_H \leq 3\epsilon$ on $[0,T[$ and the Klainerman-Sobolev inequality of Proposition \ref{KS2}. For $G_j$, recall that $G_j = K_j^q Y_q$ and use $\int_v |z^2 Y| dv \lesssim \epsilon \tau_+^{-2}$, which comes from Proposition \ref{bilaninho}, and the Cauchy-Schwarz inequality in $v$ in order to obtain
$$\left\| \int_v \left|z G_j \right| dv \right\|_{L^2(\Sigma_t)} \hspace{-0.5mm} = \left\| \int_v |z| \left| K_j^q Y_q \right| dv \right\|_{L^2(\Sigma_t)} \lesssim \left\| \int_v  \left| zY \right| dv \right\|^{\frac{1}{2}}_{L^{\infty}(\Sigma_t)}\left\| \int_v  \left| K_j^q \right|^2 \left| Y_q \right| dv \right\|^{\frac{1}{2}}_{L^1(\Sigma_t)} \hspace{-0.5mm} \lesssim \frac{\sqrt{\epsilon}}{1+t} \sqrt{\E_G(t)}.$$
It then remains to use Proposition \ref{inhoL1}, which gives $\E_G(t) \lesssim \epsilon \log^2(3+t)$.
\end{proof}
Combining this Proposition with the inequality $r |v^A| \lesssim v^0 \sum_{w \in \V} |w|$ (see Lemma \ref{weights1}), one can then improve the bootstrap assumption \eqref{bootL2} if $\epsilon$ is small enough.
\section{The energy bounds of the electromagnetic field}\label{sec9}
The last part of the proof consists in improving the bootstrap assumptions \eqref{bootF1} and \eqref{bootF2}. According to the energy estimate of Proposition \ref{energyMax1}, commutation formula of Proposition \ref{Maxcom} and $\mathcal{E}_N[F](0) \leq \epsilon$, $\mathcal{E}_N[F](t) \leq 3 \epsilon \log^{4}(3+t) $ and $\mathcal{E}^{5}[F](t) \leq  \underline{C} \epsilon$ for all $t \in [0,T[$ follow, if $\epsilon$ is small enough and $\underline{C}$ choosen large enough, from
\begin{eqnarray}
\nonumber \sum_{|\gamma| \leq N} \sum_{|\beta| \leq N} \int_0^t \int_{\Sigma_s} \left| \overline{K}_0^{\mu} \mathcal{L}_{Z^{\gamma}}(F)_{\mu \nu} \int_v \frac{v^{\nu}}{v^0} \widehat{Z}^{\beta} f dv \right| dx ds & \lesssim & \epsilon^{\frac{3}{2}} \log^4 (3+t), \\ 
\nonumber \sum_{|\gamma| \leq N} \sum_{|\beta| \leq N} \int_0^t \int_{\Sigma_s} \frac{\tau_-^2}{\log^5(1+\tau_-)} \left|  \mathcal{L}_{Z^{\gamma}}(F)_{0 \nu} \int_v \frac{v^{\nu}}{v^0} \widehat{Z}^{\beta} f dv \right| dx ds & \lesssim & \epsilon^{\frac{3}{2}}.
\end{eqnarray}
Fix $|\beta| \leq N$, $|\gamma| \leq N$, denote by $(\alpha, \underline{\alpha}, \rho, \sigma)$ the null decomposition of $\mathcal{L}_{Z^{\gamma}}(F)$ and recall that $\overline{K}_0^L=\frac{1}{2} \tau_+^2$ and $\overline{K}_0^{\underline{L}}=\frac{1}{2} \tau_-^2$. Expanding $\overline{K}_0^{\mu} \mathcal{L}_{Z^{\gamma}}(F)_{\mu \nu} J( \widehat{Z}^{\beta} f )^{\nu}$ and $ \mathcal{L}_{Z^{\gamma}}(F)_{0 \nu} J( \widehat{Z}^{\beta} f )^{\nu}$ in null coordinates, we can observe that it suffices to prove that
\begin{eqnarray}
\nonumber  I & := & \int_0^t \int_{\Sigma_s} \int_v \left(\tau_+^2 |\rho| \frac{v^{\underline{L}}}{v^0}+\tau_+^2 |\alpha| \frac{|v^A|}{v^0} +\tau_-^2 |\rho|\frac{v^L}{v^0}+\tau_-^2|\underline{\alpha}|\frac{|v^{A}|}{v^0} \right) \left| \widehat{Z}^{\beta} f \right| dv dx ds \hspace{2mm} \lesssim \hspace{2mm} \epsilon^{\frac{3}{2}} \log^4 (3+t) , \\
\nonumber  I_0 & := & \int_0^t \int_{\Sigma_s} \int_v \frac{\tau_-^2}{\log^5(1+\tau_-)} \left( |\rho| + |\alpha|+ |\underline{\alpha}|\frac{|v^{A}|}{v^0} \right) \left| \widehat{Z}^{\beta} f \right| dv dx ds \hspace{2mm} \lesssim \hspace{2mm} \epsilon^{\frac{3}{2}} .
\end{eqnarray}
Using the Cauchy-Schwarz inequality in $x$, $ \tau_+v^{\underline{L}}+\tau_+|v^A|+\tau_- v^L \lesssim v^0 \sum_{w \in \V} |w| $ (see Lemmas \ref{weights1}), the bootstrap assumption \eqref{bootF1} and Proposition \ref{L2esti}, we have
\begin{eqnarray}
\nonumber I & \lesssim & \sum_{w \in \V} \int_{0}^t  \left\| \tau_+ |\rho| +\tau_+ |\alpha|+\tau_- \underline{\alpha} \right\|_{L^2(\Sigma_s)}  \left\| \int_v \left|w \widehat{Z}^{\beta} f \right| dv \right\|_{L^2(\Sigma_s)}ds \\ \nonumber
& \lesssim & \epsilon \int_{0}^t  \sqrt{\mathcal{E}_N[F](s)} \frac{ \log(3+s)}{1+s} ds \hspace{2mm} \lesssim \hspace{2mm} \epsilon^{\frac{3}{2}} \int_{0}^t \frac{\log^3(3+s)}{1+s}ds \hspace{2mm} \lesssim \hspace{2mm} \epsilon^{\frac{3}{2}} \log^4(3+t).
\end{eqnarray}
Similarly, using $ \tau_+|v^A|+\tau_- v^0 \lesssim v^0 \sum_{w \in \V} |w| $, we obtain
\begin{eqnarray}
\nonumber I_0 & \lesssim & \sum_{w \in \V} \int_{0}^t  \left\| \tau_+ |\rho| +\tau_+ |\alpha|+\tau_- \underline{\alpha} \right\|_{L^2(\Sigma_s)}  \left\| \frac{\tau_-}{\tau_+ \log^5(1+\tau_-)}\int_v \left|w \widehat{Z}^{\beta} f \right| dv \right\|_{L^2(\Sigma_s)}ds \\ \nonumber
& \lesssim &  \int_{0}^t  \sqrt{\mathcal{E}_N[F](s)} \left\| \frac{\tau_+ \sqrt{\tau_-}}{\tau_+^{\frac{3}{2}}} \int_v \left|w \widehat{Z}^{\beta} f \right| dv \right\|_{L^2(\Sigma_s)} ds \hspace{2mm} \lesssim \hspace{2mm} \epsilon^{\frac{3}{2}} \int_{0}^t \frac{\log(3+s)}{(1+s)^{\frac{3}{2}}} ds \hspace{2mm} \lesssim \hspace{2mm} \epsilon^{\frac{3}{2}}.
\end{eqnarray}
These two estimates allow us to improve the bootstrap assumptions \eqref{bootF1} and \eqref{bootF2} if $\epsilon$ is small enough, which concludes the proof.

\appendix

\section{The Vlasov field vanishes for small velocities}\label{appendixA}

Let $F$ be a smooth $2$-form defined on $[0,T[ \times \R^3$ which satisfies
\begin{equation}\label{decayF}
\forall \hspace{0.5mm} (t,x) \in [0,T[ \times \R^3, \hspace{1cm} |F|(t,x) \hspace{1mm} \lesssim \hspace{1mm} \frac{\sqrt{\epsilon}}{\tau_+ \tau_-} \hspace{1cm} \text{and} \hspace{1cm} |\rho(F)|(t,x) \hspace{1mm} \lesssim \hspace{1mm} \frac{\sqrt{\epsilon}}{\tau_+^{\frac{3}{2}}}.
\end{equation}
The aim of this section is to prove the following result.
\begin{Pro}\label{velocity}
Let $f$ be a classical solution to $T^{\chi}_F(f)=0$ such that $f(0,.,v)=0$ for all $|v| \leq 3$. Then if $\epsilon$ is small enough, we have
$$\forall \hspace{0.5mm} (t,x,v) \in [0,T[ \times \R^3 \times (\R^3 \setminus \{ 0 \} ), \hspace{16mm}  |v| \leq 1 \hspace{2mm} \Rightarrow \hspace{2mm} f(t,x,v)=0.$$
\end{Pro}
The proof is based on the study of the characteristics of the system. As $f_0(.,v)=0$ for all $|v| \leq 3$, we consider $(x,v) \in \R^3 \times \R^3$ such that $|v| \geq 3$ and $(X,V)$ the characteristic of the operator $T^{\chi}_F$ such that $(X(0),V(0))=(x,v)$. Our goal is to prove $\inf_{[0,T[} |V| \geq 1$, which would imply Proposition \ref{velocity}. Then, suppose that $|V|$ reaches the value $1$ and define
$$t_1 := \inf \{ s \in [0,T[ \hspace{1mm} / \hspace{1mm}   |V(s)|=1 \}, \hspace{12mm} t_0 := \sup \{ s \in [0,t_1] \hspace{1mm} / \hspace{1mm} |V(s)| = 2 \}.$$
As $V$ is continuous, $t_0$ and $t_1$ are well defined. In view of the support of $\chi$, $(X,V)$ satisfies the following system of ODEs on $[t_0,t_1]$,
\begin{equation}\label{eq:characteristic2}
\forall \hspace{0.5mm} 1 \leq i \leq 3, \hspace{8mm} \frac{dX^i}{ds}(s)=\frac{V^i(s)}{|V(s)|} \hspace{8mm} \text{and} \hspace{8mm} \frac{dV^i}{ds}(s)=F_{0i}(s,X(s))+\frac{V^j(s)}{|V(s)|} F_{ji}(s,X(s)).
\end{equation}
We then deduce, since $F$ is a $2$-form, that $\frac{d\left( |V|^2 \right)}{ds}=2V^i F_{0i}$, which implies
\begin{equation}\label{eq:characteristic}
\forall \hspace{0.5mm} t_0 \leq t \leq t_1, \hspace{1cm} |V(t)|^2 \geq |V(t_0)|^2-2\int_{t_0}^t \left| V^i(s) F_{0i}(s,X(s)) \right| ds.
\end{equation}
Before presenting the strategy of the proof, let us introduce certain subsets of $[t_0,t_1]$ and two constants. Note that if $\epsilon$ is small enough, we can suppose that $22 \leq t_0 < t_1$. We can then introduce two constants $\delta >0$ and $K>0$ independent of $\epsilon$ and satisfying 
\begin{equation}\label{deltaK}
 4\delta \leq 1+\delta \leq  K < \frac{\pi}{4 \sqrt{2}} t_0^{\frac{1}{4}} \hspace{15mm} \text{and} \hspace{15mm}  2^{-\frac{5}{2}} K^2-\delta > 2 \delta . 
  \end{equation}
We also consider, for $Q >0$, the following subsets of $[t_0,t_1]$,
$$ \mathcal{A}_Q := \{ s \in [t_0,t_1] \hspace{1mm} / \hspace{1mm} |s-|X(s)|| \geq Q s^{\frac{1}{4}} \} \hspace{12mm} \text{and} \hspace{12mm} \overline{\mathcal{A}}_Q := [t_0,t_1] \setminus \mathcal{A}_Q.$$ 
Then, using \eqref{eq:characteristic} and $\sup_{[t_0,t_1]} |V| \leq 2$, we have for all $t \in [t_0,t_1]$,
\begin{eqnarray}
\nonumber |V(t)|^2 & \geq & 4-4\int_{\mathcal{A}_{\delta}} \left| \frac{V^i(s)}{|V(s)|} F_{0i}(s,X(s)) \right| ds -4 \int_{\overline{\mathcal{A}}_{\delta}} \left| \frac{V^i(s)}{|V(s)|}-\frac{X^i(s)}{s} \right| \left| F_{0i}(s,X(s)) \right| ds \\ \nonumber
& & -4 \int_{\overline{\mathcal{A}}_{\delta}} \frac{|X(s)|}{s} \left| \frac{X^i(s)}{|X(s)|}  F_{0i}(s,X(s)) \right| ds.
\end{eqnarray}
The result would ensue if we could bound the three integrals on the right hand side of the last inequality by $C \sqrt{\epsilon}$, with $C>0$ a constant independant of $T$ and $(x,v)$. Indeed, we would then obtain, for $\epsilon < (2C)^{-2}$,
$$\forall \hspace{0.5mm} t \in [t_0,t_1], \hspace{10mm} |V(t)| \geq  \sqrt{2},$$
which would contradict $|V(t_1)| = 1$. We can easily bound two of these integrals, using either the strong decay rate of $F$ away from the light cone or the strong decay rate satisfied by the null component $\rho$ and that $ \frac{|X(s)|}{s}$ is bounded near the light cone. More precisely, using \eqref{decayF}, the definition of $\mathcal{A}_{\delta}$ and $t_0 \geq 1$, one has
\begin{eqnarray}
\nonumber \int_{\mathcal{A}_{\delta}} \left| \frac{V^i(s)}{|V(s)|} F_{0i}(s,X(s)) \right| ds & \lesssim & \int_0^{+\infty} \frac{\sqrt{\epsilon} ds}{(1+s)(1+\delta s^{\frac{1}{4}})} \hspace{2mm} \lesssim \hspace{2mm} \sqrt{\epsilon}, \\
\nonumber \int_{\overline{\mathcal{A}}_{\delta}} \frac{|X(s)|}{s} \left| \frac{X^i(s)}{|X(s)|}  F_{0i}(s,X(s)) \right| ds & \lesssim & \int_{\overline{\mathcal{A}}_{\delta}}(1+\delta ) \left|\rho(F)(s,X(s)) \right| ds \hspace{2mm} \lesssim \hspace{2mm} \int_0^{+ \infty} \frac{\sqrt{\epsilon}}{(1+s)^{\frac{3}{2}}}ds \hspace{2mm} \lesssim \hspace{2mm} \sqrt{\epsilon}.
\end{eqnarray}
For the last integral, observe, in view of \eqref{decayF}, that
\begin{equation}\label{eq:velocity1}
I_1:=\int_{\overline{\mathcal{A}}_{\delta}} \left| \frac{V^i(s)}{|V(s)|}-\frac{X^i(s)}{s} \right| \left| F_{0i}(s,X(s)) \right| ds \lesssim \sqrt{\epsilon} \int_{\overline{\mathcal{A}}_{\delta}} \left| \frac{V(s)}{|V(s)|}-\frac{X(s)}{s} \right| \frac{1}{1+|s-|X(s)||} \frac{ds}{1+s}.
\end{equation}
As $|s-|X(s)||$ is small on $\overline{\mathcal{A}}_{\delta}$, the goal is to obtain enough decay from $\frac{V(s)}{|V(s)|}-\frac{X(s)}{s}$. The rough idea behind the following computations is the following. As $|s-X(s)|$ is small, then, by \eqref{eq:characteristic2}, $s \sim |X(s)| \sim \left| \int_0^s \frac{V(\tau)}{|V(\tau)|} d \tau \right| $ and we almost have equality in the triangular inequality 
$$ \left| \int_0^s \frac{V(\tau)}{|V(\tau)|} d \tau \right| \leq \int_0^s \left|  \frac{V(\tau)}{|V(\tau)|}  \right| d \tau = s.$$
$\frac{V}{|V|}$ then almost keeps a constant direction $\vec{u}$, so that $$ \frac{X(s)}{s} \sim \frac{1}{s} \int_0^s \frac{V(\tau)}{|V(\tau)|} d \tau \sim \vec{u} \sim \frac{V}{|V|}.$$
In order to bound \eqref{eq:velocity1}, let us introduce, for $Q > 0$ and $\delta_0>0$, the following subsets of $[t_0,t_1]$,
$$ \mathcal{B}_Q := \left\{ s \in [t_0,t_1] \hspace{1mm} / \hspace{1mm} \left| \frac{V(s)}{|V(s)|}-\frac{X(s)}{s} \right| > \frac{Q}{ s^{\frac{1}{4}} } \right\}, \hspace{12mm} \overline{\mathcal{B}_Q}:= [t_0,t_1] \setminus \mathcal{B}_Q \hspace{12mm} \text{and} \hspace{12mm} \mathcal{C}^{\delta_0}_Q:= \overline{\mathcal{A}}_{\delta_0} \cap \mathcal{B}_Q.$$
In view of the definition of $\mathcal{C}^{\delta}_{4K}$ and \eqref{eq:velocity1}, $I_1 \lesssim \sqrt{\epsilon}$ would ensue if we prove
\begin{equation}\label{boundI}
 I \hspace{1mm} := \hspace{1mm} \int_{\mathcal{C}^{\delta}_{4K}} \left| \frac{V^i(s)}{|V(s)|}-\frac{X^i(s)}{s} \right| \left| F_{0i}(s,X(s)) \right| ds \hspace{1mm} \lesssim \hspace{1mm} \sqrt{\epsilon}.
 \end{equation}
From now, we suppose that $\mathcal{C}^{\delta}_{4K} \neq \varnothing$ as otherwise, $I=0$. We start by the following two results.
\begin{Lem}\label{point1}
Let $s \in \mathcal{C}^{\delta}_{4K}$. Then, $[s,\min(t_1,s+s^{\frac{3}{4}})] \subset \mathcal{B}_{2K}$.
\end{Lem}
\begin{proof}
Let $t \in [s,\min(t_1,s+s^{\frac{3}{4}})]$. As $s \in \mathcal{B}_{4K}$ and by the triangle inequality, one has
$$\frac{4K}{s^{\frac{1}{4}}} \leq \left| \frac{V^i(s)}{|V(s)|}-\frac{X^i(s)}{s} \right| \leq \left| \frac{V^i(t)}{|V(t)|}-\frac{X^i(t)}{t} \right| +\left| \frac{V^i(s)}{|V(s)|}-\frac{V^i(t)}{|V(t)|} \right|+\left| \frac{X^i(s)}{s}-\frac{X^i(t)}{t} \right|.$$
By the mean value theorem applied to the function $\frac{V}{|V|}$ and using the estimate \eqref{decayF}, we have
$$\left| \frac{V^i(s)}{|V(s)|}-\frac{V^i(t)}{|V(t)|} \right| \hspace{1mm} \leq \hspace{1mm} C_1 \sup_{\tau \in [s,t[} |F(\tau,X(\tau))| |t-s|  \hspace{1mm} \leq \hspace{1mm} \frac{C_0\sqrt{\epsilon}}{1+s}|t-s| \hspace{1mm} \leq \hspace{1mm} \frac{C_0\sqrt{\epsilon}}{s}|t-s|.$$
Using $|X(s)| \leq s+\delta s^{\frac{1}{4}}$ and $|X(t)-X(s)| \leq |t-s|$, which results from \eqref{eq:characteristic2}, we have
$$ \left| \frac{X^i(s)}{s}-\frac{X^i(t)}{t} \right| \hspace{1mm} \leq \hspace{1mm} \frac{(t-s)|X^i(s)|+s|X^i(t)-X^i(s)|}{ts} \hspace{1mm} \leq \hspace{1mm} \frac{2+\delta}{t}|t-s| \hspace{2mm} \leq \hspace{2mm} \frac{ 2+\delta}{s}|t-s|.$$
Thus, as $s \leq t \leq s+s^{\frac{3}{4}}$ and $2K \geq 2+2 \delta$ (see \eqref{deltaK}), it follows, for $\epsilon$ small enough,
$$\frac{4K}{s^{\frac{1}{4}}}-\frac{C_0 \sqrt{\epsilon}+2+\delta}{s}|t-s| \geq \frac{4K-C_0 \sqrt{\epsilon}-2-\delta}{s^{\frac{1}{4}}} > \frac{2K}{t^{\frac{1}{4}}}, \hspace{8mm} \text{so that} \hspace{8mm} t \in \mathcal{B}_{2K}.$$
\end{proof}
\begin{Lem}\label{point2}
Suppose that $s \in \mathcal{C}^{\delta}_{4K}$ and let $t_*(s)$ be equal to $ \inf \{ t \in [s,t_1] \hspace{1mm} / \hspace{1mm} t \notin \mathcal{C}_{2K}^{2 \delta} \}$ if it is well defined and $t_*(s)=t_1$ otherwise. Then,
$$ \forall \hspace{0.5mm} t \in [s,t_*(s)], \hspace{8mm} t-|X(t)| \geq s-|X(s)|+K^2\sqrt{t}-K^2\sqrt{s}. $$
\end{Lem}
\begin{proof}
Let $g : t \mapsto t-|X(t)|$, so that $g'(t)= 1-\langle \frac{V(t)}{|V(t)|}, \frac{X(t)}{|X(t)|} \rangle $. Let us estimate $\theta \in [0, \pi[$, the angle between $V$ and $X$. For $t \in [s,t_*(s)[$, we have 
\begin{itemize}
\item $\left| \frac{V(t)}{|V(t)|}-\frac{X(t)}{|X(t)|} \right| \leq |\theta(t) |$ since $\frac{V}{|V|}$ and $\frac{X}{|X|}$ are unit vectors.
\item $\frac{2K}{t^{\frac{1}{4}}} \leq \left| \frac{V(t)}{|V(t)|}-\frac{X(t)}{t} \right|$ as $t \in B_{2K}$ and $\left| \frac{|X(t)|}{t}-1 \right|  \leq \frac{2\delta}{t^{\frac{3}{4}}}$ since $t \in \overline{A}_{2\delta}$.
\end{itemize}
We then obtain, using $4 \delta \leq K$, the trivial fact $\left| \frac{X(t)}{t}-\frac{X(t)}{|X(t)|} \right|=\left| \frac{|X(t)|}{t}-1 \right|$ and the triangle inequality, that
$$\frac{ \sqrt{2} K}{t^{\frac{1}{4}}} \hspace{1mm} \leq \hspace{1mm} \frac{2K}{t^{\frac{1}{4}}}-\frac{2\delta}{t^{\frac{3}{4}}} \hspace{1mm} \leq \hspace{1mm} \left| \frac{V(t)}{|V(t)|}-\frac{X(t)}{t} \right|-\left| \frac{X(t)}{t}-\frac{X(t)}{|X(t)|} \right| \hspace{1mm} \leq \hspace{1mm} \left| \frac{V(t)}{|V(t)|}-\frac{X(t)}{|X(t)|} \right| \hspace{1mm} \leq \hspace{1mm} |\theta (t) |.$$
Consequently, as $1-\frac{1}{4}\phi^2 \geq \cos \phi$ for $\phi \in [0, \frac{\pi}{4}]$ and since $\sqrt{2} K \leq \frac{\pi}{4}t_0^{\frac{1}{4}}$, we obtain
$$g'(t) =1-\cos \theta (t) \geq 1-\cos \left( \frac{\sqrt{2} K}{t^{\frac{1}{4}}} \right) \geq \frac{K^2}{ 2\sqrt{t}} \hspace{12mm} \text{and then} \hspace{12mm} g(t) \geq g(0)+K^2(\sqrt{t}-\sqrt{s}) .$$
\end{proof}
The strategy now is to prove that $\mathcal{C}^{\delta}_{4K}$ is composed of pieces sufficiently well separated for \eqref{boundI} to hold. It is then convenient to consider a dyadic partition of $[t_0,t_1]$, which leads us to introduce, for all $i \in \mathbb{N}$,
$$ \mathcal{C}^i_{4K} := \mathcal{C}^{\delta}_{4K} \cap [2^i,2^{i+1}[ \hspace{12mm} \text{and, if} \hspace{2mm} \mathcal{C}^i_{4K} \neq \varnothing, \hspace{12mm} s^i = \inf \mathcal{C}^i_{4K} \hspace{12mm} \text{and} \hspace{12mm} t_*^i = t_*(s^i).$$
\begin{Cor}\label{point3}
Let $i \in \mathbb{N}$ such that $\mathcal{C}^i_{4K} \neq \varnothing$. Then, $\mathcal{C}^i_{4K} \subset [s^i, \min(t_*^i,2^{i+1})]$. Moreover, if $|X(s^i)|-s^i >0$, $t \mapsto t-|X(t)|$ is positive on $[2^{i+2},t_1]$.
\end{Cor}
\begin{proof}
Let $i \in \mathbb{N}$ and suppose that $\mathcal{C}^i_{4K} \neq \varnothing$. We assume moreover that $t_*^i < t_1$ since there is nothing to prove otherwise and we introduce $T^i=\min(t_1,2^{i+2})$. We will use several times that $s-|X(s)| > 0$, for $s \in [t_0,t_1]$, implies that $t \mapsto t-|X(t)|$ is positive on $[s,t_1]$ since it is an increasing function. As $t_*^i \in \mathcal{A}_{2 \delta} \cup \overline{\mathcal{B}}_{2K}$ by definition, we have two cases to study.
\begin{itemize}
\item If $t_*^i \in \overline{\mathcal{B}}_{2K}$, then, by Lemma \ref{point1}, $t_*^i \geq \tau^i := s^i+|s^i|^{\frac{3}{4}}$. Hence, using Lemma \ref{point2}, we get
$$\tau^i -|X(\tau^i)| \geq s^i -|X(s^i)|+ K^2 \sqrt{\tau^i}-K^2 \sqrt{s^i} \geq -\delta \left| s^i \right|^{\frac{1}{4}}+K^2 \sqrt{s^i} \left( \sqrt{1+|s^i|^{-\frac{1}{4}}}-1 \right).$$
Since $4\sqrt{1+h}-4 \geq h$ for all $h \in [0,1]$ and $t \mapsto t -|X(t)|$ increases, we obtain, for all $t \in [\tau^i,T^i]$,
\begin{equation}\label{taui}
t-|X(t)| \geq \tau^i -|X(\tau^i)|  \geq -\delta 2^{\frac{i+2}{4}}+ K^2 \sqrt{s^i} \frac{1}{4}|s^i|^{-\frac{1}{4}} \geq (2^{-\frac{10}{4}}K^2-\delta)2^{\frac{i+2}{4}}  \geq 2 \delta \left| T^i \right|^{\frac{1}{4}} \geq 2 \delta t^{\frac{1}{4}} .
\end{equation}
so that\footnote{We then necessarily have $\tau^i=t_*^i$.} $[\tau^i,T_i] \subset \mathcal{A}_{2 \delta} $. As $C^i_{4K} \subset \overline{A}_{\delta}$ and $\mathcal{A}_{2 \delta} \cap \overline{A}_{\delta} = \varnothing$, we obtain $C^i_{4K} \subset [s^i, \tau^i] \subset [s^i, t_*^i]$. Observe also that $\tau^i-|X(\tau^i)| >0$, which implies that $t \mapsto t-|X(t)|$ is positive on $[\tau^i,t_1] \subset [2^{i+2},t_1]$.
\item Otherwise $t_*^i \in \mathcal{A}_{2 \delta}$. If $t_*^i \geq T^i$, we have $\tau^i \leq 2^{i+2} \leq t_*^i$ so that $C^i_{4K} \subset [s^i, 2^{i+1}] \subset [s^i, t_*^i]$. The positivity of $t \mapsto t-|X(t)|$ on $[2^{i+2},t_1]$ then follows from \eqref{taui}. Otherwise, as $|t_*^i-|X(t_*^i)|| \geq 2 \delta |t_*^i|^{\frac{1}{4}}$ and $t \mapsto t -|X(t)|$ increases, we necessarily have $t_*^i-|X(t_*^i)| \geq 0$ since $s^i-|X(s^i)| \geq - \delta |s^i|^{\frac{1}{4}}$. Hence, using again that $t \mapsto t -|X(t)|$ increases, it yields
$$ \forall \hspace{0.5mm} t \in [t_*^i,T^i], \hspace{10mm} t-|X(t)| \geq 2 \delta \left| t_*^i \right|^{\frac{1}{4}} > \delta t^{\frac{1}{4}}, \hspace{5mm} \text{as} \hspace{5mm} 2 \delta \left| t_*^i \right|^{\frac{1}{4}} \geq 2 \delta 2^{\frac{i}{4}} > \delta 2^{\frac{i+2}{4}} .$$
We can then conclude that $[t_*^i,T^i] \subset \mathcal{A}_{\delta}$, implying that $C^i_{4K} \subset [s^i,t_*^i]$ since $A_{\delta} \cap C^i_{4K} = \varnothing$. We also proved that $t \mapsto t-|X(t)|$ on $[2^{i+2},t_1]$ as $t_*^i-|X(t_*^i)| \geq 0$ and $t_*^i \leq 2^{i+2}$.
\end{itemize}
\end{proof}
We are now able to bound $I$. Let $\mathcal{D} := \{ i \in \mathbb{N} \hspace{1mm} / \hspace{1mm} C^i_{4K} \neq \varnothing \}$. Suppose first that $t \geq |X(t)|$ for all $t \in \mathcal{C}_{4K}^{\delta}$. Then, using \eqref{decayF}, Lemma \ref{point2} and Corollary \ref{point3},
\begin{eqnarray}
\nonumber I & \lesssim & \sum_{i \in \mathcal{D} } \int_{\mathcal{C}^i_{4K}} \frac{\sqrt{\epsilon}}{(1+s)(1+|s-|X(s)||)} ds \hspace{2mm} \lesssim \hspace{2mm} \sum_{i \in \mathcal{D} } \int_{s^i}^{2^{i+1}} \frac{\sqrt{\epsilon}}{(1+s)(1+K^2\sqrt{s}-K^2\sqrt{s^i})} ds \\ \nonumber
& \lesssim &  \sum_{i \in \mathcal{D} } \sqrt{\epsilon} \int_{s^i}^{2^{i+1}} \frac{1+\sqrt{s}+\sqrt{s^i}}{(1+s)(1+s-s^i)} ds \hspace{2mm} \lesssim \hspace{2mm}  \sum_{i =0 }^{+ \infty} \frac{\sqrt{\epsilon}}{2^{\frac{i}{2}}} \int_{2^i}^{2^{i+1}} \frac{1}{(1+s-2^i)} ds \\ \nonumber 
& \lesssim & \sqrt{\epsilon} \sum_{i =0 }^{+ \infty} \frac{\log(1+2^i)}{2^{\frac{i}{2}}} \hspace{2mm} \lesssim \hspace{2mm} \sqrt{\epsilon}.
\end{eqnarray}
Otherwise, with $p= \min \mathcal{D}$ and according to Corollary \ref{point3}, we have $t \geq |X(t)|$ for all $t \geq 2^{p +2}$ and the result then follows from
$$ \sqrt{\epsilon} \int_{\mathcal{C}^{p}_{4K} \cup \mathcal{C}^{p+1}_{4K}} \frac{ds}{1+s} \leq \sqrt{\epsilon} \int_{2^p}^{2^{p +2}} \frac{ds}{1+s} \leq \sqrt{\epsilon} \log (4).$$ 
In order to apply this result in Subsection \ref{subsecH}, we need to adapt it to an echeloned system of transport equations. 
\begin{Cor}\label{velocity2}
Let $k \in \mathbb{N}^*$ and, for all $1 \leq j <  i \leq k$ and $1 \leq q \leq 3$, let $A^{q,j}_i$ be a sufficiently regular matrix valued function defined on $[0,T[ \times \R^3 \times \left( \R^3 \setminus \{ 0 \} \right)$. Consider $g=(g_1,...,g_k)$, where each $g_i$ is a vector valued field, a classical solution on $[0,T[$ to the system
$$T^{\chi}_F(g_1)=0, \hspace{2.5cm} T_F^{\chi}(g_i)=A^{q,j}_i \partial_{v^q} g_{j} \hspace{1cm} 2 \leq i \leq k.$$
If $g(0,.,v)=0$ for all $|v| \leq 3$, then $g(t,.,v)=0$ for all $|v| \leq 1$.
\end{Cor}
\begin{proof}
Let $(t,x,v) \in [0,T[ \times \R^3 \times \left( \R^3 \setminus \{ 0 \} \right)$ such that $|v| < 1$. We denote by $(X_s,V_s)$ the value in $s$ of the characteristic of the operator $T_F^{\chi}$ which was equal to $(x,v)$ in $s=t$. By Duhamel's formula, we have
\begin{eqnarray}
\nonumber g_1(t,x,v) & = & g_1(s,X_s,V_s) \hspace{2mm} = \hspace{2mm} g_1(0,X_0,V_0), \\  g_i(t,x,v) & = & g_i(0,X_0,V_0)+\int_0^t A^{q,j}_i(s,X_s,V_s) \partial_{v^q} g_j(s,X_s,V_s) ds \hspace{1cm} \text{for all $2 \leq i \leq k$.} \label{eq:A1}
\end{eqnarray} 
 According to the proof of Proposition \ref{velocity},
\begin{equation}\label{eq:A2}
|V_0| < 3, \hspace{1.5cm} \text{so that} \hspace{1.5cm} \forall \hspace{0.5mm} 1 \leq i \leq k, \hspace{1cm} g_i(0,X_0,V_0)=0
\end{equation}
since otherwise we would have $|V_t|=|v| \geq 1$. Fix now $ s \in [0,t]$ and consider $w \in \R^3$ such that $|w|< \frac{1}{2}|v|$. We denote by $(X_{\tau}^{w,s}, V_{\tau}^{w,s})$ the value in $\tau$ of the characteristic of $T^{\chi}_F$ which was equal to $(X_s,V_s+w)$ in $\tau=s$. Then,
$$ g_1(s,X_s,V_s+w)=g_1(0,X_0^{w,s},V_0^{w,s}).$$
By continous dependence on the initial condition of the solutions to 
$$ \frac{dX^i}{ds}(s)=\frac{V^i(s)}{|V(s)|}, \hspace{1.5cm} \frac{dV^i}{ds}(s)=\chi(|V(s)|)F_{0i}(s,X(s))+\frac{V^j(s)}{|V(s)|} \chi(|V(s)|) F_{ji}(s,X(s)), \hspace{1cm} 1 \leq i \leq 3,$$ and since $|V_0| < 3$, there exists $\overline{\delta} >0$ (depending on $(t,s,x,v)$) such that $|V_0^{w,s}| <3$ for all $|w| < \overline{\delta}$. Hence,
\begin{equation}\label{eq:A3}
 \forall \hspace{0.5mm} |w| < \overline{\delta}, \hspace{1cm} g_1(s,X_s,V_s+w)=0, \hspace{1cm} \text{so that} \hspace{1cm} \forall \hspace{0.5mm} 1 \leq q \leq 3, \hspace{5mm}  \partial_{v^q} g_1 (s,X_s,V_s)=0.
\end{equation}
Repeating the argument, one can obtain
$$\forall \hspace{0.5mm} |w| \leq \overline{\delta}, \hspace{3mm} \forall \hspace{0.5mm} \tau \in [0,s], \hspace{3mm} \exists \underline{\delta} > 0, \hspace{3mm} \forall \hspace{0.5mm} |\underline{w}| \leq \underline{\delta}, \hspace{15mm} g_1(\tau,X_{\tau}^{w,s},V_{\tau}^{w,s}+\underline{w})=0 ,$$
which implies
\begin{equation}\label{eq:A4}
\forall \hspace{0.5mm} |w| \leq \overline{\delta}, \hspace{3mm} \forall \hspace{0.5mm} \tau \in [0,s], \hspace{0.3cm} \forall \hspace{0.5mm} 1 \leq q \leq 3, \hspace{15mm}  \partial_{v^q} g_1 (\tau,X_{\tau}^{w,s},V_{\tau}^{w,s})=0.
\end{equation}
Combining \eqref{eq:A1}, \eqref{eq:A2} and \eqref{eq:A3}, we get $g_2(t,x,v)=0$. Using \eqref{eq:A4} and $|V_0^{w,s}| <3$ for all $|w| < \overline{\delta}$, it yields, in view of the support of $g_2(0,.,.)$,
$$\forall \hspace{0.5mm} |w| < \overline{\delta}, \hspace{1cm} g_2(s,X_s,V_s+w)=g_2(0,X_0^{w,s},V_0^{w,s})+\int_0^s A^{q,j}_i( \tau, X_{\tau}^{w,s}, V_{\tau}^{w,s}) \partial_{v^q} g_1 ( \tau, X_{\tau}^{w,s}, V_{\tau}^{w,s}) d \tau=0.$$
We then deduce that $\partial_{v^q} g_2(s,X_s,V_s)=0$ for all $q \in \llbracket 1,3 \rrbracket$ and $s \in [0,t]$, so that, by \eqref{eq:A1}, $g_3(t,x,v)=0$. We then proved the desired result if $k=3$. The general case can be treated similarly by an induction.

\end{proof}

\section{Bounding the initial norms}\label{AppendixB}
We consider in this section $(f^0,F^0)$ satisfying the hypotheses of Theorem \ref{theorem} and $(f,F)$ the unique classical solution of \eqref{VM1}-\eqref{VM3} arising from these data.
\begin{Pro}\label{PropB}
There exists a constant $C_1$, depending only on $N$, such that
$$ \mathcal{E}_N[F](0) \leq C_1 \epsilon \hspace{1.5cm} \text{and} \hspace{1.5cm} \E^2_{N+3}[f](0) \leq C_1 \epsilon .$$
\end{Pro}
The proof is a corollary of the following proposition and that, for all $\widehat{Z}^{\beta} \in \K^{|\beta|}$ and $Z \in \mathbb{K}^{|\gamma|}$,
$$|\widehat{Z}^{\beta} f | \lesssim \sum_{|\alpha_2|+|\alpha_1|+p \leq |\beta|} \tau_+^{|\alpha_1|+p} |v|^{|\alpha_2|} |  \partial_t^p  \partial_x^{\alpha_1}\partial_{v}^{\alpha_2} f |, \hspace{1.5cm} \left| \mathcal{L}_{Z^{\gamma}}(F) \right| \lesssim \sum_{|\kappa|+q \leq |\gamma|} \tau_+^{|\kappa|+q} |  \nabla_{\partial_t}^q \nabla_x^{\kappa} F|.$$
\begin{Pro}\label{ProB}
We have, for all $|\alpha_2|+|\alpha_1|+p \leq N+3$ and $|\kappa|+q \leq N+2$,
\begin{equation}\label{iterB}
\hspace{-1mm} \left\| (1+|x|)^{|\alpha_1|+p+2} |v|^{|\alpha_2|} \partial_t^p  \partial_x^{\alpha_1}\partial_{v}^{\alpha_2} f \right\|_{L^1_v L^1(\Sigma_0)} \lesssim \epsilon, \hspace{0.8cm} \text{and} \hspace{0.8cm} \left\| (1+|x|)^{|\kappa|+q+1}  \nabla^q_{\partial_t} \nabla_x^{\kappa}  F \right\|^2_{L^2(\Sigma_0)} \lesssim \epsilon.
\end{equation}
\end{Pro}
\begin{proof}
Note that $\tau_+^2=1+|x|^2$ on $\Sigma_0$. The proof consists of an induction on $\max(p,q)$. The result holds for $\max(p,q)=0$ in view of the hypotheses on $(f^0,F^0)$. Let $r \in \llbracket 1 , N+2 \rrbracket$ and suppose that \eqref{iterB} is satisfied for all $p \leq r$ and all $ q \leq r$. If $r < N+2$, fix $|\kappa| \leq N+2-(r+1)$ and notice that, using \eqref{VM2} and Lemma \ref{maxwellbis},
$$ \partial_t F_{0i} = \partial^j F_{ji}+\int_{v \in \R^3} \frac{v^i}{v^0} f dv \hspace{1cm} \text{and} \hspace{1cm} \partial_t F_{ij}=\partial_i F_{0j}+\partial_j F_{i0}.$$
Then, by a standard $L^2_x-L^1_x$ Sobolev inequality and using the induction hypothesis twice, we get
\begin{eqnarray}
\nonumber \left\| \tau_+^{|\kappa|+r+2} \nabla^{r+1}_{\partial_t} \nabla_x^{\kappa}  F   \right\|^2_{L^2(\Sigma_0)} & \leq  & 2\left\| \tau_+^{|\kappa|+r+2}  \nabla_x \nabla^r_{\partial_t} \nabla_x^{\kappa}  F \right\|^2_{L^2(\Sigma_0)}+\left\| \tau_+^{|\kappa|+r+2}  \int_{v} \partial^r_t \partial_x^{\kappa}  f dv \right\|^2_{L^2(\Sigma_0)} \\ \nonumber
& \lesssim & \epsilon + \sum_{|\beta| \leq 2} \left\| \tau_+^{|\kappa|+r+2} \int_{v} \partial_x^{\beta} \partial^r_t \partial_x^{\kappa}  f dv \right\|^2_{L^1(\Sigma_0)} \hspace{2mm} \lesssim \hspace{2mm} \epsilon+\epsilon^2,
\end{eqnarray} 
since $r+|\kappa|+|\beta| \leq N+3$. We now turn on the Vlasov field and we do not suppose $r < N+2$ anymore. Start by fixing $\alpha_1$ and $\alpha_2$ such that $|\alpha_1|+|\alpha_2|+r+1 \leq N+3$. Iterating the commutation formula $T_F(\partial_{\mu} f ) = - \mathcal{L}_{\partial_{\mu}} (F)(v, \nabla_v f)$ and using $\mathcal{L}_{\partial_{\mu}}= \nabla_{\partial_{\mu}}$, we get
$$\partial_t \partial_t^r \partial^{\alpha_1}_x f=-\frac{v^i}{|v|} \partial_i \partial_t^r \partial^{\alpha_1}_x f+\sum_{q+p=r} \sum_{\begin{subarray}{} \hspace{2mm} |\kappa|+|\beta_1| = |\alpha_1| \\ p+|\beta_1| \leq r+|\alpha_1|-1 \end{subarray}}  \frac{v^{\mu}}{|v|} \nabla^q_{\partial_t} \nabla^{\kappa}_x {F_{\mu}}^{ j} \partial_{v^j} \partial_t^p \partial_x^{\beta_1} f.$$
Taking the $\partial^{\alpha_2}_v$ derivative of each side and multiplying them by $\tau_+^{|\alpha_1|+r+3}|v|^{|\alpha_2|}$, one obtains
\begin{eqnarray} \label{eqB1}
\hspace{-5mm} \tau_+^{|\alpha_1|+r+1+2}|v|^{|\alpha_2|} |\partial_t^{r+1} \partial^{\alpha_1}_x \partial^{\alpha_2}_v f| & \leq & \sum_{|\alpha_3|+n=|\alpha_2|} \tau_+^{|\alpha_1|+1+r+2} |v|^{|\alpha_2|-n}| \partial_x \partial_t^r \partial^{\alpha_1}_x \partial^{\alpha_3}_v f|+ \\ 
& & \hspace{-4mm} \sum_{\begin{subarray}{} \hspace{1mm} |\kappa|+|\beta_1| = |\alpha_1| \\ p+|\beta_1| \leq r+|\alpha_1|-1 \end{subarray}} \sum_{\begin{subarray}{} \hspace{2mm} q+p=r \\ |\alpha_3|+n=|\alpha_2| \end{subarray}} \tau_+^{|\alpha_1|+r+3} |v|^{|\alpha_2|-n} \left| \nabla^q_{\partial_t} \nabla^{\kappa}_x F \partial_{v} \partial_t^p \partial_x^{\beta_1} \partial_{v}^{\alpha_3} f \right|. \label{eqB2}
\end{eqnarray}
By the induction hypothesis, the $L^1_v L^1(\Sigma_0)$ norm of the terms of \eqref{eqB1} are bounded by $\epsilon$. Consider parameters such as in the sum in \eqref{eqB2}.

$\bullet$ If $q+|\kappa| \leq N$, then, using a standard $L^{\infty}-L^2$ Sobolev inequality on $\tau_+^{q+|\kappa|+1} \nabla^q_{\partial_t} \nabla^{\kappa}_x F$ and that $|v| \geq 3$ on the support of $f^0$, we get
$$ \tau_+^{|\alpha_1|+r+3} |v|^{|\alpha_2|-n} \left| \nabla^q_{\partial_t} \nabla^{\kappa}_x F \partial_{v} \partial_t^p \partial_x^{\beta_1} \partial_{v}^{\alpha_3} f \right| \lesssim  \tau_+^{|\beta_1|+p+2} |v|^{|\alpha_3|+1} \left| \partial_t^p \partial_x^{\beta_1} \partial_v \partial_{v}^{\alpha_3} f \right| \sum_{|\gamma| \leq |\kappa|+2} \left\| \tau_+^{|\kappa|+q+1} \nabla^{q}_{\partial_t} \nabla_x^{\gamma}  F   \right\|_{L^2_x}.$$
The $L^1_v L^1(\Sigma_0)$ norm of the left hand side of the previous inequality is then bounded by $\epsilon^{\frac{3}{2}}$ according to the induction hypothesis and $p+q=r$.

$\bullet$ Otherwise $|\beta_1|+p \leq 2$ and, by the Cauchy-Schwarz inequality in $x$,
$$\left\| \tau_+^{|\alpha_1|+r+3} |v|^{|\alpha_2|-n}  \nabla^q_{\partial_t} \nabla^{\kappa}_x F \partial_{v} \partial_t^p \partial_x^{\beta_1} \partial_{v}^{\alpha_3} f \right\|_{L^1_{v,x}} \leq \left\| \tau_+^{|\kappa|+q+1}  \nabla^q_{\partial_t} \nabla^{\kappa}_x F  \right\|_{L^2_{x}} \left\| \tau_+^{|\beta_1|+p+2} |v|^{|\alpha_3|}  \partial_t^p \partial_x^{\beta_1} \partial_v \partial_{v}^{\alpha_3} f \right\|_{L^1_vL^2_{x}}.$$
The left hand side of the previous inequality can be bounded by $\epsilon^{\frac{3}{2}}$. Indeed, as $|v| \geq 3$ on the support of $f^0$ and using a $L^2_x-L^1_x$ Sobolev inequality\footnote{To deal with the lack of regularity of the absolute value of, say, $\int_v|g| dv$ one may apply first a $L^2-L^1$ Sobolev inequality in the variables $(x^1,x^2)$ and then in the variable $x^3$, as we did in the proof of Proposition \ref{KS2}.}, it yields
 $$ \left\| \tau_+^{|\beta_1|+p+2} |v|^{|\alpha_3|}  \partial_t^p \partial_x^{\beta_1} \partial_v \partial_{v}^{\alpha_3} f \right\|_{L^1_vL^2(\Sigma_0)} \lesssim \sum_{ |\beta| \leq |\beta_1|+2} \left\| \tau_+^{|\beta_1|+p+2} |v|^{|\alpha_3|+1}  \partial_t^p \partial_x^{\beta} \partial_v \partial_{v}^{\alpha_3} f \right\|_{L^1_v L^1 (\Sigma_0) }.$$
It remains to use the induction hypothesis twice. This concludes the induction and then the proof.
\end{proof}

\section{All derivatives of $F$ are chargeless}\label{AppendixC}

The aim of this section is to prove the following result, which also applies to massive particles.
\begin{Pro}
Let $N_0 \geq 2$ and $(f,F)$ be a sufficiently regular solution to the Vlasov-Maxwell system on $[0,T]$ such that
$$ \forall \hspace{0.5mm} t \in [0,T], \hspace{1cm} \sum_{ |\beta| \leq N_0} \int_{\Sigma_t} \int_{v \in \R^3} \left| \widehat{Z}^{\beta} f \right| dv dx +\sum_{  |\gamma| \leq N_0} \int_{\Sigma_t}  \left| \mathcal{L}_{Z^{\gamma}}(F) \right|^2 dx < + \infty.$$ 
Then, for all $1 \leq |\gamma| \leq N_0$, $\mathcal{L}_{Z^{\gamma}}(F)$ is chargeless, i.e.
$$  -\lim_{r \rightarrow + \infty} \int_{\mathbb{S}_{0,r}} \rho \left( \mathcal{L}_{Z^{\gamma}}(F) \right) d \mathbb{S}_{0,r} = 0 .$$
\end{Pro}
This ensues from the commutation formula of Proposition \ref{Maxcom} and the following lemma.
\begin{Lem}\label{lemappen}
Fix $p \geq 1$ and let $H_0$, $H_1$, ..., $H_p$ be sufficiently regular $2$-forms defined on $[0,T] \times \R^3$ and $h_0$, $h_1$, ..., $h_p$ be sufficiently regular functions defined on $[0,T] \times \R^3_x \times \R^3_v$ such that
$$ \nabla^{\mu} \left( H_0 \right)_{ \mu \nu} = J(h_0)_{\nu} \hspace{12mm} \text{and} \hspace{12mm} T(h_0)=\sum_{1 \leq i \leq p} H_i \left( v, \nabla_v h_i \right).$$
Suppose moreover that
$$\forall \hspace{0.5mm} t \in [0,T], \hspace{1cm} \sum_{\lambda=0}^p \|H_{\lambda} \|_{L^2_x} \hspace{-0.4mm} (t)+\| h_{\lambda} \|_{L^1_{x,v}} \hspace{-0.5mm} (t) + \|(1+r)  \nabla_{t,x} h_{\lambda} \|_{L^1_{x,v}} \hspace{-0.5mm} (t)+ \|(1+|v|)  \nabla_{v} h_{\lambda} \|_{L^1_{x,v}}\hspace{-0.5mm} (t) < + \infty. $$
Then, $\mathcal{L}_{Z}(H_0)$ is chargeless for all $Z \in \mathbb{K}$.
\end{Lem}
\begin{Rq}
Note that in dimension $n \neq 3$, we merely have that $\mathcal{L}_S(H_0)$ is chargeless if and only if $H_0$ is.
\end{Rq}
\begin{proof}
To lighten the proof, we suppose that $p=1$ and we denote $(H_0,h_0)$ by $(G,g)$ and $(h_1,H_1)$ by $(H,h)$. Consider first $Z \in \mathbb{P}$. Then, by Lemma \ref{comumax1}, $\nabla^{\mu} \mathcal{L}_Z(G)_{\mu 0}= J(\widehat{Z} g)_0$ so that, using the divergence theorem, 
\begin{equation}\label{eq:chargeap}
 Q(t):= \lim_{r \rightarrow + \infty} \int_{ \Sp_{t,r} } \rho \left( \mathcal{L}_{Z}(G) \right) d \Sp_{t,r} = - \int_{x \in \R^3} \int_{v \in \R^3} \widehat{Z} g dv dx.
 \end{equation}
\begin{itemize}
\item If $Z=\Omega \in \Or$ is a rotational vector field, the result does not depend of the source term of the Maxwell equations. Indeed, using Lemma \ref{randrotcom} and the divergence theorem on $\mathbb{S}_{t,r}$, one has,
$$ \int_{ \Sp_{t,r} } \rho \left( \mathcal{L}_{\Omega}(G) \right) d \Sp_{t,r} = \int_{ \Sp_{t,r} } \Omega \left( \rho ( G ) \right) d \Sp_{t,r}= \int_{ \Sp_{t,r} } \slashed{div} ( \Omega) \rho(G) d \Sp_{t,r} .$$
As $\Omega$ is tangential to $\Sp_{t,r}$, we have $\slashed{div} (\Omega )(t,r,\omega_1,\omega_2) = div_{\R^3} (\Omega)(t,r,\omega_1,\omega_2)=0$, so that $Q(t)=0$.
\item If $Z = \partial_i$ is a spatial translation, an integration by parts on the right hand side of \eqref{eq:chargeap} gives the result. 
\item If $Z=\partial_t$, then, as $\partial_t g=-\frac{v^j}{v^0} \partial_j g+H \left( \frac{v}{v^0} , \nabla_v h \right)$, integration by parts (in $x$ and in $v$) gives us, as $H$ is a $2$-form,
\begin{equation}\label{eq:chargeapped}
Q(t)= \int_{x \in \R^3} \int_{v \in \R^3} \frac{v^i}{v^0} \partial_i g-H \left( \frac{v}{v^0} , \nabla_v h \right) dv dx = \int_{x \in \R^3} \int_{v \in \R^3} \frac{v^i v^j}{(v^0)^3} H_{ij} h dv dx =0.
\end{equation}
\item If $Z= \Omega_{0i}$ is a Lorentz boost, then an integration by parts in $x$ on $t \partial_i g$ and in $v$ on $v^0 \partial_{v^i} g$ gives
$$Q(t) = - \int_{x \in \R^3} \int_{v \in \R^3} \left( t \partial_i g+x^i \partial_t g+v^0 \partial_{v^i} g \right) dv dx = \int_{x \in \R^3} \int_{v \in \R^3}\frac{v^i}{v^0} gdv dx-\int_{x \in \R^3} \int_{v \in \R^3} x^i \partial_t g dv dx.$$
Using again $\partial_t g=-\frac{v^j}{v^0} \partial_j g+H \left( \frac{v}{v^0} , \nabla_v h \right)$ and integrating by parts, we have
\begin{eqnarray}
\nonumber \int_{x \in \R^3} \int_{v \in \R^3} x^i \partial_t g dv dx & = & -\int_{v \in \R^3} \frac{v^j}{v^0} \int_{x \in \R^3}x^i \partial_j g dx dv+ \int_{x \in \R^3} x^i \int_{v \in \R^3} H \left( \frac{v}{v^0} , \nabla_v h \right) dv dx \\ \nonumber 
& = &  \int_{x \in \R^3} \int_{v \in \R^3} \frac{v^i}{v^0}g dv dx-\int_{x \in \R^3} \int_{v \in \R^3} x^i\frac{v^k v^j}{(v^0)^3} H_{kj} h dv dx .
\end{eqnarray}
As $H$ is a $2$-form, we finally obtain that $Q(t)=0$.
\end{itemize}
For the case of the scaling vector field, note first by Lemma \ref{comumax1} that $\nabla^{\mu} \mathcal{L}_S (G)_{\mu 0 } = J(Sg)_0+3J(g)_0$. Hence,
$$ \lim_{r \rightarrow + \infty} \int_{ \Sp_{t,r} } \rho \left( \mathcal{L}_{S}(G) \right) d \Sp_{t,r} = - \int_{x \in \R^3} \int_{v \in \R^3} \left( x^{\mu} \partial_{\mu}g+3g \right) dv dx = -t \int_{x \in \R^3} \int_{v \in \R^3}  \partial_t g dv dx .$$
Recall from \eqref{eq:chargeapped} that the integral on the right hand side of the last equation is equal to $0$. This concludes the proof.
\end{proof}

\section{Proof of Lemmas \ref{randrotcom}, \ref{null} and \ref{alphaem}}\label{appendixD}

Let $G$ be a $2$-form and $J$ be a $1$-form, both sufficiently regular and defined on $[0,T[ \times \R^3$, such that
\begin{eqnarray}
\nonumber \nabla^{\mu} G_{\mu \nu} & = & J_{\nu}, \\ \nonumber
\nabla^{\mu} {}^* \! G_{ \mu \nu } & = & 0.
\end{eqnarray}
Let us successively prove Lemmas \ref{randrotcom}, \ref{null} and \ref{alphaem}.
\begin{Lem}
Let $\Omega \in \Or$. Then, denoting by $\zeta$ any of the null component $\alpha$, $\underline{\alpha}$, $\rho$ or $\sigma$,
$$ [\mathcal{L}_{\Omega}, \nabla_{\partial_r}] G=0, \hspace{1.2cm} \mathcal{L}_{\Omega}(\zeta(G))= \zeta ( \mathcal{L}_{\Omega}(G) ) \hspace{1.2cm} \text{and} \hspace{1.2cm} \nabla_{\partial_r}(\zeta(G))= \zeta ( \nabla_{\partial_r}(G) ) .$$
Similar results hold for $\mathcal{L}_{\Omega}$ and $\nabla_{\partial_t}$, $\nabla_L$ or $\nabla_{\underline{L}}$. For instance, $\nabla_{L}(\zeta(G))= \zeta ( \nabla_{L}(G) )$.
\end{Lem}
\begin{proof}
Let $\Omega \in \Or$. The property $[\mathcal{L}_{\Omega}, \nabla_{\partial_r}] G=0$ follows from $[\Omega, \partial_r]=0$, straightforward computations and that, in cartesian coordinates and for a vector field $X$,
\begin{equation}\label{Liecoordinate}
 \mathcal{L}_{X} (G)_{\mu \nu} = X(G_{\mu \nu} )+ \partial_{\mu} ( X^{\lambda} ) G_{\lambda \nu}+\partial_{\nu} (X^{\lambda} ) G_{\mu \lambda} \hspace{6mm} \text{and} \hspace{6mm} \nabla_X(G)_{\mu \nu} = X(G_{\mu \nu} ).
 \end{equation}
Aside from $\mathcal{L}_{\Omega} (\sigma(G))= \sigma ( \mathcal{L}_{\Omega}(G))$, the remaining identities ensue from $[\Omega, L]=[\Omega, \underline{L}]=\nabla_{\partial_r} L=\nabla_{\partial_r} \underline{L}=\nabla_{\partial_r} e_A=0$ since, for instance,
$$ 2 \nabla_{\partial_r} \rho(G) = (\nabla_{\partial_r} G) (L, \underline{L} )+G(\nabla_{\partial_r} L, \underline{L})+G(L, \nabla_{\partial_r} \underline{L}) \hspace{8mm} \text{and} \hspace{8mm} \mathcal{L}_{\Omega}(\alpha)=\mathcal{L}_{\Omega}(G)( \cdot , L)+G(\cdot, [\Omega,L]).$$ 
Using that $[\Omega, e_A]=C_{\Omega}^B(\omega_1,\omega_2) e_B$, where $C^B_{\Omega}$ are bounded functions on the sphere, we directly obtain $|\Omega \sigma(G) | \lesssim |\sigma ( \mathcal{L}_{\Omega}(G)) |+|\sigma(G)|$, which is good enough for proving the following results of this appendix. To obtain $\Omega(\sigma(G))= \sigma ( \mathcal{L}_{\Omega}(G))$, one can check, with straightforward computations that, for a $2$-form $H$,
$$ \rho ({}^* \! H ) = - \sigma (H) \hspace{1cm} \text{and} \hspace{1cm} {}^* \! \mathcal{L}_{\Omega} (H)=  \mathcal{L}_{\Omega} ( {}^* \! H).$$
It then follows
$$\Omega \sigma(G) =-\Omega \rho ( {}^* \! G ) = -\rho (\mathcal{L}_{\Omega} ( {}^* \! H))=-\rho ({}^* \! \mathcal{L}_{\Omega} ( G))=\sigma (\mathcal{L}_{\Omega} ( G))).$$
For the results concerning the operator $\nabla_{\partial_t}$, use $\nabla_{\partial_t} L=\nabla_{\partial_t} \underline{L}=\nabla_{\partial_t} e_A=0$. Finally, for $\nabla_L$ and $\nabla_{\underline{L}}$, recall that $L=\partial_t+\partial_r$ and $\underline{L}=\partial_t-\partial_r$.
\end{proof}
\begin{Lem}
Denoting by $\zeta$ any of the null component $\alpha$, $\underline{\alpha}$, $\rho$ or $\sigma$, we have
$$\tau_- \left| \nabla_{\underline{L}} \zeta (G)\right|+\tau_+ \left| \nabla_{L} \zeta (G) \right| \lesssim \sum_{|\gamma| \leq 1} \left|  \zeta \left( \mathcal{L}_{Z^{\gamma}}(G) \right) \right| \hspace{5mm} \text{and} \hspace{5mm} (1+r)\left| \slashed{\nabla} \zeta (G) \right| \lesssim |\zeta(G)|+\sum_{ \Omega \in \Or} \left|  \zeta \left( \mathcal{L}_{\Omega}(G) \right) \right| .$$
\end{Lem}
\begin{proof}
Let $\zeta \in \{ \alpha, \underline{\alpha}, \rho, \sigma \}$. As $\nabla_{\underline{L}}$ commute with the null decomposition (see the previous Lemma) and $(t-r)\underline{L}= S-\frac{x^i}{r} \Omega_{0i}$ (see Lemma \ref{goodderiv}), we have,
$$ (t-r) \nabla_{\underline{L}} \zeta \left( G \right) = \zeta \left( \nabla_{(t-r) \underline{L}} G \right) = \zeta \left( \nabla_S G  -\frac{x^i}{r} \nabla_{\Omega_{0i}} G \right)
=  \zeta \left(\mathcal{L}_S (G) \right)-2 \zeta (G)-\frac{x^i}{r} \zeta \left( \mathcal{L}_{\Omega_{0i}}(G) \right)+\frac{x_i}{r} \zeta \left( H^i \right), $$
since $\mathcal{L}_S(G)=\nabla_S(G)+2G$ and where $H^i=\mathcal{L}_{\Omega_{0i}}(G)-\nabla_{\Omega_{0i}}(G)$. We have, using \eqref{Liecoordinate},
\begin{equation}\label{defHi}
\hspace{-2mm} H^i_{\mu \nu} =0 \hspace{5mm} \text{if} \hspace{5mm} \mu, \nu \notin \{0,i \} \hspace{3mm} \text{or} \hspace{3mm} \mu, \nu \in \{ 0, i \}. \hspace{8mm} \text{If} \hspace{5mm} \nu \notin \{0,i \}, \hspace{5mm} H^i_{0\nu}=G_{i \nu} \hspace{5mm} \text{and} \hspace{5mm} H^i_{i \nu} = G_{0 \nu}.
 \end{equation}
Consequently,
$$ \frac{x_i}{r} \rho \left( H^i \right) = \frac{x_i}{r} H^i_{r 0 } = \frac{x^i}{r}\frac{x^j}{r} G_{j i} =0, \hspace{5mm} \text{so that} \hspace{5mm} | \underline{L} \rho( G)| \lesssim \frac{1}{\tau_-} \sum_{|\gamma| \leq 1} |\rho(\mathcal{L}_{Z^{\gamma}}(G)|.$$
It remains to study $x_i \zeta \left( H^i \right)$ for $\zeta \in \{ \alpha, \underline{\alpha}, \sigma \}.$ Let us treat together the case of $\alpha$ and $\underline{\alpha}$ by computing $x_iH^i_{A0}$ and $x_iH^i_{Ar}$. Recall that $e_A = \sum_{1 \leq k < l \leq 3} C^{k,l}(\omega_1,\omega_2) \frac{\Omega_{kl}}{r}$ where $C^{k,l}$ are bounded functions on the sphere. As
\begin{eqnarray}
\nonumber x_i H^i \hspace{-1mm} \left( \frac{\Omega_{kl}}{r} , \partial_t \hspace{-1mm} \right) \hspace{-3mm} & = & \hspace{-3mm} \frac{x_ix^k}{r} H^i_{l0}-\frac{x_ix^l}{r} H^i_{k0}=\frac{x^ix^k}{r} G_{li}-\frac{x^ix^l}{r} G_{ki}=rG \left( \frac{\Omega_{kl}}{r}, \partial_r \right), \\
\nonumber x_i H^i \hspace{-1mm} \left( \frac{\Omega_{kl}}{r} , \partial_r \hspace{-1mm} \right) \hspace{-3mm} & = & \hspace{-3mm} \frac{x_ix^kx^j}{r^2} H^i_{lj}-\frac{x_ix^lx^j}{r^2} H^i_{kj} \\ \nonumber & = & \hspace{-3mm} x^k \hspace{-0.3mm} \left( \hspace{-0.2mm} \frac{x_ix^i \hspace{-0.4mm}- \hspace{-0.mm}(x^l)^2}{r^2} G_{l0} \hspace{-0.3mm} + \hspace{-0.3mm} \frac{x_lx^j}{r^2} G_{0j} \hspace{-0.3mm} - \hspace{-0.3mm} \frac{(x^l)^2}{r^2} G_{0l} \hspace{-0.2mm} \right) \hspace{-0.2mm} - \hspace{-0.1mm} x^l \hspace{-0.3mm} \left( \hspace{-0.2mm} \frac{x_ix^i \hspace{-0.4mm}- \hspace{-0.3mm} (x^k)^2}{r^2} G_{k0} \hspace{-0.3mm} + \hspace{-0.3mm} \frac{x_kx^j}{r^2} G_{0j} \hspace{-0.3mm} - \hspace{-0.3mm} \frac{(x^k)^2}{r^2} G_{0k} \hspace{-0.2mm} \right) \\ \nonumber
& = & \hspace{-3mm} r G \left( \frac{\Omega_{kl}}{r} , \partial_t \right),
\end{eqnarray}
we obtain
$$| \nabla_{\underline{L}} \left( \alpha( G) \right)_A| \lesssim \frac{1}{\tau_-} \sum_{|\gamma| \leq 1} |\alpha(\mathcal{L}_{Z^{\gamma}}(G)_A| \hspace{6mm} \text{and} \hspace{6mm} |\nabla_{\underline{L}} \left( \underline{\alpha}( G) \right)_A| \lesssim \frac{1}{\tau_-} \sum_{|\gamma| \leq 1} |\underline{\alpha}(\mathcal{L}_{Z^{\gamma}}(G)_A|.$$
For the remaining case, $\zeta=\sigma$, straightforward computations give
$$ x_i H^i \left( \frac{\Omega_{kl}}{r}, \frac{\Omega_{pq}}{r} \right) = \frac{x^k x^p x_i}{r^2} H^i_{lq}+\frac{x^l x^q x_i}{r^2} H^i_{kp}-\frac{x^k x^q x_i}{r^2} H^i_{lp}-\frac{x^l x^p x_i}{r^2} H^i_{kq} =0, \hspace{5mm} \text{so that} \hspace{5mm} x_i H^i_{AB}=0.$$
The proof for $\nabla_{L}$ is similar as it also commutes with the null decomposition and since $(t+r)L=S+\frac{x^i}{r}\Omega_{0i}$. Finally, for the angular derivatives, use that $\mathcal{L}_{\Omega}$ commute with the null decomposition and that for a function $u$ and a $1$-form $U$ tangential to the $2$-spheres,
$$ r \left| \slashed{\nabla} u \right| \leq \sum_{\Omega \in \Or} \left| \Omega u \right| \hspace{10mm} \text{and} \hspace{10mm}  r \left| \slashed{\nabla} U \right|  \lesssim |U|+ \sum_{\Omega \in \Or} \left| \mathcal{L}_{\Omega} U \right|. $$
In order to prove the last inequality, recall that $|\Gamma_{Ar}^B| \lesssim r^{-1}$, where $\Gamma$ are the christofel symbols of the Minkowski spacetime in the nonholonomic basis $(\partial_t, \partial_r,e_1,e_2)$, so that
$$ r \left| \slashed{\nabla}_{e_A} U \right|  \leq r\left| \nabla_{e_A} U -\Gamma^D_{Ar}U_Ddr\right| \lesssim |U|+\sum_{\Omega \in \Or} \left| \nabla_{\Omega} U \right|+\lesssim |U|+\sum_{\Omega \in \Or} \left| \mathcal{L}_{\Omega}( U) \right|.  $$
\end{proof}
\begin{Lem}
Denoting by $(\alpha, \underline{\alpha}, \rho, \sigma)$ the null decomposition of $G$, we have
$$\nabla_{\underline{L}} \alpha_A-\frac{\alpha_A}{r}+\slashed{\nabla}_{e_A} \rho+\varepsilon_{BA} \slashed{\nabla}_{e_B} \sigma=J_A .$$
\end{Lem}
\begin{proof}
Let us start by proving, for $B \neq A$,
\begin{equation}\label{eq:covsphere}
(\nabla_{e_B} G)(e_B,e_A)=\varepsilon_{BA} \slashed{\nabla}_{e_B} \sigma -\frac{1}{2r}\alpha(e_A)+\frac{1}{2r} \underline{\alpha}(e_A).
\end{equation}
Suppose for instance that $e_B=e_1$ and $e_A=e_2$. Then, as $\nabla_{e_C} e_D = \slashed{\nabla}_{e_C} e_D-\frac{\delta_{C,D}}{r} \nabla_{\partial_r}$,
\begin{eqnarray}
\nonumber \varepsilon_{BA} \slashed{\nabla}_{e_B} \sigma \hspace{1mm} = \hspace{1mm} e_1 \left( G(e_1,e_2) \right) & = &  (\nabla_{e_1} G )(e_1,e_2)+G(\nabla_{e_1} e_1, e_2)+G( e_1, \nabla_{e_1} e_2) \\ \nonumber
& = & (\nabla_{e_1} G )(e_1,e_2)+\frac{1}{2r}\alpha(e_2)-\frac{1}{2r} \underline{\alpha}(e_2)+G(\slashed{\nabla}_{e_1} e_1, e_2)+G( e_1, \slashed{\nabla}_{e_1} e_2).
\end{eqnarray}
It remains to notice that, as $G$ is a $2$-form and $ 2 \left< \slashed{\nabla}_{e_1}e_A,e_A \right> =\slashed{\nabla}_{e_1} \left( \left< e_A,e_A \right> \right)=0$,
$$G(\slashed{\nabla}_{e_1} e_1, e_2)+G( e_1, \slashed{\nabla}_{e_1} e_2) = G( \left< \slashed{\nabla}_{e_1} e_1,e_1 \right> e_1, e_2)+G( e_1, \left<\slashed{\nabla}_{e_1} e_2,e_2 \right> e_2) =0.$$
Recall now from Lemma \ref{maxwellbis} that $\nabla_{[\lambda} G_{\mu \nu ] }=0$. Taking the $(L,\underline{L},A)$ component of this tensorial equation, we get, as $\nabla_L \underline{L}=\nabla_{\underline{L}} L=0$ and $\nabla_{e_A} L=-\nabla_{e_A} \underline{L} = \frac{1}{r}e_A$,
\begin{eqnarray}
\nonumber \nabla_{[L} G_{\underline{L} A ] } =0 & \Leftrightarrow & (\nabla_L G)(\underline{L} , e_A)+(\nabla_{\underline{L}} G)(e_A,L)+(\nabla_{e_A} G)(L,\underline{L}) =0 \\ \nonumber
& \Leftrightarrow & -(\nabla_L \underline{\alpha})( e_A)+(\nabla_{\underline{L}} \alpha)(e_A)+2 \nabla_{e_A} \rho-G( \nabla_{e_A} L, \underline{L} )-G(L, \nabla_{e_A} \underline{L} ) =0 \\ 
& \Leftrightarrow & -(\nabla_L \underline{\alpha})( e_A)+(\nabla_{\underline{L}} \alpha)(e_A)+2 \slashed{\nabla}_{e_A} \rho-\frac{1}{r}\underline{\alpha}( e_A )-\frac{1}{r} \alpha (e_A) =0. \label{eq:alpha1}
\end{eqnarray}
Similarly, taking $\nu =A$ in $\nabla^{\mu} G_{\mu \nu}= J_{\nu}$, we obtain, using \eqref{eq:covsphere} and since $\nabla^L = -\frac{1}{2} \nabla_{\underline{L}}$ and $\nabla^{\underline{L}} = -\frac{1}{2} \nabla_L$,
\begin{eqnarray}
\nonumber \nabla^{\mu} G_{\mu A} = J_A & \Leftrightarrow & -\frac{1}{2}(\nabla_L G)(\underline{L} , e_A)-\frac{1}{2}(\nabla_{\underline{L}} G)(L,e_A)+(\nabla_{e_B} G)(e_B,e_A) =J_A \\ 
& \Leftrightarrow & \frac{1}{2}(\nabla_L \underline{\alpha})( e_A)+\frac{1}{2}(\nabla_{\underline{L}} \alpha)(e_A)+ \varepsilon_{BA} \slashed{\nabla}_{e_B} \sigma -\frac{1}{2r}\alpha(e_A)+\frac{1}{2r} \underline{\alpha}(e_A)=J_A.  \label{eq:alpha2}
\end{eqnarray}
It remains to add half of \eqref{eq:alpha1} to \eqref{eq:alpha2}.
\end{proof}

\renewcommand{\refname}{References}
\bibliographystyle{abbrv}
\bibliography{biblio}

\end{document}